\documentclass[11pt,reqno]{amsart}
\usepackage{amsthm,amssymb,amsmath}
\usepackage{graphicx}
\usepackage[foot]{amsaddr}

\newtheorem{theorem}{Theorem}
\newtheorem*{theorem*}{Theorem}
\newtheorem{lemma}[theorem]{Lemma}
\newtheorem{corollary}{Corollary}
\newtheorem{proposition}{Proposition}
\newtheorem{claim}{\it Claim}

\theoremstyle{definition}

\newtheorem*{definition*}{\sc Definition}

\newtheorem{remark}{Remark}
\newtheorem*{remark*}{Remark}

\newtheorem*{example*}{\bf Example}

\newcommand{\loc}{{\rm loc}}
\newcommand{\supp}{{\rm sprt\,}}

\newcommand{\Real}{{\rm Re}}
\newcommand{\Imag}{{\rm Im}}

\newcommand{\sprt}{{\rm sprt\,}}
\newcommand{\clos}{{\rm clos}}
\newcommand{\sgn}{{\rm sgn\,}}

\makeatletter
\newcommand{\dotminus}{\mathbin{\text{\@dotminus}}}

\newcommand{\@dotminus}{%
  \ooalign{\hidewidth\raise1ex\hbox{.}\hidewidth\cr$\m@th-$\cr}%
}
\makeatother

\makeatletter
\@namedef{subjclassname@2020}{%
  \textup{2020} Mathematics Subject Classification}
\makeatother

\begin{document}

\title[Fractional Kolmogorov operator and desingularizing weights]{Fractional Kolmogorov operator and desingularizing weights}


\author{D.\,Kinzebulatov}

\address{Universit\'{e} Laval, D\'{e}partement de math\'{e}matiques et de statistique, 1045 av.\,de la M\'{e}decine, Qu\'{e}bec, QC, G1V 0A6, Canada}

\email{damir.kinzebulatov@mat.ulaval.ca}

\author{Yu.\,A.\,Sem\"{e}nov}

\address{University of Toronto, Department of Mathematics, 40 St.\,George Str, Toronto, ON, M5S 2E4, Canada}

\email{semenov.yu.a@gmail.com}


\thanks{The research of D.K. is supported by grants from NSERC  and FRQNT}

\begin{abstract}
We establish sharp upper and lower bounds on the heat kernel of the fractional Laplace operator perturbed by Hardy-type drift by transferring it to appropriate weighted space with singular weight.
\end{abstract}

\maketitle

\section{Introduction}

The fractional Kolmogorov operator $(-\Delta)^{\frac{\alpha}{2}} + \mathsf{f}\cdot \nabla$, $1<\alpha<2$ with a (locally unbounded) vector field $\mathsf{f}:\mathbb R^d \rightarrow \mathbb R^d$, $d\geq 3$, plays important role in probability theory where it arises as the generator of symmetric $\alpha$-stable process with a drift (in contrast to diffusion processes, $\alpha$-stable process has long range interactions). It has been the subject of intensive study over the past two decades. There is now a well developed theory of this operator with $\mathsf{f}$ belonging to the corresponding Kato class.
This class, in particular, contains the vector fields $\mathsf{f}$ with $|\mathsf{f}| \in L^p$, $p>\frac{d}{\alpha-1}$ and is, indeed, responsible for existence of the standard (local in time) two-sided bound on the heat kernel $e^{-t\Lambda}(x,y)$, $\Lambda\supset(-\Delta)^{\frac{\alpha}{2}} + \mathsf{f}\cdot \nabla$, in terms of $e^{-t(-\Delta)^{\frac{\alpha}{2}}}(x,y)$, see \cite{BJ}. 

The authors in \cite{KSS} studied the fractional Kolmogorov operator $$\Lambda=(-\Delta)^{\frac{\alpha}{2}} + b \cdot \nabla, \quad b(x)=\kappa|x|^{-\alpha}x, \quad 0<\kappa<\kappa_0,$$ where $\kappa_0$ is the borderline constant for existence of $e^{-t\Lambda}(x,y)\geq 0$. The model vector field $b$ lies outside of the scope of the Kato class, and exhibits critical behaviour both at $x=0$ and at infinity making the standard upper bound on $e^{-t\Lambda}(x,y)$ 
in terms of $e^{-t(-\Delta)^{\frac{\alpha}{2}}}(x,y)$ invalid.  Instead, 
the two-sided bounds $e^{-t\Lambda}(x,y) \approx e^{-t(-\Delta)^{\frac{\alpha}{2}}}(x,y)\varphi_t(y)$ $(y\neq 0)$ hold for an appropriate  weight $\varphi_t \geq \frac{1}{2}$ unbounded at $y=0$ \cite[Theorem 3]{KSS}.

The present paper continues \cite{KSS}. We study the heat kernel $e^{-t\Lambda}(x,y)$ of the fractional Kolmogorov operator with the drift of opposite sign (``repulsion case'')
\begin{equation}
\label{b} 
\begin{array}{c}
\Lambda=(-\Delta)^{\frac{\alpha}{2}} - b \cdot \nabla, \\[2mm]
b(x)=\kappa|x|^{-\alpha}x, \quad 0<\kappa<\infty.
\end{array}
\end{equation}
Although the standard (global) upper bound in terms of $e^{-t(-\Delta)^{\frac{\alpha}{2}}}(x,y)$ holds true for $e^{-t\Lambda}(x,y)$  (Theorem \ref{nash_west} below), the singularity of $b$ at $x=0$ makes it off the mark. 
Namely, in Theorem \ref{thm_est2} and Theorem \ref{thm_lb} below we establish sharp upper and lower bounds
\begin{equation*}
\label{ub}
\tag{$ULB_w$}
e^{-t\Lambda}(x,y) \approx e^{-t(-\Delta)^{\frac{\alpha}{2}}}(x,y)\psi_t(y), \quad x,y \in \mathbb R^d,\quad t>0,
\end{equation*}
where the continuous weight $0 \leq \psi_t(y) \leq 2$ vanishes at $y=0$ as $|y|^\beta$, $\beta>0$ (Theorem \ref{nash_est}). (Here notation $a(z) \approx b(z)$ means that $c^{-1}b(z)\leq a(z) \leq cb(z)$ for some constant $c>1$ and all admissible $z$.) 
The order of vanishing $\beta~(<\alpha)$ depends explicitly on the value of the multiple $\kappa>0$ and tends to $\alpha$ as $\kappa\uparrow\infty$.

The key step in proving the upper and lower bound \eqref{ub} is the weighted Nash initial estimate
\begin{equation}
\label{ab}
\tag{$NIE_w$}
0 \leq e^{-t\Lambda}(x,y)\leq Ct^{-\frac{d}{\alpha}}\psi_t(y), \quad x,y \in \mathbb R^d, \quad t>0.
\end{equation}
The proof of \eqref{ab} uses the method of desingularizing weights \cite{MS0,MS1,MS} based on ideas set forth by J.\,Nash \cite{N}: it depends on the ``desingularizing'' ($L^1,L^1$) bound on the weighted semigroup $\psi_t e^{-t\Lambda} \psi_t^{-1}$. 
%The proof of \eqref{ab} uses a modification of the method of \cite{KSS}. 
%We will address the matter of $\psi_t$-weighted lower bound in a forthcoming paper.

The operator \eqref{b} in the local case $\alpha=2$ has been studied in \cite{MSS, MSS2} by considering it in the space $L^2(\mathbb R^d,|x|^\gamma dx)$ for appropriate $\gamma$ where the operator becomes symmetric. This approach, however, does not work for $\alpha<2$.

 Recently, the authors in \cite{CKSV}, \cite{JW} considered the fractional Schr\"{o}dinger operator 
$H_+=(-\Delta)^{\frac{\alpha}{2}}+V$, $V(x)= \kappa |x|^{-\alpha}$, $0<\alpha<2$, $\kappa>0$, 
and established, using different methods, sharp two-sided bounds
$$
e^{-tH_+}(x,y) \approx e^{-t(-\Delta)^{\frac{\alpha}{2}}}(x,y)\psi_t(x)\psi_t(y)
$$
for appropriate weights $\psi_t(x)$ vanishing at $x=0$. We apply some ideas from \cite{JW} (in the proof of Theorem \ref{thm_est2}).
 
In contrast to the cited papers, this work deals with purely non-local and non-symmetric situation. This leads to new difficulties, and requires new ideas. 
Even the proof of the standard upper bound $ e^{-t\Lambda}(x,y)\leq Ce^{-t(-\Delta)^{\frac{\alpha}{2}}}(x,y)$ (Theorem \ref{nash_west}), as well as the construction of semigroups $e^{-t\Lambda}$, $e^{-t\Lambda^*}$  (Sections \ref{sect_d} and \ref{sect_d2}) become non-trivial. 
The same applies to the Sobolev regularity of $e^{-t\Lambda}f$, $f \in C_c^\infty$ established in Section \ref{sect_constr_d3}. We consider these results, along with Theorem \ref{thm_est2} and Theorem \ref{thm_lb}, as the main results of this article.

Below we apply the scheme of the proof of the upper and lower bounds in \cite{KSS}, 
although with comprehensive modifications in the method, both at the level of the abstract desingularization theorem (Theorem \ref{thmB}) and in the proofs of \eqref{ab}, \eqref{ub} and of the standard upper bound. 

%We consider the fact that the same general scheme works in the setting of this paper as an advantage of this method.

We note that the heat kernel of the operator $(-\Delta)^{\frac{\alpha}{2}}+\mathsf{f}\cdot \nabla$ with 
${\rm div}\,\mathsf{f}=0$ was studied in \cite{MM,MM2}. 
For properties of the Feller process determined by 
\eqref{b} see \cite{KM}.

Let us mention that the vector field $b(x)=\kappa|x|^{-\alpha}x$ exhibits critical behaviour even if we remove the singularity of $b$ at the origin. Namely, if we consider $\Lambda$ with $b$ bounded in $B(0,1)$ but having slower decay at infinity, $b(x)=\kappa |x|^{-\alpha+\varepsilon}x$, $\varepsilon>0$ for $|x| \geq 1$, then the global in time upper bound $ e^{-t\Lambda}(x,y)\leq Ce^{-t(-\Delta)^{\frac{\alpha}{2}}}(x,y)$ of Theorem \ref{nash_west} would no longer be valid.

\smallskip

\tableofcontents

\section{Desingularization in abstract setting}

We first prove a general desingularization theorem in abstract setting, that we will apply in the next section to the fractional Kolmogorov operator.

Let $X$ be a locally compact topological space, and $\mu$ a $\sigma$-finite Borel measure on $X$. Set $L^p=L^p(X,\mu)$, $p \in [1,\infty]$, a (complex) Banach space. We use the notation
$$
\langle u,v\rangle = \langle u\bar{v}\rangle :=\int_{X}u\bar{v}d\mu,\quad \|\cdot\|_{p \rightarrow q}=\|\cdot\|_{L^p \rightarrow L^q}.
$$
Let $-\Lambda$ be the generator of a contraction $C_0$ semigroup $e^{-t\Lambda}$, $t>0$, in $L^2$.

Assume that, for some constants $M\geq 1$, $c_S>0$, $j>1$, $c$,
\[
\|e^{-t\Lambda}f\|_1\leq M\|f\|_1 , \quad t\geq 0, \quad f\in L^1\cap L^2 .\tag{$B_{11}$}
\]
\[
\text{Sobolev embedding property:} \quad \Real\langle\Lambda u,u\rangle\geq c_S\|u\|^2_{2j}, \quad u\in D(\Lambda).\tag{$B_{12}$}
\]
\[
\|e^{-t\Lambda}\|_{2 \rightarrow \infty} \leq ct^{-\frac{j'}{2}}, \quad t>0, \quad j'=\frac{j}{j-1}.\tag{$B_{13}$}
\]

\smallskip

Assume also that there exists a family of real valued weights $\psi=\{\psi_s\}_{s>0}$ on $X$ such that, for all $s>0$, 
\[
0 \leq \psi_s, \psi_s^{-1} \in L^1_{\loc}(X - N,\mu), \quad \text{where } N \text{ is a closed null set},\tag{$B_{21}$}
\]
and there exist constants $\theta \in ]0,1[$, $\theta \neq \theta (s)$, $c_i\neq c_i(s)$ {\rm($i=2,3$)} and a measurable set $\Omega^s \subset X$ such that
\[
\psi_s(x)^{-\theta } \leq c_2 \text{ for all } x\in X-\Omega^s, \tag{$B_{22}$}
\]
\[
\|\psi_s^{-\theta }\|_{L^{q^\prime}(\Omega^s)}\leq c_3 s^{j^\prime /q^\prime},\text{ where } q^\prime=\frac{2}{1-\theta}.\tag{$B_{23}$}
\]

\begin{theorem}
\label{thmB}
In addition to $(B_{11})-(B_{23})$ assume that there exists a constant $c_1\neq c_1(s)$ such that, for all $\frac{s}{2}\leq t \leq s$,
\[
\|\psi_s e^{-t\Lambda}\psi_s^{-1}f\|_{1} \leq c_1\|f\|_{1}, \quad f \in L^1.
\tag{$B_3$}
\]
Then there is a constant $C$ such that, for all $t>0$ and $\mu$ a.e. $x,y \in X$,
\[
|e^{-t\Lambda}(x,y)|\leq C t^{-j^\prime}\psi_t(y).
\]  
\end{theorem}

\begin{remark}
In application of Theorem \ref{thmB} to concrete operators, the main difficulty is in verification of the assumption ($B_3$).
\end{remark}

%\begin{remark}
%\label{rem0}
%Of course, $(B_{11})+(B_{12})$ implies the bound $\|e^{-t\Lambda}\|_{1\rightarrow 2}\leq \hat{c}t^{-\frac{j^\prime}{2}}$ and hence $(B_1)\equiv (B_{11})+(B_{12})+(B_{13})$ implies the bound $e^{-t\Lambda}(x,y)\leq \tilde{c}t^{-j^\prime}$.
%\end{remark}

\begin{proof}[Proof of Theorem \ref{thmB}]
Set $\psi \equiv \psi_s$ and put $L^2_\psi :=L^2(X,\psi^2 d\mu)$.
Define a unitary map $\Psi: L^2_\psi \to L^2$  by $\Psi f=\psi f$. Set $\Lambda_\psi = \Psi^{-1}\Lambda\Psi$ of domain $D(\Lambda_\psi)=\Psi^{-1}D(\Lambda)$. Then
 $$e^{-t\Lambda_\psi}=\Psi^{-1}e^{-t\Lambda}\Psi, \quad \|e^{-t\Lambda_\psi}\|_{2, \psi \rightarrow 2, \psi} =\|e^{-t\Lambda}\|_{2\rightarrow 2}, \quad t\geq 0.$$
 Here and below the subscript $\psi$ indicates that the corresponding quantities are related to the measure $\psi^2 d \mu$.

Set $u_t=e^{-t\Lambda_{\psi}}f$, $f\in L^2_\psi\cap L^1_\psi$. Applying ($B_{12}$), and then the H\"older inequality, we have
\begin{align*}
-\frac{1}{2}\frac{d}{dt}\langle u_t,u_t \rangle_\psi & = \Real \langle \Lambda_\psi u_t,u_t \rangle_\psi\\
& = \Real\langle \Lambda\psi u_t, \psi u_t \rangle\\
& \geq c_S\|\psi u_t \|_{2j}^2 \\
&\geq c_S \frac{\langle u_t,u_t \rangle_\psi^{r}}{\|\psi u_t \|_q^{2(r-1)}},
\end{align*}
where $q=\frac{2}{1+\theta}(<2)$ and $r=\frac{(1+\theta)j-1}{j\theta }$.

Noticing that {\rm($B_{11}$)} + {\rm($B_{12}$)} implies
the bound $\|e^{-t\Lambda}\|_{1\rightarrow 2}\leq \hat{c}t^{-\frac{j^\prime}{2}}$ (for details, if needed, see Remark \ref{rem0} below), we have by the interpolation inequality 
%(for every $1 \leq q \leq 2$)
$$\|e^{-t\Lambda}\|_{1\to q}\leq c_4 t^{-\frac{j^\prime}{q'}}, \quad q^\prime=\frac{q}{q-1}, \quad c_4=M^{\frac{2}{q}-1}\hat{c}^\frac{2}{q^\prime};$$ also, by {\rm($B_{11}$)} and interpolation, $\|e^{-t\Lambda}\|_{q\to q}\leq M^{\frac{2}{q}-1}$. Therefore,
\begin{align*}
\|\psi u_t\|_q & = \|e^{-t\Lambda}\psi f\|_q = \|e^{-t\Lambda} |\psi|^{-\theta}|\psi|^\frac{2}{q} f\|_q \\
& (\text{we are applying $(B_{22}),(B_{23})$}) \\
&\leq c_2 \|e^{-t\Lambda}\|_{q\to q} \|f\|_{q,\psi} + \|e^{-t\Lambda}\|_{1\to q}\||\psi|^{-\theta}\|_{L^{q^\prime}(\Omega^s)}\|f\|_{q,\psi} \\
&\leq \big(c_2M^{\frac{2}{q}-1} + c_3 c_4 (s/t)^\frac{j^\prime}{q^\prime} \big)\|f\|_{q,\psi}.
\end{align*}

Thus, setting $w=\langle u_t,u_t \rangle_\psi$, we obtain
\[
\frac{d}{dt}w^{1-r}\geq 2(r-1)c_S \big(c_2M^{\frac{2}{q}-1} +c_3 c_4  (s/t)^\frac{j^\prime}{q^\prime} \big)^{-2(r-1)} \|f\|_{q,\psi}^{-2(r-1)}.
\]
Integrating this differential inequality yields 
\[
\|u_t\|_{2,\psi_s}\leq C_1 t^{-j^\prime\big(\frac{1}{q}-\frac{1}{2}\big)}\|f\|_{q,\psi_s}, \quad s/2\leq t \leq s.
\]
The last inequality and $(B_3)$ rewritten in the form $\|u_t\|_{1,\psi}\leq c_1 \|f\|_{1,\psi}$ yield according to the Coulhon-Raynaud Extrapolation Theorem (Theorem \ref{thm_cr} in Appendix \ref{appendix_B})
\[
\|u_t\|_{2,\psi_s}\leq C_2t^{-\frac{j^\prime}{2}} \|f\|_{1,\psi_s}, \quad s/2\leq t \leq s,
\]
or
\begin{equation}
\label{est_12}
\|e^{-t\Lambda}h\|_2\leq C_2t^{-\frac{j^\prime}{2}} \|h\|_{1,\sqrt{\psi_s}},\quad h\in L^2\cap L^1_{\sqrt{\psi_s}}, \quad s/2\leq t \leq s,
\end{equation}
where $L^1_{\sqrt{\psi_s}} :=L^1(X,\psi_s d\mu)$.

Since $\|e^{-2t\Lambda}h\|_\infty\leq\|e^{-t\Lambda}\|_{2\rightarrow \infty}\|e^{-t\Lambda}h\|_2$, we have, employing $(B_{13})$,
\[
\|e^{-2t\Lambda}h\|_\infty\leq cC_2t^{-j^\prime}\|h\|_{1,\sqrt{\psi_s}},
\]
%In turn, applying to  \eqref{est_12} and $\|e^{-t\Lambda} f\|_{\infty} \leq c_0\|f\|_\infty$, $f \in L^1 \cap L^\infty$ the dual variant of the Extrapolation Theorem (Corollary \ref{cor_cr} in Appendix \ref{appendix_B}), we obtain 
%$$
%\|u_t\|_{\infty }\leq C t^{-j^\prime} \|f\|_{1,\psi_s}, \quad s/2\leq t \leq s,
%$$
and so the assertion of Theorem \ref{thmB} follows.
\end{proof}

\begin{remark}
\label{rem0} The standard argument yields: {\rm($B_{11}$)} + {\rm($B_{12}$)} $\Rightarrow$ $\|e^{-t\Lambda}\|_{1\rightarrow 2}\leq \hat{c}t^{-\frac{j^\prime}{2}}$, $t>0$. Indeed, setting 
$ u_t:=e^{-t\Lambda}f$, $f \in L^2 \cap L^1$, we have applying ($B_{12}$), H\"{o}lder's inequality and ($B_{11}$)
\begin{align*}
-\frac{1}{2}\frac{d}{dt}\|u_t\|_2^2 & =  \Real\langle \Lambda u_t,u_t\rangle \\
& \geq c_S\|u_t\|_{2j}^2 \\ 
& \geq  c_S\|u_t\|_2^{2 + \frac{2}{j'}} \|u_t\|_1^{-\frac{2}{j'}} \\
& \geq c_S M^{-\frac{2}{j'}}\|u_t\|_2^{2 + \frac{2}{j'}} \|f\|_1^{-\frac{2}{j'}}.
\end{align*}
Thus, $w := \|u_t\|_2^2$ satisfies 
$
\frac{d}{dt} w^{-\frac{1}{j'}} \geq C \|f\|_{1}^{-\frac{2}{j'}}$, $C = \frac{2 c_S M^{-\frac{2}{j'}}}{j'},
$
so integrating this inequality we obtain $\|e^{-t \Lambda} \|_{1 \rightarrow 2} \leq C^{-\frac{j'}{2}}  t^{-\frac{j'}{2}}$.

It is now seen that $(B_1)\equiv (B_{11})+(B_{12})+(B_{13})$ implies the bound $e^{-t\Lambda}(x,y)\leq \tilde{c}t^{-j^\prime}$.
\end{remark}

\bigskip

\section{Heat kernel $e^{-t\Lambda}(x,y)$  for $\Lambda=(-\Delta)^{\frac{\alpha}{2}} - \kappa|x|^{-\alpha}x \cdot \nabla$, $1<\alpha<2$, $\kappa>0$}

We now state in detail our main result concerning the fractional Kolmogorov operator $(-\Delta)^{\frac{\alpha}{2}} - \kappa|x|^{-\alpha}x \cdot \nabla$, $1<\alpha<2$, $\kappa>0$.

\medskip

\textbf{1.}~Let us outline the construction of an appropriate operator realization $\Lambda_r$ of $(-\Delta)^{\frac{\alpha}{2}} - \kappa|x|^{-\alpha}x \cdot \nabla$ in $L^r$, $1 \leq r < \infty$. Set $$b_\varepsilon(x):=\kappa|x|_\varepsilon^{-\alpha}x, \quad |x|_\varepsilon:=\sqrt{|x|^2+\varepsilon},\;\varepsilon>0,$$ define the approximating operators in $L^r$ $$\Lambda^\varepsilon \equiv \Lambda_r^\varepsilon:=(-\Delta)^{\frac{\alpha}{2}} - b_\varepsilon \cdot \nabla, \quad D(\Lambda_r^\varepsilon)=\mathcal W^{\alpha,r} := \big(1+(-\Delta)^{\frac{\alpha}{2}}\big)^{-1}L^r, \quad 1 \leq r < \infty,$$
and in $C_u$ (the space of uniformly continuous bounded  functions with standard sup-norm),
$$
\Lambda^\varepsilon \equiv \Lambda_{C_u}^\varepsilon:=(-\Delta)^{\frac{\alpha}{2}} - b_\varepsilon \cdot \nabla, \quad D(\Lambda_{C_u}^\varepsilon)=D((-\Delta)^{\frac{\alpha}{2}}_{C_{u}}).
$$
The operator $-\Lambda^\varepsilon$ is the generator of a holomorphic semigroup in $L^r$ and in $C_{u}$. For details, if needed, see Section \ref{sect_d} below.

It is well known that $$e^{-t\Lambda^\varepsilon}L^r_+\subset L^r_+\text{ and }e^{-t\Lambda^\varepsilon}C_u^+\subset C_u^+ $$ where $L^r_+:=\{f\in L^r\mid f\geq 0\}$, $C_u^+:=\{f\in C_u\mid f\geq 0\}$. Also $$\|e^{-t\Lambda^\varepsilon}f\|_\infty \leq \|f\|_\infty, \quad f \in L^r \cap L^\infty,\text{ or } f\in C_u.$$

In Proposition \ref{constr_d} below we show that, for every $r \in [1,\infty[$, the limit
$$
s\mbox{-}L^r\mbox{-}\lim_{\varepsilon \downarrow 0} e^{-t\Lambda_r^{\varepsilon}} \quad (\text{loc.\,uniformly in $t \geq 0$})
$$ 
exists and determines a positivity preserving, contraction $C_0$ semigroup in $L^r$, say $e^{-t\Lambda_r}$; the (minus) generator $\Lambda_r$ is an appropriate operator realization of the fractional Kolmogorov operator $(-\Delta)^{\frac{\alpha}{2}} - \kappa|x|^{-\alpha}x \cdot \nabla$ in $L^r$; there exists a constant $c$ such that
$$
\|e^{-t\Lambda_r}\|_{r \rightarrow q} \leq ct^{-\frac{d}{\alpha}(\frac{1}{r}-\frac{1}{q})}, \quad t>0,
$$
for all $1 \leq r < q \leq \infty$; by construction, the semigroups $e^{-t\Lambda_r}$ are consistent: 
\begin{equation*}
e^{-t\Lambda_r} \upharpoonright L^r \cap L^p = e^{-t\Lambda_p} \upharpoonright L^r \cap L^p.
\end{equation*}
Using Proposition \ref{constr_d}, we obtain
$$\langle \Lambda_ru,h\rangle=\langle u, (-\Delta)^\frac{\alpha}{2}h\rangle + \langle u,b\cdot \nabla h\rangle + \langle u, ({\rm div\,}b)h\rangle, \quad u \in D(\Lambda_r), \quad h \in C_c^\infty
$$
(cf.\,\cite[Prop.\,9]{KSS}).

\medskip

\textbf{2.}~We now introduce the desingularizing weights for $e^{-t\Lambda}$. Define $\beta$ by $$\beta\frac{d+\beta-2}{d+\beta-\alpha} \frac{\gamma(d+\beta-2)}{\gamma(d+\beta-\alpha)}=\kappa,$$
where $$\gamma(\alpha):=\frac{2^\alpha\pi^\frac{d}{2}\Gamma(\frac{\alpha}{2})}{\Gamma(\frac{d}{2}-\frac{\alpha}{2})}.$$ 
Direct calculations show that $\beta \in ]0,\alpha[$ exists (see Figure \ref{fig1}), and that $|x|^{\beta}$ is a Lyapunov's function of the formal adjoint operator $\Lambda^*=(-\Delta)^{\frac{\alpha}{2}} + \nabla \cdot b$, i.e.\,$\Lambda^*|x|^{-\beta}=0$.

\begin{figure}
\begin{center}
\includegraphics[width=0.6\textwidth]{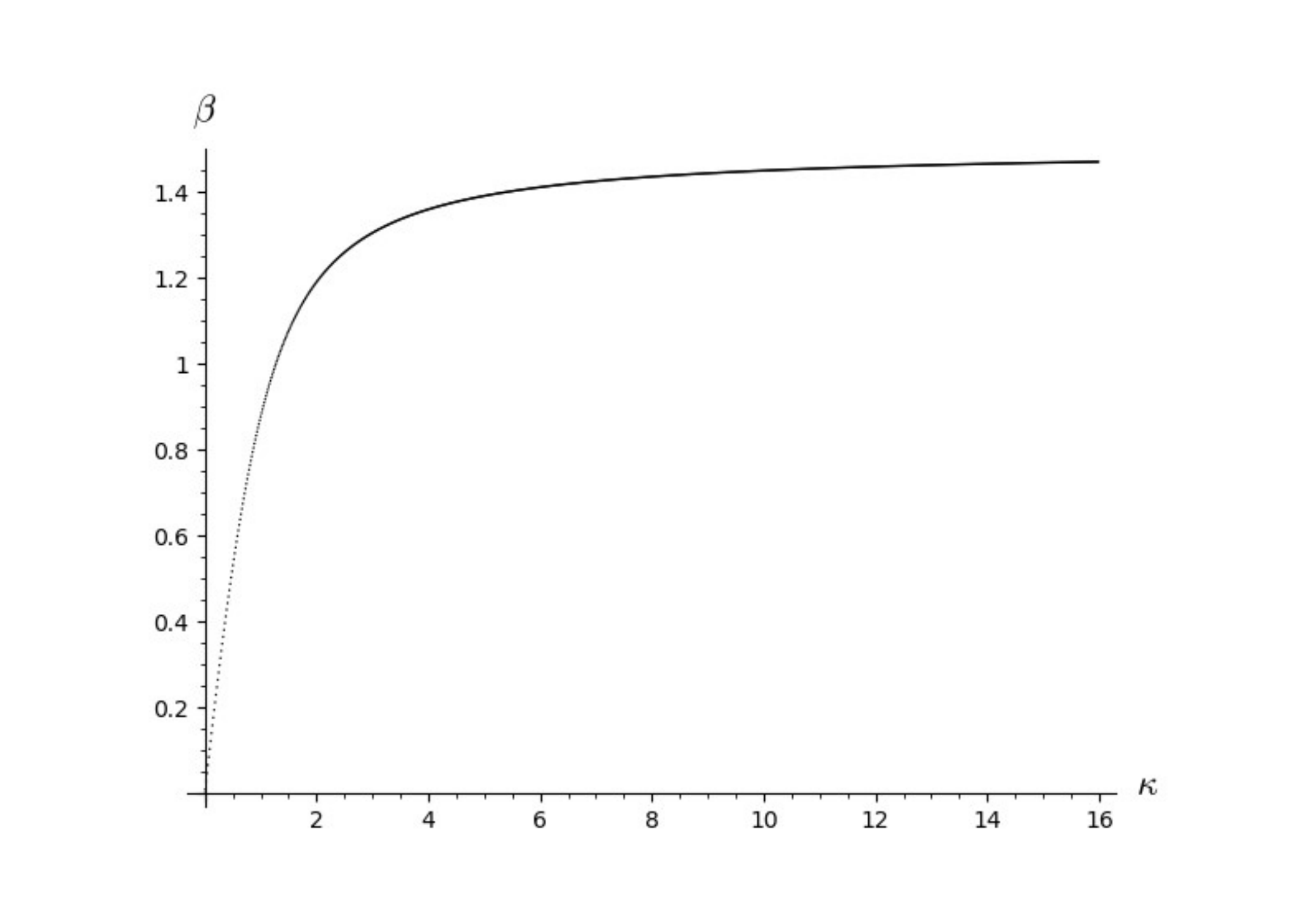}
\end{center}
\vspace*{-5mm}
\caption{The function $\kappa \mapsto \beta$ for $d=3$ and $\alpha=\frac{3}{2}$. \label{fig1}}
\end{figure}

Set $\psi(x)\equiv \psi_s(x):=\eta(s^{-\frac{1}{\alpha}}|x|)$, where $\eta$ is given by
\[
\eta(t)=\left \{
\begin{array}{ll}
t^\beta, & 0<t< 1, \\
\beta t(2-\frac{t}{2})+1-\frac{3}{2}\beta, &1\leq t\leq 2, \\
1+\frac{\beta}{2}, & t\geq 2.
\end{array}
\right. 
\]
%for $c_\beta:=1+\frac{\beta}{2}$.

Applying Theorem \ref{thmB} to the operator $\Lambda_r$ and the weights $\psi_s$, we obtain

\begin{theorem}
\label{nash_est}
$e^{-t\Lambda_r}$ is an integral operator for each $t > 0$ with integral kernel $e^{-t\Lambda}(x,y)\geq 0$. There exists a constant $c_{N,w}$ such that the
weighted Nash initial estimate
\begin{equation}
\label{nie}
\tag{{$NIE_w$}}
e^{-t\Lambda}(x,y)\leq c_{N,w} t^{-\frac{d}{\alpha}}\psi_t(y).
\end{equation}
is valid for all $x,y\in \mathbb R^d$ and $t>0$.
\end{theorem}

The next step is to deduce the following global in time ``standard'' upper bound on $e^{-t\Lambda}(x,y)$.

\begin{theorem}
\label{nash_west}
{\rm(\textit{i})} There is a constant $C_1$ such that, for all $t>0$, $x,y\in \mathbb R^d$,
\begin{equation*}
%\label{nic}
%\tag{{$UB$}}
e^{-t\Lambda}(x,y)\leq C_1e^{-t(-\Delta)^{\frac{\alpha}{2}}}(x,y).
\end{equation*}
{\rm(\textit{ii})} Moreover, for a given $\delta\in ]0,1[$, there is a constant $D=D_\delta>0$ such that 
\begin{equation*}
%\label{nid}
%\tag{{$ iUB$}}
e^{-t\Lambda}(x,y)\leq (1+\delta)e^{-t(-\Delta)^{\frac{\alpha}{2}}}(x,y),\qquad |x|>Dt^\frac{1}{\alpha},\;\;y \in \mathbb R^d.
\end{equation*}
\end{theorem}

Theorem \ref{nash_est} and Theorem \ref{nash_west} are the key tools which allow us to establish the upper bound on $e^{-t\Lambda}(x,y)$:
\begin{theorem}
\label{thm_est2}
There is a constant $C$ such that, for all $t>0$, $x,y \in \mathbb R^d$,
\begin{equation}
\label{nie_}
\tag{{$UB_w$}}
e^{-t\Lambda}(x,y)\leq Ce^{-t(-\Delta)^{\frac{\alpha}{2}}}(x,y)\psi_t(y).
\end{equation}
\end{theorem}

Using Theorem \ref{thm_est2}, we prove the lower bound on $e^{-t\Lambda}(x,y)$:

\begin{theorem}
\label{thm_lb}
There is a constant $\tilde{C}>0$ such that, for all $t>0$, $x,y \in \mathbb R^d$,
\begin{equation}
\label{nie_0}
\tag{{$LB_w$}}
e^{-t\Lambda}(x,y)\geq \tilde{C}e^{-t(-\Delta)^{\frac{\alpha}{2}}}(x,y)\psi_t(y).
\end{equation}
\end{theorem}

\bigskip

\section{Proof of Theorem \ref{nash_est}: The weighted Nash initial estimate}

The proof follows by applying Theorem \ref{thmB} to $e^{-t\Lambda_r}$.

The conditions ($B_{11}$) and ($B_{13}$) (with $j'=\frac{d}{\alpha}$) are satisfied by Proposition \ref{constr_d}. 
Let us prove ($B_{12}$). By Proposition \ref{prop_contr}  ($\Lambda^\varepsilon \equiv \Lambda^\varepsilon_2$), 
$$
\Real\big\langle \Lambda^\varepsilon(1+\Lambda^\varepsilon)^{-1}g,(1+\Lambda^\varepsilon)^{-1}g\big\rangle \geq c_S \|(1+\Lambda^\varepsilon)^{-1}g\|_{2j}^2, \quad g \in L^2, \quad j=\frac{d}{d-\alpha}, \quad c_S \neq c_S(\varepsilon),
$$
i.e.
$$
\Real\big\langle g-(1+\Lambda^\varepsilon)^{-1}g,(1+\Lambda^\varepsilon)^{-1}g\big\rangle \geq c_S \|(1+\Lambda^\varepsilon)^{-1}g\|_{2j}^2.
$$
Using the convergence $(1+\Lambda^\varepsilon)^{-1} \overset{s}{\rightarrow } (1+\Lambda)^{-1}$ in $L^2$ as $\varepsilon \downarrow 0$
(Proposition \ref{constr_d}), we pass to the limit $\varepsilon \downarrow 0$ in the last inequality
to obtain $\Real\big\langle \Lambda(1+\Lambda)^{-1}g,(1+\Lambda)^{-1}g\big\rangle \geq c_S \|(1+\Lambda)^{-1}g\|_{2j}^2$ for all $g \in L^2$, and so ($B_{12}$) is proven.

The condition ($B_{21}$) is evident from the definition of the weights $\psi_s$. It is easily seen that $(B_{22}),(B_{23})$ hold with $\Omega^s=B(0,s^{\frac{1}{\alpha}})$ and $\theta=\frac{(2-\alpha)d}{(2-\alpha)d+8\beta}$. It remains to prove the desingularizing $(L^1,L^1)$ bound ($B_3$), which presents the main difficulty.

\medskip

\noindent \textit{Proof of {\rm($B_3$)}.}
We modify the proof of the analogous $(L^1,L^1)$ bound in \cite{KSS} (see also Remark \ref{comp_rem} below).  We will appeal to the Lumer-Phillips Theorem applied to specially constructed $C_0$ semigroups in $L^1$, corresponding to operators with smooth coefficients and smooth weights, which approximate $\psi_s e^{-t\Lambda}\psi_s^{-1}$.

Recall that $
b_\varepsilon(x):=\kappa|x|_\varepsilon^{-\alpha}x$, $|x|_\varepsilon:=\sqrt{|x|^2+\varepsilon}$, $\varepsilon>0$,
$$\Lambda^\varepsilon:=(-\Delta)^{\frac{\alpha}{2}} - b_\varepsilon \cdot \nabla, \quad D(\Lambda^\varepsilon)=\mathcal W^{\alpha,1} := \big(1+(-\Delta)^{\frac{\alpha}{2}}\big)^{-1}L^1,$$
$$(\Lambda^\varepsilon)^*=(-\Delta)^{\frac{\alpha}{2}} + \nabla \cdot b_\varepsilon, \quad D(\Lambda^\varepsilon)=\mathcal W^{\alpha,1}.$$
By the Hille Perturbation Theorem, for each $\varepsilon>0$, both $e^{-t \Lambda^\varepsilon}$, $e^{-t(\Lambda^\varepsilon)^*}$ can be viewed as $C_0$ semigroups in  $L^1$ and $C_{u}$ (see Sections \ref{sect_d} and \ref{sect_d2}). 

Define approximating weights
\begin{equation*}
\phi_{n,\varepsilon}:=n^{-1} + e^{-\frac{(\Lambda^\varepsilon)^*}{n}}\psi, \quad \psi=\psi_s.
\end{equation*}

\begin{remark}
This choice of the regularization of $\psi$ is dictated by the method:  $e^{-\frac{(\Lambda^\varepsilon)^*}{n}}$ will be needed below to control the auxiliary potential $U_\varepsilon$. See also Remark \ref{rem_weights} below. 
\end{remark}

In $L^1$ define operators
\[
Q=\phi_{n,\varepsilon} \Lambda^\varepsilon \phi_{n,\varepsilon}^{-1}, \quad D(Q)=\phi_{n,\varepsilon} D(\Lambda^\varepsilon), 
\]
where $\phi_{n,\varepsilon} D(\Lambda^\varepsilon):=\{\phi_{n,\varepsilon} u \mid u \in D(\Lambda^\varepsilon)\}$, 
\[
F_{\varepsilon,n}^t=\phi_{n,\varepsilon} e^{-t\Lambda^\varepsilon}\phi_{n,\varepsilon}^{-1}. 
\]
Since $\phi_{n,\varepsilon}, \phi_{n,\varepsilon}^{-1}\in L^\infty$, these operators are well defined.
In particular, $F^t_{\varepsilon,n}$ are bounded $C_0$ semigroups in $L^1$, say $F^t_{\varepsilon,n}=e^{-tG}$. 

Set 
\begin{align*}
M:=&\,\phi_{n,\varepsilon}(1+(-\Delta)^{\frac{\alpha}{2}})^{-1}[L^1 \cap C_{u}] \\
=&\,\phi_{n,\varepsilon} (\lambda_\varepsilon+\Lambda^\varepsilon)^{-1}[L^1 \cap C_{u}], \quad 0<\lambda_\varepsilon\in \rho(-\Lambda^\varepsilon).
\end{align*}
Clearly, $M$ is a dense subspace of $L^1$, $M\subset D(Q)$ and $M\subset D(G)$. Moreover, $Q\upharpoonright M\subset G$. Indeed, for $f=\phi_{n,\varepsilon} u\in M$,
\[
Gf=s\mbox{-}L^1\mbox{-}\lim_{t\downarrow 0}t^{-1}(1-e^{-tG})f=\phi_{n,\varepsilon} s\mbox{-}L^1\mbox{-}\lim_{t\downarrow 0}t^{-1}(1-e^{-t\Lambda^\varepsilon})u=\phi_{n,\varepsilon} \Lambda^\varepsilon u=Qf.
\]
Thus $Q\upharpoonright M$ is closable and $\tilde{Q}:=(Q\upharpoonright M)^{\rm clos}\subset G$.

\begin{proposition}
\label{prop1}
The range $R(\lambda_\varepsilon +\tilde{Q})$ is dense in $L^1$.
\end{proposition}

\begin{proof}[Proof of Proposition \ref{prop1}]
If $\langle(\lambda_\varepsilon +\tilde{Q})h,v\rangle=0$ for all $h\in D(\tilde{Q})$ and some $v\in L^\infty$, $\|v\|_\infty=1$, then taking $h\in M$ we would have $\langle(\lambda_\varepsilon+Q)\phi_{n,\varepsilon} (\lambda_\varepsilon+\Lambda^\varepsilon)^{-1}g,v\rangle=0$, $g\in L^1 \cap C_{u}$, or $\langle\phi_{n,\varepsilon} g,v\rangle=0$. Choosing $g=e^\frac{\Delta}{k}(\chi_m v)$, where $\chi_m\in C^\infty_c$ with $\chi_m(x)=1$ when $x\in B(0,m)$, we would have $\lim_{k\uparrow\infty}\langle \phi_{n,\varepsilon} g,v\rangle=\langle\phi_n\chi_m,|v|^2\rangle=0$, and so $v= 0$. Thus,  $R(\lambda_\varepsilon +\tilde{Q})$ is dense in $L^1$.
\end{proof}

\begin{proposition}
\label{prop2}
There are constants $\hat{c}>0$ and $\varepsilon_n>0$ such that, for every $n$ and all $0<\varepsilon \leq \varepsilon_n$,
 \[
  \lambda+\tilde{Q} \textit{ is accretive whenever } \lambda\geq \hat{c} s^{-1} + n^{-1}. 
 \] 

\end{proposition}

\begin{proof}[Proof of Proposition \ref{prop2}]
Recall that both $e^{-t\Lambda^\varepsilon}$, $e^{-t(\Lambda^\varepsilon)^*}$ are holomorphic in $L^1$ and $C_u$ due to Hille's Perturbation Theorem. We have $$\psi=\psi_{(1)} + \psi_{(u)}, \qquad 0 \leq \psi_{(1)} \in D((-\Delta)^{\frac{\alpha}{2}}_1), \qquad 0 \leq \psi_{(u)} \in D((-\Delta)^{\frac{\alpha}{2}}_{C_{u}}).$$
For instance, $$\psi_{(u)}:=1+\frac{\beta}{2}, \quad \psi_{(1)}:=\psi-1-\frac{\beta}{2} \quad (\text{so, $\supp \psi_{(1)} \subset B(0,2s^{\frac{1}{\alpha}})$}).$$ In $B(0,s^{\frac{1}{\alpha}})$, the weight $\psi$ coincides with $\tilde{\psi}(x) \equiv \tilde{\psi}_s(x):=s^{-\frac{\beta}{\alpha}}|x|^{\beta}$, so $\psi_{(1)} \in D((-\Delta)_1)$. Thus, $\psi_{(1)} \in D((-\Delta)^{\frac{\alpha}{2}}_1)$ (see, e.g.\,\cite[Ch.V, sect.3.11]{Ka}).
Therefore,
$$(\Lambda^\varepsilon)^*\psi\;\; \big(=(\Lambda^\varepsilon)^*_{L^1}\psi_{(1)} + (\Lambda^\varepsilon)^*_{C_u}\psi_{(u)}\big)$$ is well defined and belongs to $L^1 + C_u=\{w+v\mid w\in L^1, v\in C_u\}$.

\smallskip

We verify that $\Real\langle (\lambda + \tilde{Q})f,\frac{f}{|f|}\rangle \geq 0$ for all $f \in D(\tilde{Q})$.
For $f=\phi_{n,\varepsilon} u\in M$, we have
\begin{align*}
\langle Qf,\frac{f}{|f|}\rangle=&\langle \phi_{n,\varepsilon} \Lambda^\varepsilon u,\frac{f}{|f|}\rangle=\lim_{t\downarrow 0}t^{-1}\langle\phi_{n,\varepsilon}(1-e^{-t \Lambda^\varepsilon})u,\frac{f}{|f|}\rangle,\\
\Real\langle Qf,\frac{f}{|f|}\rangle &\geq\lim_{t\downarrow 0}t^{-1}\langle(1-e^{-t \Lambda^\varepsilon})|u|,\phi_{n,\varepsilon}\rangle \notag \\
&=\lim_{t\downarrow 0}t^{-1}\langle (1-e^{-t\Lambda^\varepsilon}) |u|,n^{-1}\rangle + \lim_{t\downarrow 0}t^{-1}\langle (1-e^{-t\Lambda^\varepsilon}) e^{-\frac{\Lambda^\varepsilon}{n}}|u|,\psi\rangle\\
&=\lim_{t\downarrow 0}t^{-1}\langle |u|,(1-e^{-t(\Lambda^\varepsilon)^*})n^{-1}\rangle+\lim_{t\downarrow 0}t^{-1}\langle e^{-\frac{\Lambda^\varepsilon}{n}}|u|,(1-e^{-t(\Lambda^\varepsilon)^*})\psi\rangle \\
&=\langle |u|, (\Lambda^\varepsilon)^*n^{-1}\rangle + \langle e^{-\frac{\Lambda^\varepsilon}{n}}|u|, (\Lambda^\varepsilon)^*\psi\rangle,
\end{align*}
where the first term is positive since $
(\Lambda^\varepsilon)^*n^{-1}=n^{-1}{\rm div\,} b_\varepsilon=n^{-1}\big(d|x|_\varepsilon^{-\alpha}-\alpha|x|_\varepsilon^{-\alpha-2}|x|^2 \big) \geq n^{-1}(d-\alpha)|x|_\varepsilon^{-\alpha} \geq 0.
$
Thus,
\begin{equation}
\label{Q}
\Real\langle Qf,\frac{f}{|f|}\rangle \geq \langle e^{-\frac{\Lambda^\varepsilon}{n}}|u|, (\Lambda^\varepsilon)^*\psi\rangle,
\end{equation}
so it remains to bound $J:=\langle e^{-\frac{\Lambda^\varepsilon}{n}}|u|, (\Lambda^\varepsilon)^*\psi\rangle$ from below. For that, we estimate from below 
\begin{equation*}
%\label{lambda_id}
(\Lambda^\varepsilon)^*\psi=(-\Delta)^{\frac{\alpha}{2}}\psi + {\rm div\,}  (b_\varepsilon \psi).
\end{equation*}

\begin{claim}
\label{claim_est1}
$(-\Delta)^{\frac{\alpha}{2}}\psi \geq -\beta(d+\beta-2) \frac{\gamma(d+\beta-2)}{\gamma(d+\beta-\alpha)}|x|^{-\alpha}\tilde{\psi}$.
\end{claim}
\begin{proof}[Proof of Claim \ref{claim_est1}]
All identities are in the sense of distributions:
\begin{align*}
(-\Delta)^{\frac{\alpha}{2}}\psi & = - I_{2-\alpha}\Delta \psi  \\
& = - I_{2-\alpha}\Delta \tilde{\psi} - I_{2-\alpha}\Delta (\psi-\tilde{\psi}),
\end{align*}
where $I_\nu=(-\Delta)^{-\frac{\nu}{2}}$ is the Riesz potential, and we evaluate the first term
\begin{align*}
- I_{2-\alpha}\Delta \tilde{\psi} & = - s^{-\frac{\beta}{\alpha}}\beta(d+\beta-2) I_{2-\alpha}|x|^{\beta-2} \\
&= -s^{-\frac{\beta}{\alpha}}\beta(d+\beta-2) \frac{\gamma(d+\beta-2)}{\gamma(d+\beta-\alpha)}|x|^{\beta-\alpha},
\end{align*}
while the second term is positive and can be omitted:
$
- I_{2-\alpha}\Delta (\psi-\tilde{\psi}) \geq 0$ (see Remark \ref{pos_rem} below for detailed calculation). The proof of Claim \ref{claim_est1} is completed. 
\end{proof}

\begin{claim}
\label{claim_est2}
${\rm div\,} (b_\varepsilon \psi) \geq {\rm div\,}  (b \tilde{\psi}) - U_\varepsilon \tilde{\psi} - \hat{c}s^{-1}\psi$ for a constant $\hat{c} \neq \hat{c}(\varepsilon,n)$, 
where $U_\varepsilon(x):=\kappa(d+\beta-\alpha)(|x|^{-\alpha}-|x|_\varepsilon^{-\alpha})>0$.
\end{claim}

\begin{proof}
We represent
$$
{\rm div\,}  (b_\varepsilon \psi) = {\rm div\,}  (b\tilde{\psi})  + {\rm div\,}  (b_\varepsilon \psi) -{\rm div\,}  (b \tilde{\psi})
$$
and estimate the difference ${\rm div\,}  (b_\varepsilon \psi) - {\rm div\,}  (b \tilde{\psi})$:
\begin{align*}
{\rm div\,}  (b_\varepsilon \psi) - {\rm div\,}  ( b \tilde{\psi}) & = {\rm div\,} \bigl[b(\psi-\tilde{\psi})\bigr] + {\rm div\,}  \bigl[(b_\varepsilon-b) \psi \bigr] \\
& = h_1 + {\rm div\,}  \bigl[(b_\varepsilon-b) \psi \bigr],
\end{align*}
where $h_1 \in C_\infty$ (continuous functions vanishing at infinity), $h_1=0$ in $B(0,s^{\frac{1}{\alpha}})$.
In turn,
\begin{align*}
{\rm div\,}  \bigl[(b_\varepsilon-b) \psi \bigr] &= (b_\varepsilon-b) \cdot \nabla \psi + ({\rm div\,}  b_\varepsilon - {\rm div\,}  b)\psi \\
&= \kappa (|x|_\varepsilon^{-\alpha}-|x|^{-\alpha})x \cdot \nabla \tilde{\psi} + h_2 + \kappa\big[d|x|_\varepsilon^{-\alpha}-\alpha |x|_\varepsilon^{-\alpha-2}|x|^2 - (d-\alpha)|x|^{-\alpha} \big]\psi \\
& \text{(where $h_2:=\kappa (|x|_\varepsilon^{-\alpha}-|x|^{-\alpha})x \cdot \nabla (\psi-\tilde{\psi}) \in C_\infty$, $h_2=0$ in $B(0,s^{\frac{1}{\alpha}})$)} \\
& = \kappa (|x|_\varepsilon^{-\alpha}-|x|^{-\alpha})\beta \tilde{\psi} + h_2 + \kappa\big[d|x|_\varepsilon^{-\alpha}-\alpha |x|_\varepsilon^{-\alpha-2}|x|^2 - (d-\alpha)|x|^{-\alpha} \big]\psi \\
& \geq \kappa (|x|_\varepsilon^{-\alpha}-|x|^{-\alpha})\beta \tilde{\psi} + h_2 + \kappa(d-\alpha)(|x|_\varepsilon^{-\alpha}-|x|^{-\alpha})\psi.
\end{align*}
Thus,
\begin{align*}
{\rm div\,}  (b_\varepsilon \psi) \geq {\rm div\,}  ( b \tilde{\psi}) + \kappa(d+\beta-\alpha)(|x|_\varepsilon^{-\alpha}-|x|^{-\alpha})\tilde{\psi}+h_1 + h_2 + h_3,
\end{align*}
where $h_3:=\kappa(d-\alpha)(|x|_\varepsilon^{-\alpha}-|x|^{-\alpha})(\psi-\tilde{\psi}) \in C_\infty$, 
$h_3=0$ in $B(0,s^{\frac{1}{\alpha}})$.

A straightforward calculation shows that $h_i \geq -c_i \psi s^{-1}$ with $c_i \neq c_i(\varepsilon,n)$, $i=1,2,3$ (we have used that $h_i=0$ in $B(0,s^{\frac{1}{\alpha}})$). The assertion of Claim \ref{claim_est2} follows. 
\end{proof}

Now, we combine Claim \ref{claim_est1} and Claim \ref{claim_est2}: In view of the choice of $\beta$,

 $-\beta(d+\beta-2) \frac{\gamma(d+\beta-2)}{\gamma(d+\beta-\alpha)}|x|^{-\alpha}\tilde{\psi} + {\rm div\,}  (b\tilde{\psi})=0$ (that is, formally, $\Lambda^*\tilde{\psi}=0$), and so
$$
(\Lambda^\varepsilon)^* \psi \geq - U_\varepsilon \tilde{\psi} - \hat{c}s^{-1}\psi.
$$

It follows that
\begin{align*}
J \equiv \langle e^{-\frac{\Lambda^\varepsilon}{n}}|u|, (\Lambda^\varepsilon)^*\psi\rangle & \geq -\hat{c}s^{-1}\langle e^{-\frac{\Lambda^\varepsilon}{n}}|u|,\psi \rangle - \langle e^{-\frac{\Lambda^\varepsilon}{n}}|u|,U_\varepsilon \tilde{\psi} \rangle \\
& \geq -\hat{c}s^{-1}\langle |u|,e^{-\frac{(\Lambda^\varepsilon)^*}{n}}\psi \rangle - \langle e^{-\frac{\Lambda^\varepsilon}{n}}|u|,U_\varepsilon \tilde{\psi} \rangle \\
& \geq -\hat{c}s^{-1}\langle |u|,n^{-1} + e^{-\frac{(\Lambda^\varepsilon)^*}{n}}\psi \rangle - \langle e^{-\frac{\Lambda^\varepsilon}{n}}|u|,U_\varepsilon \tilde{\psi} \rangle \\
& (\text{recall that $|u|=\phi_{n,\varepsilon}^{-1} |f|$ and $\phi_{n,\varepsilon}=n^{-1} + e^{-\frac{(\Lambda^\varepsilon)^*}{n}}\psi$}) \\
& = -\hat{c}s^{-1}\|f\|_1 - \langle |u|,e^{-\frac{(\Lambda^\varepsilon)^*}{n}}(U_\varepsilon \tilde{\psi}) \rangle.
\end{align*}
Now, for every $n \geq 1$, we have 
\begin{align*}
\|e^{-\frac{(\Lambda^\varepsilon)^*}{n}}(U_\varepsilon \tilde{\psi})\|_\infty & \leq  \|e^{-\frac{(\Lambda^\varepsilon)^*}{n}}( \mathbf{1}_{B^c(0,R)}U_\varepsilon \tilde{\psi})\|_\infty + \|e^{-\frac{(\Lambda^\varepsilon)^*}{n}}( \mathbf{1}_{B(0,R)}U_\varepsilon \tilde{\psi})\|_\infty \\
& (\text{we are using that $e^{-t(\Lambda^\varepsilon)^*}$ is a $L^\infty$ contraction and ultra-contraction,} \\
& \text{see Proposition \ref{prop_contr2}}) \\
& \leq \|\mathbf{1}_{B^c(0,R)}U_\varepsilon \tilde{\psi}\|_\infty + c_N n^{\frac{d}{\alpha}}\|\mathbf{1}_{B(0,R)}U_\varepsilon \tilde{\psi}\|_1 \\
& (\text{we fix $R=R_n$ such that $\|\mathbf{1}_{B^c(0,R)}U_\varepsilon \tilde{\psi}\|_\infty \leq 2^{-1}n^{-2}$} \\
& \text{and choose $\varepsilon_n>0$ such that for all $\varepsilon\leq \varepsilon_n$ $\|\mathbf{1}_{B(0,R)}U_\varepsilon \tilde{\psi}\|_1 \leq 2^{-1}n^{-2}(c_N n^{\frac{d}{\alpha}})^{-1}$}) \\
& \leq n^{-2}.
\end{align*}
Therefore, since  $\phi_{n,\varepsilon} \geq n^{-1}$, we have for every $n$ and all $\varepsilon\leq \varepsilon_n$ $\|\phi_{n,\varepsilon}^{-1}e^{-\frac{(\Lambda^\varepsilon)^*}{n}}(U_\varepsilon \tilde{\psi})\|_\infty \leq n^{-1}$ and so $\langle |u|,e^{-\frac{(\Lambda^\varepsilon)^*}{n}}(U_\varepsilon \tilde{\psi}) \rangle \leq n^{-1}\|f\|_1$. Thus,
$$
J \geq -\big(\hat{c}s^{-1}+n^{-1}\big)\|f\|_1.
$$
Returning to \eqref{Q}, one can easily see that the latter yields the assertion of  Proposition \ref{prop2}.
\end{proof}

\begin{remark}
\label{pos_rem}
Let us show that $-\Delta (\psi-\tilde{\psi}) \geq 0$. Without loss of generality, $s=1$. 
The inequality is evidently true on $\{0<|x| \leq 1\} \cup \{|x| \geq 2\}$. Now, let $1<|x|<2$. Then
\begin{align*}
\Delta (\tilde{\psi}-\psi) &= \beta(d+\beta-2)|x|^{\beta-2}-\eta''(|x|)|x|^{-2}-\eta'(|x|)(d-1)|x|^{-1} \\
&=\beta(d+\beta-2)|x|^{\beta-2} +  \beta |x|^{-2} - \beta(2-|x|)(d-1)|x|^{-1} \\
&=\beta|x|^{-2}\bigl((d+\beta-2)|x|^{\beta} + 1 - (d-1)(2-|x|)|x| \bigr) \\
& \geq \beta|x|^{-2} \bigl((d+\beta-2)+1-(d-1)\bigr) \geq 0. 
\end{align*}
\hfill \qed
\end{remark}

The fact that $\tilde{Q}$ is closed together with Proposition \ref{prop1} and Proposition \ref{prop2}  imply $R(\lambda_\varepsilon +\tilde{Q})=L^1$ (Appendix \ref{appC}). Then, by the Lumer-Phillips Theorem, $\lambda+\tilde{Q}$ is the (minus) generator of a contraction semigroup, and $\tilde{Q}=G$ due to $\tilde{Q}\subset G$. Thus, it follows that, for all $n$ and all $\varepsilon \leq \varepsilon_n$
 \[
\label{star}
 \|e^{-tG}\|_{1\to 1}\equiv\|\phi_{n,\varepsilon} e^{-t \Lambda^\varepsilon}\phi_{n,\varepsilon}^{-1}\|_{1\to 1}\leq e^{\omega t} , \quad \omega=\hat{c} s^{-1}+n^{-1}. \tag{$\star$}
 \]

\bigskip

To obtain ($B_3$), it remains to pass to the limit in \eqref{star}: first in $\varepsilon \downarrow 0$ and then in $n \rightarrow \infty$. It suffices to prove ($B_3$) on positive functions. By  \eqref{star},
\[
\|\phi_{n,\varepsilon} e^{-t \Lambda^\varepsilon}\phi_{n,\varepsilon}^{-1}f\|_1\leq e^{\omega t}\|f\|_1 , \quad 0\leq f\in L^1,
\]
or taking $f=\phi_{n,\varepsilon}h$, $0 \leq h\in L^1$,
\[
\|\phi_{n,\varepsilon} e^{-t \Lambda^\varepsilon}h\|_1\leq e^{\omega t}\|\phi_{n,\varepsilon}h\|_1.
\]
Using Proposition \ref{constr_d}, we have $$\|\phi_{n,\varepsilon} e^{-t \Lambda^\varepsilon}h\|_1=\langle n^{-1}e^{-t \Lambda^\varepsilon}h\rangle+\langle \psi,e^{-(t+\frac{1}{n})\Lambda^\varepsilon} h\rangle\rightarrow\langle n^{-1}e^{-t \Lambda}h\rangle+\langle \psi,e^{-(t+\frac{1}{n})\Lambda} h\rangle \quad \text{ as } \varepsilon \downarrow 0,$$
and $$\|\phi_{n,\varepsilon}h\|_1 =n^{-1}\langle h\rangle+\langle\psi,e^{-\frac{\Lambda^\varepsilon}{n}}h\rangle\rightarrow n^{-1}\langle h\rangle+\langle\psi,e^{-\frac{\Lambda}{n}}h\rangle \quad \text{ as } \varepsilon\downarrow 0.$$
Thus, $$\langle n^{-1}e^{-t \Lambda}h\rangle+\langle \psi,e^{-(t+\frac{1}{n})\Lambda} h\rangle\leq e^{\omega t}\big(n^{-1}\langle h\rangle+\langle\psi,e^{-\frac{\Lambda}{n}}h\rangle\big).$$ Taking $n\rightarrow\infty$, we obtain $\langle \psi e^{-t\Lambda} h\rangle\leq e^{\hat{c}s^{-1} t}\langle\psi h\rangle.$  ($B_3$) now follows.

\medskip

The proof of Theorem \ref{nash_est} is completed. \hfill \qed

\bigskip

\begin{remark}[On the choice of the regularization $\phi_{n,\varepsilon}$ of the weight $\psi$]
\label{rem_weights}
In \cite{KSS}, we construct the regularization of the weight in the same way as above, although there the factor  $e^{-\frac{1}{n}(\Lambda^\varepsilon)^*}$ serves a different purpose (in \cite{KSS} the drift term $b \cdot \nabla$ has the opposite sign, and so the corresponding weight is unbounded). (As a by-product, this  allows us to consider $(-\Delta)^{\frac{\alpha}{2}}$ perturbed by two drift terms, as in the present paper and as in \cite{KSS}, possibly having singularities at different points.)
\end{remark}

\begin{remark}
\label{comp_rem}
In the proof of the analogous $(L^1,L^1)$ bound in \cite[proof of Theorem 2]{KSS}, where we consider the vector field $b$ of the opposite sign, we first pass to the limit in $n \rightarrow \infty$, and then in $\varepsilon \downarrow 0$. In the proof of Theorem \ref{nash_est} above this order is  naturally reversed. 
\end{remark}

As a consequence of the $(L^1,L^1)$ bound ($B_3$), we obtain

\begin{corollary}
\label{cor1}
$ \langle e^{-t\Lambda}(\cdot,x) \psi_t(\cdot)\rangle \leq c_1  \psi_t(x)$ for all $x \in \mathbb R^d$, $x \neq 0$, $t>0$.
\end{corollary}

As a consequence of Corollary \ref{cor1} and $(NIE_w)$, we obtain
\begin{corollary}
\label{cor2}
$\langle e^{-t\Lambda}(\cdot,x)\rangle =\langle e^{-t\Lambda^*}(x,\cdot)\rangle \leq C_2  \psi_t(x)$ for all $x \in \mathbb R^d$, $x \neq 0$, $t>0$.
\end{corollary}

\begin{proof}
We have
\begin{align*}
 \langle e^{-t\Lambda^*}(x,\cdot)\rangle & \leq \big\langle \mathbf{1}_{B(0,t^{\frac{1}{\alpha}})}(\cdot)e^{-t\Lambda^*}(x,\cdot)\big\rangle + \big\langle \mathbf{1}_{B^c(0,t^{\frac{1}{\alpha}})}(\cdot)e^{-\Lambda^*}(x,\cdot)\psi_t(\cdot)\big\rangle \\
&=:I_1 + I_2.
\end{align*}
By \eqref{nie}, $I_1 \leq c' \psi_t(x)$, and by Corollary \ref{cor1}, $I_2 \leq c'' \psi_t(x)$, for appropriate constants $c'$, $c''<\infty$. Set $C_2:=c'+c''$. 
\end{proof}

\bigskip

\section{Proof of Theorem \ref{nash_west}: The standard upper bounds}

\label{est_prop_proof}

(\textit{i}) For brevity, put $A:=(-\Delta)^\frac{\alpha}{2}$. Recall that 
\begin{equation*}
k_0^{-1} t\bigl(|x-y|^{-d-\alpha} \wedge t^{-\frac{d+\alpha}{\alpha}}\bigr) \leq e^{-tA}(x,y) \leq k_0 t\bigl(|x-y|^{-d-\alpha} \wedge t^{-\frac{d+\alpha}{\alpha}}\bigr)
\end{equation*}
for all $x,y \in \mathbb R^d$, $x \neq y$, $t>0$, for a constant $k_0=k_0(d,\alpha)>1$.

\smallskip

In view of Proposition \ref{constr_d}, it suffices to prove the a priori bound
\begin{equation*}
e^{-t\Lambda^\varepsilon}(x,y)\leq C_1e^{-tA}(x,y), \quad x,y \in \mathbb R^d, \quad t>0, \quad C_1 \neq C_1(\varepsilon).
\end{equation*}
By duality, it suffices to prove 
\begin{equation*}
%\label{sub}
e^{-t(\Lambda^\varepsilon)^*}(x,y)\leq C_1e^{-tA}(x,y),  \quad x,y \in \mathbb R^d, \quad t>0,  \quad C_1\neq C_1(\varepsilon).
\end{equation*}

\bigskip

\textit{\textbf{Step 1:} For every $D>1$ and all $t>0$, $|x|\leq Dt^{\frac{1}{\alpha}}$, $|y|\leq Dt^{\frac{1}{\alpha}}$ the following bound
$$
e^{-t(\Lambda^\varepsilon)^*}(x,y)\leq  k_0 c_N (2D)^{d+\alpha} e^{-tA}(x,y)
$$
is valid.}

\medskip

In fact, we will prove

\begin{lemma}
\label{lem_D}
Let $t>0$ and $D>1$. Then

{\rm(\textit{i})}\qquad $e^{-t(\Lambda^\varepsilon)^*}(x,y)\leq k_0 c_N (2D)^{d+\alpha} e^{-tA}(x,y)$, \qquad $|x|\leq Dt^\frac{1}{\alpha}$, $|y| \leq Dt^\frac{1}{\alpha}$.

\smallskip

{\rm(\textit{ii})} \quad $e^{-t\Lambda^*}(x,y)\leq k_0 c_{N,w} (1+D)^{d+\alpha} e^{-tA}(x,y)\psi_t(x)$, \qquad $|x|\leq t^\frac{1}{\alpha}$, $|y| \leq Dt^\frac{1}{\alpha}$.

\end{lemma}
\begin{proof}
$(i)$ Note that ($|x|\leq Dt^\frac{1}{\alpha}$, $|y| \leq Dt^\frac{1}{\alpha}$) $\Rightarrow$ $t^{-\frac{d}{\alpha}}\leq (2D)^{d+\alpha} t|x-y|^{-d-\alpha}$. The  latter means that $t^{-\frac{d}{\alpha}}\leq k_0 (2D)^{d+\alpha} e^{-tA}(x,y)$. In Proposition \ref{constr_d2}, the Nash initial estimate 
\[
e^{-t(\Lambda^\varepsilon)^*}(x,y)\leq c_N t^{-\frac{d}{\alpha}},  \quad x,y \in \mathbb R^d, \quad t>0 \tag{$NIE$}
\]
is proved. Therefore, 
\[
e^{-t(\Lambda^\varepsilon)^*}(x,y)\leq c_N t^{-\frac{d}{\alpha}}\leq k_0 c_N (2D)^{d+\alpha} e^{-tA}(x,y).
\]

$(ii)$ Clearly, ($|x|\leq Dt^\frac{1}{\alpha}$, $|y| \leq t^\frac{1}{\alpha}$) $\Rightarrow$ $t^{-\frac{d}{\alpha}}\leq (1+D)^{d+\alpha} t|x-y|^{-d-\alpha}$, and so the inequality $t^{-\frac{d}{\alpha}}\leq k_0 (1+D)^{d+\alpha} e^{-tA}(x,y)$ is valid. By $(NIE_w)$ (Theorem \ref{nash_est}), $e^{-t\Lambda^*}(x,y)\leq c_{N,w} t^{-\frac{d}{\alpha}}\psi_t(x)$ for all $t>0$, $x,y\in\mathbb{R}^d$.   Therefore,
\[
e^{-t\Lambda^*}(x,y)\leq k_0 c_{N,w}(1+D)^{d+\alpha} e^{-tA}(x,y)\psi_t(x).
\]
\end{proof}

\bigskip

In what follows, we will need the following estimates.

\begin{lemma}
\label{claim_lem}

Set $E^t(x,y) =t\bigl( |x-y|^{-d-\alpha-1} \wedge t^{-\frac{d+\alpha+1}{\alpha}}\bigr)$,
$E^tf(x):=\langle E^t(x,\cdot)f(\cdot)\rangle$, $t>0$. 

Then there exist constants $k_i$ {\rm($i=1,2,3$)} such that for all $0<t<\infty$, $x$, $y \in \mathbb R^d$

{\rm (\textit{i})} $|\nabla_x e^{-tA}(x,y)| \leq k_1 E^t(x,y)$;

{\rm (\textit{ii})} $\int_0^t \langle e^{-(t-\tau)A}(x,\cdot)E^\tau(\cdot,y)\rangle d\tau \leq k_2 t^\frac{\alpha-1}{\alpha} e^{-tA}(x,y)$;

{\rm (\textit{iii})} $\int_0^t \langle E^{t-\tau}(x,\cdot) E^\tau(\cdot,y)\rangle d\tau \leq k_3 t^\frac{\alpha-1}{\alpha} E^t(x,y).$ 
\end{lemma}

\begin{proof} For the proof of (\textit{i}), (\textit{ii}) see e.g.\,\cite{BJ}. Essentially the same argument yields (\textit{iii}), see e.g.\,\cite[sect.\,5]{KSS} for details. 
\end{proof}

\medskip

\textit{\textbf{Step 2:} Fix $\delta\in ]0,2^{-1}[$.  Set $C_g:=\kappa k_1(2k_2+k_3)$, $R:=(C_g\delta^{-1})^\frac{1}{\alpha-1}$ and $m=1+2k_0 k_1$.}

\textit{If $D\geq Rm$, then the following bound 
\begin{equation}
\label{sub2}
e^{-t(\Lambda^\varepsilon)^*}(x,y)\leq  (1+\delta) e^{-tA}(x,y), \quad x \in \mathbb R^d, \quad |y| > Dt^{\frac{1}{\alpha}}, \quad t>0
\end{equation}
is valid.}

\smallskip

We use the Duhamel formula
\begin{align}
e^{-t(\Lambda^\varepsilon)^*} &=e^{-tA} + \int_0^t e^{-\tau(\Lambda^\varepsilon)^*}(B^t_{\varepsilon,R} + B_{\varepsilon,R}^{t,c})e^{-(t-\tau)A}d\tau \notag \\
& =: e^{-tA} + K_{R}^t + K_{R}^{t,c}, \quad R:=(C_g\delta^{-1})^\frac{1}{\alpha-1},\label{K_R_repr}
\end{align}
where 
\[
B^t_{\varepsilon,R}:=\mathbf{1}_{B(0,Rt^{\frac{1}{\alpha}})}B_\varepsilon, \quad B^{t,c}_{\varepsilon,R}:=\mathbf{1}_{B^c(0,Rt^{\frac{1}{\alpha}})}B_\varepsilon,\quad B_\varepsilon:=-b_\varepsilon \cdot \nabla - W_\varepsilon,
\]
where $W_\varepsilon(x):=\kappa(d|x|_\varepsilon^{-\alpha}-\alpha |x|_\varepsilon^{-\alpha-2}|x|^2)$.

\medskip

Set
$$
M^t_R(x,y):=(d-\alpha)\kappa\int_0^t \langle e^{-\tau(\Lambda^\varepsilon)^*}(x,\cdot)\mathbf{1}_{B(0,Rt^{\frac{1}{\alpha}})}(\cdot)|\cdot|_\varepsilon^{-\alpha} e^{-(t-\tau)A}(\cdot,y)\rangle d\tau.
$$

\begin{claim}
\label{claim_n1}
For every $D\geq Rm$ and all $|y| > Dt^{\frac{1}{\alpha}}$, $x \in \mathbb R^d$, we have
$$
K^t_R(x,y) 
 \leq -\frac{1}{2}M^t_R(x,y).
$$
\end{claim}
\begin{proof}[Proof of Claim \ref{claim_n1}]
Using Lemma \ref{claim_lem}(\textit{i}), we obtain
\begin{align*}
K^t_R(x,y) & \equiv \int_0^t \big\langle e^{-\tau(\Lambda^\varepsilon)^*}(x,\cdot)B^t_{\varepsilon,R}(\cdot)e^{-(t-\tau)A}(\cdot,y)\big\rangle d\tau \\
& \leq k_1\int_0^t \langle e^{-\tau(\Lambda^\varepsilon)^*}(x,\cdot)\mathbf{1}_{B(0,Rt^{\frac{1}{\alpha}})}(\cdot)|b_\varepsilon(\cdot)| E^{t-\tau}(\cdot,y)\rangle d\tau \\ 
&- \int_0^t \langle e^{-\tau(\Lambda^\varepsilon)^*}(x,\cdot)\mathbf{1}_{B(0,Rt^{\frac{1}{\alpha}})}(\cdot)W_\varepsilon(\cdot) e^{-(t-\tau)A}(\cdot,y)\rangle d\tau =:I_1+I_2, 
\end{align*}
where $|b_\varepsilon(x)|=\kappa|x|_\varepsilon^{-\alpha}|x|$.

Using $E^{t-\tau}(z,y) \leq k_0 e^{-(t-\tau)A}(z,y)|z-y|^{-1}$, we obtain 
\begin{align*}
I_1 & \leq  k_0 k_1 \int_0^t   \langle e^{-\tau(\Lambda^\varepsilon)^*}(x,\cdot)\mathbf{1}_{B(0,Rt^{\frac{1}{\alpha}})}(\cdot)|b_\varepsilon(\cdot)|e^{-(t-\tau)A}(\cdot,y)|\cdot-y|^{-1}\rangle d\tau \\
& (\text{we are using $\mathbf{1}_{B(0,Rt^{\frac{1}{\alpha}})}(\cdot)|b_\varepsilon(\cdot)||\cdot-y|^{-1} \leq \mathbf{1}_{B(0,Rt^{\frac{1}{\alpha}})}(\cdot) R(D-R)^{-1}\kappa |\cdot|^{-\alpha}_\varepsilon$}) \\
& \leq k_0 k_1 R(D-R)^{-1}\kappa  \int_0^t   \langle e^{-\tau(\Lambda^\varepsilon)^*}(x,\cdot)\mathbf{1}_{B(0,Rt^{\frac{1}{\alpha}})}(\cdot)|\cdot|_\varepsilon^{-\alpha}e^{-(t-\tau)A}(\cdot,y)\rangle d\tau \\
& = k_0 k_1 R(D-R)^{-1}(d-\alpha)^{-1} M_R^t(x,y).
\end{align*}
We now compare the RHS of the last estimate with $I_2$. Since  $W_\varepsilon(\cdot) \geq \kappa(d-\alpha)|\cdot|_\varepsilon^{-\alpha}$,
we have
\begin{equation*}
K^t_R(x,y)  \leq \bigl(k_0 k_1 R(D-R)^{-1}(d-\alpha)^{-1}-1 \bigr) M_R^t(x,y).
\end{equation*}
Since $k_0 k_1 R(D-R)^{-1}\leq\frac{k_0 k_1}{m-1} \leq \frac{1}{2}$ and $d-\alpha>1$ by our assumptions, we end the proof of Claim \ref{claim_n1}.
\end{proof}

\begin{claim}
\label{claim_n2}
For every $D\geq Rm$ and all $|y| > Dt^{\frac{1}{\alpha}}$, $x \in \mathbb R^d$, we have  
\[
K_R^{t,c}(x,y) \leq \delta(M^t_R(x,y) +  e^{-tA}(x,y)).
\]
\end{claim}
\begin{proof}[Proof of Claim \ref{claim_n2}]
Recall that 
$$
K_R^{t,c}(x,y) \equiv \int_0^t \langle e^{-\tau(\Lambda^\varepsilon)^*}(x,\cdot)B^{t,c}_{\varepsilon,R}(\cdot)e^{-(t-\tau)A}(\cdot,y)\rangle d\tau,
$$
where
$B^{t,c}_{\varepsilon,R}=\mathbf{1}_{B^c(0,Rt^{\frac{1}{\alpha}})}(-b_\varepsilon \cdot \nabla - W_\varepsilon)$.
Thus, discarding in $K_R^{t,c}$ the term containing $-W_\varepsilon$ and using Lemma \ref{claim_lem}(\textit{i}), we obtain
\begin{align}
K_R^{t,c}(x,y) \leq k_1 \kappa R^{1-\alpha}t^{-\frac{\alpha-1}{\alpha}}\int_0^t \big\langle e^{-\tau(\Lambda^\varepsilon)^*}(x,\cdot) E^{t-\tau}(\cdot,y)\big\rangle d\tau. \label{n2} \tag{$\ast$}
\end{align}
We will have to estimate the integral in the RHS of \eqref{n2}.
	
\smallskip

By the Duhamel formula
\begin{align*}
&\int_0^t \big(e^{-\tau(\Lambda^\varepsilon)^*}E^{t-\tau}\big)(x,y)d\tau \\
& = \int_0^t \big(e^{-\tau A}E^{t-\tau}\big)(x,y)d\tau + \int_0^t \int_0^\tau \big( e^{-\tau'(\Lambda^\varepsilon)^*}(B^t_{\varepsilon,R} + B^{t,c}_{\varepsilon,R})e^{-(\tau-\tau')A}d\tau' E^{t-\tau}\big)(x,y) d\tau  \\
&  \equiv \int_0^t \big( e^{-\tau A}E^{t-\tau}\big)(x,y)d\tau + J_R(x,y) + J_R^c(x,y),
\end{align*}
where, by Lemma \ref{claim_lem}(\textit{ii}),
$
\int_0^t \langle \big(e^{-\tau A}(x,\cdot)E^{t-\tau}(\cdot,y)\rangle\big)(x,y) d\tau \leq k_2 t^\frac{\alpha-1}{\alpha} e^{-tA}(x,y).
$
Let us estimate $J_R(x,y)$ and $J_R^c(x,y)$. 

\medskip

In $J_R(x,y)$, discarding the term containing $-W_\varepsilon$ and applying Lemma \ref{claim_lem}(\textit{i}), we obtain
\begin{align*}
J_R(x,y) & \leq k_1 \int_0^t \int_0^\tau \big(e^{-\tau'(\Lambda^\varepsilon)^*}\mathbf{1}_{B(0,Rt^{\frac{1}{\alpha}})}|b_\varepsilon|E^{\tau-\tau'}d\tau' E^{t-\tau}\big)(x,y) d\tau \\
& (\text{we are changing the order of integration and applying Lemma \ref{claim_lem}(\textit{iii})}) \\
& \leq k_1 k_3 \int_0^t \big(e^{-\tau'(\Lambda^\varepsilon)^*}\mathbf{1}_{B(0,Rt^{\frac{1}{\alpha}})}|b_\varepsilon|(t-\tau')^\frac{\alpha-1}{\alpha}E^{t-\tau'}\big)(x,y)d\tau' \\
& \leq k_1 k_3 t^\frac{\alpha-1}{\alpha} \int_0^t \big(e^{-\tau'(\Lambda^\varepsilon)^*}\mathbf{1}_{B(0,Rt^{\frac{1}{\alpha}})}|b_\varepsilon|E^{t-\tau'}\big)(x,y)d\tau'.
\end{align*}
Now, repeating the corresponding argument in the proof of Claim \ref{claim_n1}, we obtain
\[
J_R(x,y) \leq C_2 t^\frac{\alpha-1}{\alpha} M^t_R(x,y),\quad C_2=k_0 k_1 k_3 R(D-R)^{-1}(d-\alpha)^{-1}\leq \frac{k_3}{2}.
\]
($C_2\leq \frac{k_0 k_1 k_3}{m-1}(d-\alpha)^{-1}\leq \frac{k_3}{2}(d-\alpha)^{-1}\leq\frac{k_3}{2}$.)
\medskip

In turn, $J_R^c=\int_0^t (J_R^c)^\tau E^{t-\tau}d\tau$, where
\[
(J_R^c)^\tau :=  \int_0^\tau e^{-\tau'(\Lambda^\varepsilon)^*} B_{\varepsilon,R}^c e^{-(\tau-\tau')A}d\tau'.
\]
Again, discarding the $-W_\varepsilon$ term in $B_{\varepsilon,R}^c$ and applying Lemma \ref{claim_lem}(\textit{i}), we obtain $$|(J_R^c)^\tau(x,y)| \leq \kappa k_1 R^{1-\alpha} t^{-\frac{\alpha-1}{\alpha}}\int_0^\tau \big(e^{-\tau'(\Lambda^\varepsilon)^*} E^{\tau-\tau'}\big)(x,y)d\tau'.$$
Due to Lemma \ref{claim_lem}(\textit{iii}),
\begin{align*}
 |J_R^c(x,y)|  & \leq \kappa k_1 k_3 R^{1-\alpha}  t^{-\frac{\alpha-1}{\alpha}} \int_0^t \langle e^{-\tau'(\Lambda^\varepsilon)^*}(x,\cdot)(t-\tau')^\frac{\alpha-1}{\alpha}E^{t-\tau'}(\cdot,y)\rangle d\tau' \\ 
& \leq \kappa k_1 k_3 R^{1-\alpha} \int_0^t \langle e^{-\tau'(\Lambda^\varepsilon)^*}(x,\cdot)E^{t-\tau'}(\cdot,y)\rangle d\tau'.
\end{align*}

Thus, due to $\kappa k_1 k_3 R^{1-\alpha}\leq \delta < \frac{1}{2}$,
\begin{align*}
&\int_0^t \langle e^{-\tau(\Lambda^\varepsilon)^*}(x,\cdot)E^{t-\tau}(\cdot,y)\rangle d\tau \notag \\ & \leq  k_2 t^\frac{\alpha-1}{\alpha} e^{-tA}(x,y) + 
\frac{k_3}{2} t^\frac{\alpha-1}{\alpha} M^t_R(x,y) 
+ \frac{1}{2}\int_0^t \langle e^{-\tau(\Lambda^\varepsilon)^*}(x,\cdot)E^{t-\tau}(\cdot,y)\rangle d\tau. 
\end{align*}
Thus, we obtain $\int_0^t \langle e^{-\tau(\Lambda^\varepsilon)^*}(x,\cdot)E^{t-\tau}(\cdot,y)\rangle d\tau \notag \leq 2k_2 t^\frac{\alpha-1}{\alpha} e^{-tA}(x,y) +  
k_3 t^\frac{\alpha-1}{\alpha} M^t_R(x,y)$. Substituting the latter in \eqref{n2}, we obtain Claim \ref{claim_n2}.
\end{proof}

Now, applying Claim \ref{claim_n1} and Claim \ref{claim_n2} in \eqref{K_R_repr}, we have
\begin{align*}
\label{final_repr}
e^{-t(\Lambda^\varepsilon)^*}(x,y) &\leq e^{-tA}(x,y) - \frac{1}{2}M^t_R(x,y) + \delta(M^t_R(x,y) +  e^{-tA}(x,y))\\
&\leq (1+\delta)e^{-tA}(x,y),
\end{align*}
thus ending the proof of Step 2.

\bigskip

\textit{\textbf{Step 3:} Set $R=1\vee (2\kappa k_3)^\frac{1}{\alpha-1}$ and let $D\geq 2R$. Then there is a constant $C=C(d,\alpha,\kappa, R)$ such that the following bound 
\begin{equation*}
e^{-t(\Lambda^\varepsilon)^*}(x,y)\leq  C e^{-tA}(x,y), \quad |x| > 2Dt^{\frac{1}{\alpha}}, \quad |y| \leq Dt^{\frac{1}{\alpha}}, \quad t>0.
\end{equation*}
is valid}

(See the proof below for explicit formula for $C(d,\alpha,\kappa,R$.)

\medskip

Using the Duhamel formula and applying Lemma \ref{claim_lem}(\textit{i}), we have

\begin{align}
e^{-t(\Lambda^\varepsilon)^*}(x,y)  & \leq e^{-tA}(x,y) + k_1\int_0^t \big(E^\tau |b_\varepsilon| e^{-(t-\tau)(\Lambda^\varepsilon)^*}\big)(x,y)d\tau \notag \\
& \leq e^{-tA}(x,y)+k_1L^t_{\varepsilon,R}(x,y)+ k_1L^{t,c}_{\varepsilon,R}(x,y). \label{e_est}
\end{align}
where
$$
L^{t}_{\varepsilon,R}(x,y):=\int_0^t \big(E^\tau \mathbf{1}_{B(0,Rt^{\frac{1}{\alpha}})}|b_\varepsilon| e^{-(t-\tau)(\Lambda^\varepsilon)^*}\big)(x,y)d\tau,
$$
$$
L^{t,c}_{\varepsilon,R}(x,y):=\int_0^t \big(E^\tau \mathbf{1}_{B^c(0,Rt^{\frac{1}{\alpha}})}|b_\varepsilon| e^{-(t-\tau)(\Lambda^\varepsilon)^*}\big)(x,y)d\tau.
$$

\medskip

Let us estimate $L^t_{\varepsilon,R}(x,y)$. 
Recalling that $E^t(x,z) =t\bigl( |x-z|^{-d-\alpha-1} \wedge t^{-\frac{d+\alpha+1}{\alpha}}\bigr)$ and taking into account that $|x| \geq 2Dt^{\frac{1}{\alpha}}$, $|z| \leq Rt^{\frac{1}{\alpha}}$, we obtain
$
E^\tau(x,z) \leq t |x-z|^{-d-\alpha-1} \leq t |x-z|^{-d-\alpha} (3R)^{-1} t^{-\frac{1}{\alpha}}.
$
Therefore,
\begin{align*}
L^t_{\varepsilon,R}(x,y) & \leq (3R)^{-1} t^{-\frac{1}{\alpha}} \int_0^t \langle t|x-\cdot|^{-\alpha-d} \mathbf{1}_{B(0,Rt^{\frac{1}{\alpha}})}(\cdot)|b_\varepsilon(\cdot)| e^{-(t-\tau)(\Lambda^\varepsilon)^*}(\cdot,y)\rangle d\tau  \\
& (\text{we are using that $|x| >2Dt^\frac{1}{\alpha} $,  $|\cdot| \leq Rt^{\frac{1}{\alpha}}$}) \\
& \leq (3R)^{-1}(4/3)^{d+\alpha }t^{-\frac{1}{\alpha}}t|x|^{-\alpha-d} \int_0^t \langle \mathbf{1}_{B(0,Rt^{\frac{1}{\alpha}})}(\cdot)|b_\varepsilon(\cdot)| e^{-(t-\tau)(\Lambda^\varepsilon)^*}(\cdot,y)\rangle d\tau\\
& (\text{we are using that $|y| \leq Dt^\frac{1}{\alpha}$, $D\geq 2R$ and setting $c=3^{-1}(16/9)^{d+\alpha}$}) \\
& \leq c R^{-1} t^{-\frac{1}{\alpha}}t|x-y|^{-\alpha-d} \int_0^t \langle \mathbf{1}_{B(0,Rt^{\frac{1}{\alpha}})}(\cdot)|b_\varepsilon(\cdot)| e^{-(t-\tau)(\Lambda^\varepsilon)^*}(\cdot,y)\rangle d\tau  \\
& (\text{we are using $t|x-y|^{-\alpha-d} = t(|x-y|^{-\alpha-d} \wedge t^{-\frac{d+\alpha}{\alpha}})$} \\
& \text{since $|x-y|^{-\alpha-d}\leq (2R)^{-d-\alpha} t^{-\frac{d+\alpha}{\alpha}}<t^{-\frac{d+\alpha}{\alpha}}$, and are re-denoting $t-\tau$ by $\tau$}) \\
& \leq k_0 c R^{-1} t^{-\frac{1}{\alpha}}e^{-tA}(x,y) \int_0^t \|e^{-\tau\Lambda^\varepsilon}\mathbf{1}_{B(0,Rt^{\frac{1}{\alpha}})}|b|\|_{\infty} d\tau  \\
& (\text{we are applying Proposition \ref{prop_contr}}) \\
& \leq k_0 c R^{-1} t^{-\frac{1}{\alpha}}e^{-tA}(x,y) c_N \int_0^t \tau^{-\frac{d}{\alpha p}} d\tau\,\|\mathbf{1}_{B(0,Rt^{\frac{1}{\alpha}})}|b|\|_p \qquad \bigl(p=\frac{d}{\alpha-\frac{1}{2}}\bigr).
\end{align*}
Since $\int_0^t \tau^{-\frac{d}{\alpha p}} d\tau=2\alpha t^\frac{1}{2\alpha}$ and $\|\mathbf{1}_{B(0,Rt^{\frac{1}{\alpha}})}|b|\|_p=\kappa R^\frac{1}{2}t^\frac{1}{2\alpha} \tilde{c} $, $\tilde{c}=\tilde{c}(d)<\infty$, we have
\[
L^t_{\varepsilon,R}(x,y) \leq C'R^{-\frac{1}{2}} e^{-tA}(x,y), \quad C'=2\kappa\alpha k_0 c c_N\tilde{c} 
\]
or, for convenience,
\begin{equation}
\label{l_est}
L^t_{\varepsilon,R}(x,y) \leq C' e^{-tA}(x,y).
%, \quad C'=\kappa c(d,\alpha).
\end{equation}

In turn, clearly,
$$
L^{t,c}_{\varepsilon,R}(x,y) \leq \kappa R^{1-\alpha}t^{-\frac{\alpha-1}{\alpha}}\int_0^t E^\tau  e^{-(t-\tau)(\Lambda^\varepsilon)*}d\tau.
$$
Let us estimate the integral in the RHS. Using the Duhamel formula, we obtain
\begin{align*}
& \int_0^t \big(E^\tau  e^{-(t-\tau)(\Lambda^\varepsilon)^*}\big)(x,y)d\tau \\
& \leq  \int_0^t \big(E^\tau  e^{-(t-\tau)A}\big)(x,y)d\tau + \int_0^t \big(E^\tau \int_0^{t-\tau} E^{t-\tau-s}|b_\varepsilon|e^{-s(\Lambda^\varepsilon)^*}ds\big)(x,y) d\tau \\
& (\text{we are applying Lemma \ref{claim_lem}(\textit{ii}) and changing the order of integration}) \\
& \leq k_2 t^\frac{\alpha-1}{\alpha}e^{-tA}(x,y) + \int_0^t \int_0^{t-s} \big(E^\tau E^{t-s-\tau}|b_\varepsilon|e^{-s(\Lambda^\varepsilon)^*}\big)(x,y)d\tau ds \\
& (\text{we are applying Lemma \ref{claim_lem}(\textit{iii})}) \\
& \leq k_2 t^\frac{\alpha-1}{\alpha} e^{-tA}(x,y) + k_3\int_0^t (t-s)^\frac{\alpha-1}{\alpha} \big(E^{t-s}|b_\varepsilon|e^{-s(\Lambda^\varepsilon)^*}\big)(x,y) ds \\
& \leq k_2 t^\frac{\alpha-1}{\alpha} e^{-tA}(x,y) + k_3 t^\frac{\alpha-1}{\alpha} \int_0^t  \big(E^{t-s}\mathbf{1}_{B(0,Rt^{\frac{1}{\alpha}})}|b_\varepsilon|e^{-s(\Lambda^\varepsilon)^*}\big)(x,y)d\tau ds \\
& + k_3 t^\frac{\alpha-1}{\alpha}  \int_0^t  \big(E^{t-s}\mathbf{1}_{B^c(0,Rt^{\frac{1}{\alpha}})}|b|e^{-s(\Lambda^\varepsilon)^*}\big)(x,y) ds \\
& \leq k_2 t^\frac{\alpha-1}{\alpha} e^{-tA}(x,y) + k_3 t^\frac{\alpha-1}{\alpha} L^t_{\varepsilon,R}(x,y) + k_3\kappa R^{1-\alpha}\int_0^t  \big(E^{t-s}e^{-s(\Lambda^\varepsilon)^*}\big)(x,y) ds \\
& (\text{we are applying \eqref{l_est} to the second term, and note that $k_3 \kappa R^{1-\alpha}\leq \frac{1}{2}$}) \\
& \leq (k_2 + k_3C') t^\frac{\alpha-1}{\alpha} e^{-tA}(x,y) + \frac{1}{2} \int_0^t  \big(E^{t-s}e^{-s(\Lambda^\varepsilon)^*}\big)(x,y) ds.
\end{align*}
Therefore,
$$
\int_0^t E^\tau  \big(e^{-(t-\tau)(\Lambda^\varepsilon)*}\big)(x,y)d\tau \leq 2(k_2+k_3C') t^\frac{\alpha-1}{\alpha}e^{-tA}(x,y),
$$
and so 
\begin{equation}
\label{l_est2}
L^{c,t}_{\varepsilon,R}(x,y) \leq 2\kappa (k_2+k_3C')R^{1-\alpha}e^{-tA}(x,y).
\end{equation}

Applying \eqref{l_est} and \eqref{l_est2} in \eqref{e_est}, we obtain the desired bound
$$
e^{-t(\Lambda^\varepsilon)^*}(x,y) \leq C e^{-tA}(x,y), \quad |x| > 2Dt^{\frac{1}{\alpha}}, \quad |y| \leq Dt^{\frac{1}{\alpha}},
$$
for all $R>1$ such that $k_3\kappa R^{1-\alpha}\leq \frac{1}{2}$, $D\geq 2R$, 
where $C:=1+k_1C'  + k_1 2\kappa(k_2+k_3C')R^{1-\alpha}$.
The assertion of Step 3 follows.

\medskip

We are in position to complete the proof of Theorem \ref{nash_west}(\textit{i}), i.e.
to prove the bound 
\begin{equation}
\label{sub}
e^{-t(\Lambda^\varepsilon)^*}(x,y) \leq C_1 e^{-tA}(x,y), \quad x,y \in \mathbb R^d, \quad t>0,
\end{equation}
for appropriate constant $C_1=C_1(d,\alpha,\kappa)$.

To prove \eqref{sub}, we combine Steps 1-3 as follows. \textit{Fix} $D$ large enough so that the assertions of both Step 2 and Step 3 hold. 

 Without loss of generality, the assertion of Step 3 holds for all
$|x| > Dt^{\frac{1}{\alpha}}$, $|y| \leq D t^{\frac{1}{\alpha}}$
(indeed, by Step 1, \eqref{sub} is true  for all $|x| \leq 2Dt^{\frac{1}{\alpha}}$, $|y| \leq 2D t^{\frac{1}{\alpha}}$ (with $C_1=C_0'(4D)^{d+\alpha}$) and so, in particular, for all $Dt^{\frac{1}{\alpha}} < |x| \leq 2Dt^{\frac{1}{\alpha}}$, $|y| \leq Dt^{\frac{1}{\alpha}}$; the rest follows from the assertion of Step 3 as stated). 
Thus, the desired bound \eqref{sub} is true for all $|x| > Dt^{\frac{1}{\alpha}}$, $|y| \leq D t^{\frac{1}{\alpha}}$ and, by Step 2, for all $x \in \mathbb R^d$, $|y|>D t^{\frac{1}{\alpha}}$. 

It remains to prove \eqref{sub} in the case $|x| \leq Dt^{\frac{1}{\alpha}}$, $|y| \leq D t^{\frac{1}{\alpha}}$. But this is the assertion of Step 1.
 
Thus, \eqref{sub} is true, with constant $C_1$ equal to the maximum of the constants in Step 1 (with $2D$ in place of $D$) and in Steps 2, 3.

\bigskip

(\textit{ii}) The result follows immediately from Step 2 in the proof of (\textit{i}) upon taking $\varepsilon \downarrow 0$ (cf.\,Proposition \ref{constr_d2}).

\medskip

The proof of Theorem \ref{nash_west} is completed. \hfill \qed

\bigskip

\section{Proof of Theorem \ref{thm_est2}: The weighted upper bound}

Recall $A \equiv (-\Delta)^{\frac{\alpha}{2}}$. 
We are going to prove that there is a constant $C<\infty$ such that
\begin{equation}
\label{min_ine}
e^{-t\Lambda}(x,y) \leq C e^{-tA}(x,y)\psi_t(y), \quad t>0, \quad x,y\in \mathbb R^d.
\end{equation}

Clearly, Theorem \ref{nash_est} and Theorem \ref{nash_west}(\textit{i}) combined, yield
\begin{align}
\label{min_ineq}
e^{-t\Lambda}(x,y)  \leq C_1c_{N,w}\biggl(e^{-tA}(x,y) \wedge \big( t^{-\frac{d}{\alpha}}\psi_t(y)\big)\biggr), \quad t>0, \quad x,y\in \mathbb R^d.
\end{align}

1.~If $|y| \geq t^\frac{1}{\alpha}$, then $\psi_t(y) \geq 1$. Then, by \eqref{min_ineq}, $$e^{-t\Lambda}(x,y)  \leq C_1c_{N,w}e^{-tA}(x,y) \leq  C_1c_{N,w} e^{-tA}(x,y)\psi_t(y),$$
i.e.\,\eqref{min_ine} holds.

\smallskip

2.~If $|x|\leq Dt^\frac{1}{\alpha}$, $|y|<t^\frac{1}{\alpha}$ for some constant $D>1$, then by \eqref{min_ineq} (cf.\,Lemma \ref{lem_D}(\textit{i}))
$$
e^{-t\Lambda}(x,y)  \leq C_1 c_{N,w} t^{-\frac{d}{\alpha}}\psi_t(y) \leq C_1 c_{N,w} k_0^{-1}(D+1)^{d+\alpha}e^{-tA}(x,y) \psi_t(y),
$$
i.e.\,\eqref{min_ine} holds.

\smallskip

3.~It remains therefore to consider the case $|x|> Dt^\frac{1}{\alpha}$, $|y|< t^\frac{1}{\alpha}$. 

\medskip

By duality (cf.\,Proposition \ref{constr_d2}), it suffices to prove the estimate 
\begin{equation}
\label{eq_8}
e^{-t\Lambda^*}(x,y) \leq C e^{-tA}(x,y)\psi_t(x)
\end{equation}
for all $|x|<t^{\frac{1}{\alpha}}$, $|y|>Dt^{\frac{1}{\alpha}}$, $t>0$, for some  $D>1$.

We will use Corollary \ref{cor2}, 
$$
\langle e^{-t\Lambda^*}(x,\cdot)\rangle \leq C_2  \psi_t(x) \quad \text{ for all } x \in \mathbb R^d, \quad t>0,$$
the ``standard'' upper bound (Theorem \ref{nash_west}(\textit{i}))
\begin{equation*}
e^{-t\Lambda^*}(x,y) \leq C_1 e^{-tA}(x,y), \quad \text{ for all }x,y \in \mathbb R^d, \quad t>0,
\end{equation*}
and its partial improvement (Theorem \ref{nash_west}(\textit{ii})): For every $\delta>0$ there exists a sufficiently large $D$ such that for all $|x|<t^{\frac{1}{\alpha}}$, $|y|>Dt^{\frac{1}{\alpha}}$ and all $z \in B(y,\frac{|y-x|}{2})$
\begin{equation}
\label{e1}
e^{-t\Lambda^*}(x,z) \leq C_\delta e^{-tA}(x,z), \qquad e^{-t\Lambda^*}(z,y) \leq C_\delta e^{-tA}(z,y), \qquad C_\delta:=1+\delta.
\end{equation}

We will need the following elementary inequality: 
\begin{equation}
\label{half_est}
2\,\big\langle \mathbf{1}_{B(y,\frac{|x-y|}{2})}(\cdot)e^{-\frac{t}{2}A}(x,\cdot)  e^{-\frac{t}{2}A}(\cdot,y)\big\rangle \leq e^{-tA}(x,y).
\end{equation}
Indeed, by symmetry, the LHS of \eqref{half_est} coincides with
\begin{align*}
\big\langle \mathbf{1}_{B(y,\frac{|x-y|}{2})}(\cdot)e^{-\frac{t}{2}A}(x,\cdot)  e^{-\frac{t}{2}A}(\cdot,y)\big\rangle & + \big\langle \mathbf{1}_{B(x,\frac{|x-y|}{2})}(\cdot)e^{-\frac{t}{2}A}(x,\cdot)  e^{-\frac{t}{2}A}(\cdot,y)\big\rangle \\
& \leq \langle e^{-\frac{t}{2}A}(x,\cdot)  e^{-\frac{t}{2}A}(\cdot,y)\rangle=e^{-tA}(x,y),
\end{align*}
i.e.\,\eqref{half_est} follows.

\begin{proposition}
\label{prop_est3}
{\rm(\textit{i})} There exists a constant $c_5$ such that
$$
e^{-t\Lambda^*}(x,y) \leq \big\langle \mathbf{1}_{B(y,\frac{|x-y|}{2})}(\cdot) e^{-\frac{t}{2}\Lambda^*}(x,\cdot)  e^{-\frac{t}{2}\Lambda^*}(\cdot,y)\big\rangle + c_5 e^{-tA}(x,y)\psi_t(x)
$$

{\rm(\textit{ii})} If $|x|<t^{\frac{1}{\alpha}}$, $|y|>Dt^{\frac{1}{\alpha}}$ with $D>1$ sufficiently large, then 
$$
e^{-t\Lambda^*}(x,y) \leq \biggl(\frac{C^2_\delta}{2} +c_5\psi_t(x) \biggr)e^{-tA}(x,y).
$$

\end{proposition}

\begin{proof}
We have
\begin{align*}
e^{-t\Lambda^*}(x,y) & = \big\langle \mathbf{1}_{B(y,\frac{|x-y|}{2})}(\cdot) e^{-\frac{t}{2}\Lambda^*}(x,\cdot)  e^{-\frac{t}{2}\Lambda^*}(\cdot,y)\big\rangle + \big\langle \mathbf{1}_{B^c(y,\frac{|x-y|}{2})} e^{-\frac{t}{2}\Lambda^*}(x,\cdot)  e^{-\frac{t}{2}\Lambda^*}(\cdot,y)\big\rangle \\
& =:J_1+J_2.
\end{align*}
(\textit{i}) For $z \in B^c(y,\frac{|x-y|}{2})$, $e^{-\frac{t}{2}\Lambda^*}(z,y) \leq C_1 e^{-\frac{t}{2}A}(z,y) \leq k_1 e^{-tA}(x,y)$. Thus,
\begin{align*}
J_2 & \leq k_1 e^{-tA}(x,y) \big\langle\mathbf{1}_{B^c(y,\frac{|x-y|}{2})}(\cdot) e^{-\frac{t}{2}\Lambda^*}(x,\cdot)\big\rangle \\
& \text{(we are applying Corollary \ref{cor2})} \\
& \leq k_1C_2 e^{-tA}(x,y)\psi_{\frac{t}{2}}(x) \leq c_5 e^{-tA}(x,y)\psi_{t}(x),
\end{align*}
and so (\textit{i}) follows.

(\textit{ii}) Using (\textit{i}), it remains to estimate $J_1$. Applying \eqref{e1}, 
%Theorem \ref{nash_west}(\textit{ii}) first to $e^{-\frac{t}{2}\Lambda^*}(x,\cdot)$ in the definition of $J_1$ (we use that $|x|<t^{\frac{1}{\alpha}}$, $|\cdot| \geq \frac{1}{3}D t^{\frac{1}{\alpha}}$), and then to $e^{-\frac{t}{2}\Lambda^*}(z,y)$ (we use that $|z| \leq 2|y|$, $|y|>Dt^{\frac{1}{\alpha}}$) 
we have
$$
J_1 \leq C_\delta^2 \big\langle\mathbf{1}_{B(y,\frac{|x-y|}{2})}(\cdot) e^{-\frac{t}{2}A}(x,\cdot)  e^{-\frac{t}{2}A}(\cdot,y)\big\rangle
$$
Finally, we use \eqref{half_est}.
\end{proof}

Let us complete the proof of Theorem \ref{thm_est2}.

By Proposition \ref{prop_est3}(\textit{ii}),
$$
e^{-t\Lambda^*}(x,y) \leq \biggl(\frac{C^2_\delta}{2} +c_5\psi_t(x) \biggr)e^{-tA}(x,y).
$$
Set $\nu:=\frac{C_\delta}{2} 2^{\frac{\beta}{\alpha}}$, so that $\frac{C_\delta}{2}\psi_{t/2} = \nu \psi_t$. Fix $\delta \in \big]0,(\sqrt{2}-1) \wedge (2^{1-\frac{\alpha}{\beta}}-1)\big[$.  Then $\frac{C^2_\delta}{2}<1$ and $\nu<1$. 
Now, suppose that, for $n=2,3,\dots$,
\begin{equation}
\label{half_west}
e^{-t\Lambda^*}(x,y) \leq \biggl(\frac{C_\delta^{n+1}}{2^n}+c_5(1+\nu+\dots+\nu^{n-1})\psi_t(x) \biggr)e^{-tA}(x,y),
\end{equation}
Then, using Proposition \ref{prop_est3}(\textit{i}), we have
\begin{align*}
e^{-t\Lambda^*}(x,y) & \leq \langle\mathbf{1}_{B(y,\frac{|x-y|}{2})}(\cdot) e^{-\frac{t}{2}\Lambda^*}(x,\cdot) C_\delta e^{-\frac{t}{2}A}(\cdot,y)\big\rangle + c_5 e^{-tA}(x,y)\psi_t(x) \\
& \leq \big\langle\mathbf{1}_{B(y,\frac{|x-y|}{2})}(\cdot) C_\delta\biggl(\frac{C_\delta^{n+1}}{2^n}+c_5(1+\nu+\dots+\nu^{n-1})\psi_\frac{t}{2}(x) \biggr)e^{-\frac{t}{2}A}(x,\cdot) e^{-\frac{t}{2}A}(\cdot,y)\big\rangle \\
& + c_5 e^{-tA}(x,y)\psi_t(x) \\
& (\text{we are applying \eqref{half_est}}) \\
&  \leq  \biggl(\frac{C_\delta^{n+2}}{2^{n+1}}+c_5(\nu+\nu^2+\dots+\nu^{n})\psi_t(x) \biggr)e^{-tA}(x,y) + c_5 e^{-tA}(x,y)\psi_t(x) \\
& = \biggl(\frac{C_\delta^{n+2}}{2^{n+1}}+c_5(1+\nu+\nu^2+\dots+\nu^{n})\psi_t(x) \biggr)e^{-tA}(x,y).
\end{align*}
Thus by induction, \eqref{half_west} holds for $n+1$.  Sending $n \rightarrow \infty$ there, we obtain
$$
e^{-t\Lambda^*}(x,y) \leq c_5(1-\nu)^{-1}e^{-tA}(x,y)\psi_t(x),
$$
as needed. The proof of \eqref{eq_8} is completed. The proof of Theorem \ref{thm_est2} is completed.

\bigskip

\section{Proof of Theorem \ref{thm_lb}: The weighted lower bound}

Recall that 
\begin{equation}
\label{st_bd}
k_0^{-1} t\bigl(|x-y|^{-d-\alpha} \wedge t^{-\frac{d+\alpha}{\alpha}}\bigr) \leq e^{-tA}(x,y) \leq k_0 t\bigl(|x-y|^{-d-\alpha} \wedge t^{-\frac{d+\alpha}{\alpha}}\bigr)
\end{equation}
for all $x,y \in \mathbb R^d$, $x \neq y$, $t>0$, for a constant $k_0=k_0(d,\alpha)>1$.

\medskip

\textbf{1.~}First, we prove the ``standard'' lower bound away from the origin.

\begin{lemma} 
\label{lem1}
There exists a generic constant $0<\gamma<\frac{1}{2}$ such that, for all $r\geq\gamma^{-2}$ and $t>0$,
\[
e^{-t\Lambda^*}(x,y) \geq \frac{1}{2} e^{-tA}(x,y)  
\]
whenever $|x| \geq rt^\frac{1}{\alpha},\; |y| \geq rt^\frac{1}{\alpha}$.
\end{lemma}

\begin{proof}
In view of Proposition \ref{constr_d} it suffices to prove the inequality $e^{-t(\Lambda^\varepsilon)^*}(x,y) \geq \frac{1}{2} e^{-tA}(x,y)$.

By the Duhamel formula,
$$
e^{-t(\Lambda^\varepsilon)^*}(x,y) \geq e^{-tA}(x,y) -  |M_t(x,y)|, \qquad M_t(x,y) := \int_0^t e^{-(t-\tau)A} \nabla \cdot b_\varepsilon\, e^{-\tau(\Lambda^\varepsilon)^*}d\tau.
$$
Using Lemma \ref{claim_lem}(\textit{i}), we have
\begin{align*}
|M_t(x,y)| & \leq k_1 \kappa\int_0^t \langle E^{t-\tau}(x,\cdot) |\cdot|^{-\alpha+1} e^{-\tau(\Lambda^\varepsilon)^*}(\cdot,y)\rangle d\tau  \\
& (\text{we are using Theorem \ref{nash_west}(\textit{i}) -- the standard upper bound}) \\
& \leq  k_1 \kappa C_1 \int_0^t \langle E^{t-\tau}(x,\cdot) |\cdot|^{-\alpha+1} e^{-\tau A} (\cdot,y)\rangle d\tau.
\end{align*}
Set
\begin{align*}
J(\mathbf{1}_{B(0,\gamma rt^\frac{1}{\alpha})} (|\cdot|^{1-\alpha})&:=\int_0^t \langle \mathbf{1}_{B(0,\gamma rt^\frac{1}{\alpha})}(\cdot)E^{t-\tau}(x,\cdot) |\cdot|^{-\alpha+1} e^{-\tau A} (\cdot,y)\rangle d\tau,\\
J(\mathbf{1}_{B^c(0,\gamma rt^\frac{1}{\alpha})} (|\cdot|^{1-\alpha})&:=\int_0^t \langle \mathbf{1}_{B^c(0,\gamma rt^\frac{1}{\alpha})}(\cdot)E^{t-\tau}(x,\cdot) |\cdot|^{-\alpha+1} e^{-\tau A} (\cdot,y)\rangle d\tau,
\end{align*} 
where $0<\gamma<2^{-1}$.

Note that if $|x|\geq rt^\frac{1}{\alpha}$,then
\[
E^{t-\tau}(x,z) \leq C_5e^{-(t-\tau)A}(x,z)|x-z|^{-1}\leq C_52r^{-1}t^{-\frac{1}{\alpha}}e^{-(t-\tau)A}(x,z)\quad z\in B(0,\gamma rt^\frac{1}{\alpha}).  
\]
Thus, using the inequality 
\begin{equation}
\label{3P}
e^{-tA}(x,z)e^{-s A}(z,y) \leq K e^{-(t+s )A}(x,y)\bigl(e^{-tA}(x,z) + e^{-s A}(z,y) \bigr), 
\end{equation}
which holds for a constant $K=K(d,\alpha)$, all $x,z,y \in \mathbb R^d$ and $t,s>0$ (see e.g. \cite{BJ}), we have 
\[
J(\mathbf{1}_{B(0,\gamma rt^\frac{1}{\alpha})}|\cdot|^{1-\alpha}) \leq C_52r^{-1}t^{-\frac{1}{\alpha}}Ke^{-tA}(x,y)\int_0^t \langle \mathbf{1}_{B(0,\gamma rt^\frac{1}{\alpha})}(\cdot)|\cdot|^{1-\alpha}(e^{-(t-\tau) A}(x,\cdot)+e^{-\tau A}(\cdot,y))\rangle d\tau.
\]
Next, for all $0<\tau < t$, $|x|\geq rt^\frac{1}{\alpha}$, $|y| \geq rt^\frac{1}{\alpha}$,
\begin{align*}
\mathbf{1}_{B(0,\gamma rt^\frac{1}{\alpha})}(\cdot)e^{-\tau A}(\cdot,y) & \leq C_6t^{-\frac{d}{\alpha}}r^{-d-\alpha} \quad\text{ if } (1-\gamma)r> 1,\\
\mathbf{1}_{B(0,\gamma rt^\frac{1}{\alpha})}(\cdot)e^{-(t-\tau)A}(x,\cdot) & \leq C_7t^{-\frac{d}{\alpha}}r^{-d-\alpha},\quad\text{ if } (1-\gamma)r> 1,
\end{align*}
and so
\begin{align*}
J(\mathbf{1}_{B(0,\gamma rt^\frac{1}{\alpha})}|\cdot|^{1-\alpha}) &\leq C_8t^{-\frac{d+1}{\alpha}}r^{-d-\alpha-1}e^{-tA}(x,y)\int_0^t \langle \mathbf{1}_{B(0,\gamma rt^\frac{1}{\alpha})}(\cdot)|\cdot|^{1-\alpha}\rangle d\tau\\
&\leq C_9r^{-2\alpha}\gamma^{d-\alpha+1}e^{-tA}(x,y)\\
&\leq C_92^{2\alpha}\gamma^{d-\alpha+1}e^{-tA}(x,y)\quad\text{ if } r>(1-\gamma)^{-1}.
\end{align*}
Therefore,
\[
J(\mathbf{1}_{B(0,\gamma rt^\frac{1}{\alpha})}|\cdot|^{1-\alpha})\leq C_{10}\gamma^{d-\alpha+1}e^{-tA}(x,y) \quad\text{ if } \quad r>(1-\gamma)^{-1},\quad 0<\gamma<2^{-1}. \tag{$\ast$}
\]
In turn,
\[
J(\mathbf{1}_{B^c(0,\gamma rt^\frac{1}{\alpha})}|\cdot|^{1-\alpha}) \leq \frac{c_1 C}{2}C_0 (\gamma rt^\frac{1}{\alpha})^{1-\alpha} t^{1-\frac{1}{\alpha}}e^{-tA}(x,y)=C_{11}(\gamma r)^{1-\alpha}e^{-tA}(x,y)
\]
as follows immediately from Lemma \ref{claim_lem}(\textit{ii}): 
$$
\int_0^t \langle e^{-(t-\tau)A}(x,\cdot)E^{\tau}(\cdot,y)\rangle d \tau \leq C_0 t^{1-\frac{1}{\alpha}}e^{-tA}(x,y).
$$
Thus, if $r\geq\gamma^{-2}$, then
\[
J(\mathbf{1}_{B^c(0,\gamma rt^\frac{1}{\alpha})}|\cdot|^{1-\alpha}) \leq C_{11}\gamma^{1-\alpha}e^{-tA}(x,y).\tag{$\ast\ast$}
\]
Finally, selecting $\gamma>0$ sufficiently small: $k_1 \kappa C(C_{10}\vee C_{11})\gamma^{\alpha-1}\leq\frac{1}{4}$, and using $(\ast)$, $(\ast\ast)$, we have
$$
|M_t(x,y)| \leq \frac{1}{2} e^{-tA}(x,y),
$$
which ends the proof.
\end{proof}

\begin{corollary}
\label{cor11}
For every $r>0$, there is a constant $c(r)>0$ such that
\[
e^{-t\Lambda^*}(x,y) \geq c(r) e^{-tA}(x,y)
\]
whenever $|x| \geq rt^\frac{1}{\alpha}$, $|y| \geq rt^\frac{1}{\alpha}$, $t>0$.
\end{corollary}
\begin{proof}
In Lemma \ref{lem1}, fix some $r \geq \gamma^{-2}$, so that
\begin{equation}
\label{e11}
e^{-t\Lambda^*}(x,y) \geq 2^{-1} e^{-tA}(x,y), \quad |x| \geq rt^\frac{1}{\alpha}, \quad |y| \geq rt^\frac{1}{\alpha},
\end{equation}
\begin{equation}
\label{e2}
e^{-t\frac{1}{2}\Lambda^*}(x,y) \geq 2^{-1} e^{-\frac{t}{2}A}(x,y), \quad |x| \geq r\bigg(\frac{t}{2}\bigg)^\frac{1}{\alpha}, \quad |y| \geq r\bigg(\frac{t}{2}\bigg)^\frac{1}{\alpha}.
\end{equation}
We now extend \eqref{e11}, by proving existence of a constant $0<c_1<2^{-1}$ such that
\begin{equation}
\label{e3}
\tag{$\ref{e11}'$}
e^{-t\Lambda^*}(x,y) \geq c_1 e^{-tA}(x,y), \quad |x| \geq r\bigg(\frac{t}{2}\bigg)^\frac{1}{\alpha}, \quad |y| \geq r\bigg(\frac{t}{2}\bigg)^\frac{1}{\alpha}.
\end{equation}
Clearly, we need to consider only the case $rt^\frac{1}{\alpha} \geq |x| \geq r\bigg(\frac{t}{2}\bigg)^\frac{1}{\alpha}$, $r \geq |y| \geq r\bigg(\frac{t}{2}\bigg)^\frac{1}{\alpha}$.
By the reproduction property,
\begin{align*}
e^{-t\Lambda^*}(x,y) &\geq \langle e^{-\frac{1}{2}t\Lambda^*}(x,\cdot)\mathbf{1}_{B^c\big(0,r\big(\frac{t}{2}\big)^\frac{1}{\alpha}\big)}(\cdot)e^{-\frac{1}{2}t\Lambda^*}(\cdot,y)\rangle \\
& (\text{we are applying \eqref{e2}}) \\
& \geq 2^{-2}\langle  e^{-\frac{1}{2}tA}(x,\cdot)\mathbf{1}_{B^c\big(0,r\big(\frac{t}{2}\big)^\frac{1}{\alpha}\big)}(\cdot)e^{-\frac{1}{2}tA}(\cdot,y)\rangle \\
& > 2^{-2}\langle  e^{-\frac{1}{2}tA}(x,\cdot)\mathbf{1}_{B\big(0,(r+1)\big(\frac{t}{2}\big)^\frac{1}{\alpha}\big) - B\big(0,r\big(\frac{t}{2}\big)^\frac{1}{\alpha}\big)}(\cdot)e^{-\frac{1}{2}tA}(\cdot,y)\rangle \\
& (\text{we are using the lower bound in \eqref{st_bd}}) \\
&\geq 2^{-2} \tilde{c}t^{-\frac{d}{\alpha}} \qquad (\tilde{c}=\tilde{c}(r)>0) \\
& (\text{we are using the upper bound in \eqref{st_bd}}) \\
&\geq c_1 e^{-tA}(x,y) \qquad \text{ for appropriate } 0<c_1=c_1(r)<2^{-1},
\end{align*}
i.e.\,we have proved \eqref{e3}.

The same argument yields
\begin{equation}
\label{e4}
\tag{$\ref{e2}'$}
e^{-\frac{1}{2}t\Lambda^*}(x,y) \geq c_1 e^{-\frac{1}{2}tA}(x,y), \quad |x| \geq r\bigg(\frac{t}{2^2}\bigg)^\frac{1}{\alpha}, \quad |y| \geq r\bigg(\frac{t}{2^2}\bigg)^\frac{1}{\alpha}.
\end{equation}
Thus, we can repeat the above procedure $m-1$ times obtaining
$$
e^{-t\Lambda^*}(x,y) \geq c_m e^{-tA}(x,y), \quad |x| \geq r\bigg(\frac{t}{2^m}\bigg)^\frac{1}{\alpha}, \quad |y| \geq r\bigg(\frac{t}{2^m}\bigg)^\frac{1}{\alpha}
$$
for appropriate $c_m>0$, from which the assertion of Corollary \ref{cor11} follows.
\end{proof}

\textbf{2.~}Next, in Proposition \ref{claim1_lb} we will prove an ``integral lower bound''. We need

\begin{lemma} 
\label{lem0}
For every $0 \leq h \in L^1$, $t>0$
\[
t^{-1}\int_0^t \|\psi_\tau h\|_1 d\tau \leq \hat{C}\| \psi_t h\|_1
\]
for a constant $\hat{C}=\hat{C}(\alpha,\beta)$.
\end{lemma}
\begin{proof}
Define $\psi_{0,t}(y)=\eta_0(t^{-\frac{1}{\alpha}}|y|)$, where
$$
\eta_0(u)=\left\{
\begin{array}{ll}
u^\beta, & 0<u<1, \\
1, & u \geq 1.
\end{array}
\right.
$$
Since $c^{-1}\psi_t \leq \psi_{0,t} \leq c \psi_t$, $c>1$, it suffices to prove Lemma \ref{lem0} for weight $\psi_{0,t}$. 

For brevity, write $\psi_t:=\psi_{0,t}$.
We have
$$
\|\psi_\tau h\|_1  =\langle \mathbf{1}_{B(0,\tau^{\frac{1}{\alpha}})}(\tau^{-\frac{1}{\alpha}}|x|)^\beta h \rangle + \langle \mathbf{1}_{B^c(0,\tau^{\frac{1}{\alpha}})}h\rangle,
$$
and so
$$
\int_0^t\|\psi_\tau h\|_1 d\tau =\langle \bigg(\int_0^t \mathbf{1}_{B(0,\tau^{\frac{1}{\alpha}})}\tau^{-\frac{\beta}{\alpha}}d\tau\bigg)  |x|^\beta h \rangle + \langle \bigg( \int_0^t \mathbf{1}_{B^c(0,\tau^{\frac{1}{\alpha}})}d\tau\bigg) h\rangle.
$$

If $|x|\leq t^{\frac{1}{\alpha}}$, then
$$
\int_0^t \mathbf{1}_{B(0,\tau^{\frac{1}{\alpha}})}(x)\tau^{-\frac{\beta}{\alpha}}d\tau=\int_{|x|^\alpha}^t \tau^{-\frac{\beta}{\alpha}} d\tau=\frac{1}{1-\frac{\beta}{\alpha}}(t^{-\frac{\beta}{\alpha}+1}-|x|^{-\beta+\alpha})
$$
and
$$
\int_0^t \mathbf{1}_{B^c(0,\tau^{\frac{1}{\alpha}})}(x)d\tau=\int_0^{|x|^\alpha}d\tau=|x|^\alpha.
$$
If $|x| > t^{\frac{1}{\alpha}}$, then
$$
\int_0^t \mathbf{1}_{B(0,\tau^{\frac{1}{\alpha}})}(x)\tau^{-\frac{\beta}{\alpha}}d\tau=0, \qquad \int_0^t \mathbf{1}_{B^c(0,\tau^{\frac{1}{\alpha}})}(x)d\tau=t.
$$

Thus,
\begin{align*}
\int_0^t\|\psi_\tau h\|_1 d\tau=&\langle \mathbf{1}_{B(0,t^{\frac{1}{\alpha}})}\frac{\alpha}{\alpha-\beta}(t^{-\frac{\beta}{\alpha}+1}-|x|^{-\beta+\alpha})|x|^\beta h\rangle + \langle \mathbf{1}_{B(0,t^{\frac{1}{\alpha}})}|x|^\alpha h\rangle + t\langle  \mathbf{1}_{B^c(0,t^{\frac{1}{\alpha}})} h\rangle\\
=&t\frac{\alpha}{\alpha-\beta}\langle \mathbf{1}_{B(0,t^{\frac{1}{\alpha}})}\psi_th\rangle-\frac{\beta}{\alpha-\beta}\langle \mathbf{1}_{B(0,t^{\frac{1}{\alpha}})}|x|^\alpha h\rangle + t\langle  \mathbf{1}_{B^c(0,t^{\frac{1}{\alpha}})}\psi_t h\rangle\\
\leq&t\frac{2\alpha-\beta}{\alpha-\beta}\langle\psi_th\rangle.
\end{align*}

\end{proof}

\begin{proposition}
\label{claim1_lb}
Define $g_t=\psi_th$, $0\leq h\in \mathcal S$-the L.\,Schwartz space of test functions. Then, there exists generic constant $\nu>0$ such that, for all $t>0$,
\[
\langle\psi_t e^{-t\Lambda}\psi_t^{-1} g_t\rangle \geq \nu\langle g_t\rangle.
\]
\end{proposition}

\begin{proof}
Recall that both $e^{-t\Lambda^\varepsilon}$, $e^{-t(\Lambda^\varepsilon)^*}$ are holomorphic in $L^1$ and $C_u$ due to Hille's Perturbation Theorem. We have $\psi=\psi_{(1)} + \psi_{(u)}$, where 
$$
\psi_{(1)} \in D((-\Delta)^{\frac{\alpha}{2}}_1)\;\big(=D((\Lambda^\varepsilon)^*_1)=D(\Lambda^\varepsilon_1)\big), $$
$$
\psi_{(u)} \in D((-\Delta)^{\frac{\alpha}{2}}_{C_{u}})\;\big(=D((\Lambda^\varepsilon)^*_{C_u})=D(\Lambda^\varepsilon_{C_u})\big)$$
(see the proof of Proposition \ref{prop2} for details), so $(\Lambda^\varepsilon)^*\psi\;\;\big(=\Lambda^\varepsilon)^*_{L^1}\psi_{(1)} + (\Lambda^\varepsilon)^*_{C_u}\psi_{(u)}\big)$ and belongs to $\in L^1 + C_u$.

Now, set $g_{s,n}=\phi_{s,n} h$, $\phi_{s,n}(x)=(e^{-\frac{(\Lambda^\varepsilon)^*}{n}}\psi_s)(x)$. We have, for $s>t>0$,
\begin{align*}
\langle g_{s,n}\rangle-\langle \phi_{s,n} e^{-t\Lambda^\varepsilon}h\rangle & =\int_0^t\langle\psi_s, \Lambda^\varepsilon e^{-\tau \Lambda ^\varepsilon}e^{-\frac{\Lambda^\varepsilon}{n}}h\rangle d\tau \\
&=\lim_{r \downarrow 0}r^{-1}\int_0^t\langle\psi_s, (1-e^{-r\Lambda^\varepsilon}) e^{-\tau \Lambda ^\varepsilon}e^{-\frac{\Lambda^\varepsilon}{n}}h\rangle d\tau \\
& =\lim_{r \downarrow 0}r^{-1}\int_0^t\langle (1-e^{-r(\Lambda^\varepsilon)^*})\psi_s,  e^{-\tau \Lambda ^\varepsilon}e^{-\frac{\Lambda^\varepsilon}{n}}h\rangle d\tau \\
& = \int_0^t\langle (\Lambda^\varepsilon)^*\psi_s,  e^{-\tau \Lambda ^\varepsilon}e^{-\frac{\Lambda^\varepsilon}{n}}h\rangle d\tau.
\end{align*}
%Set $\tilde{\psi_s}(x):=(s^{-\frac{1}{\alpha}}|x|)^{\beta}$.
Arguing as in the proof of Proposition \ref{prop2}, we represent $$(\Lambda^\varepsilon)^*\psi_s
=  \mathbf{1}_{B(0,s^\frac{1}{\alpha})}W_\varepsilon\psi_s + v_\varepsilon,$$ where
$
W_{\varepsilon}(x)=\kappa (|x|_\varepsilon^{-\alpha}-|x|^{-\alpha})\beta  + \kappa\big[d|x|_\varepsilon^{-\alpha}-\alpha |x|_\varepsilon^{-\alpha-2}|x|^2 - (d-\alpha)|x|^{-\alpha} \big]
$
and $0 \leq v_\varepsilon \in L^\infty$, $\|v_\varepsilon\|_\infty \leq \frac{c'}{s}$, $c' \neq c'(\varepsilon)$ (see Remark \ref{repr_details} below for detailed calculation).

Then
\begin{align*}
\langle g_{s,n}\rangle-\langle \phi_{s,n} e^{-t \Lambda^\varepsilon}h\rangle  \leq \int_0^t\langle \mathbf{1}_{B(0,s^\frac{1}{\alpha})}W_\varepsilon\psi_s,e^{-(\tau+\frac{1}{n})\Lambda^\varepsilon}h\rangle d\tau  + \int_0^t \langle v_\varepsilon, e^{-\tau \Lambda^\varepsilon}e^{-\frac{\Lambda^\varepsilon}{n}}h \rangle d\tau
\end{align*}
or, sending $n\to\infty$,
\begin{align*}
\langle g_s\rangle-\langle \psi_s e^{-t \Lambda^\varepsilon}h\rangle & \leq \int_0^t\langle \mathbf{1}_{B(0,s^\frac{1}{\alpha})}W_\varepsilon \psi_s,e^{-\tau\Lambda^\varepsilon}h\rangle d\tau  + \int_0^t \langle v_\varepsilon, e^{-\tau \Lambda^\varepsilon}h \rangle d\tau \\
&\leq \int_0^t\langle \mathbf{1}_{B(0,s^\frac{1}{\alpha})}W_\varepsilon\psi_s,e^{-\tau\Lambda^\varepsilon}h\rangle d\tau  + c's^{-1}\int_0^t \|e^{-\tau \Lambda^\varepsilon}h\|_1 d\tau.
\end{align*}
Next, we pass to the limit $\varepsilon \downarrow 0$:
\[
\langle g_s\rangle-\langle \psi_s e^{-t \Lambda}h\rangle \leq  c's^{-1}\int_0^t  \|e^{-\tau \Lambda}h\|_1 d\tau.\tag{$\star$}
\]

\medskip

We estimate the RHS of $(\star)$ using the upper bound:
\begin{align*}
c's^{-1}\int_0^t \|e^{-\tau \Lambda}h\|_1 d\tau & \leq c's^{-1}C\int_0^t\|e^{-\tau A}\psi_\tau h\|_1 d\tau \leq c's^{-1}C\int_0^t\|\psi_\tau h\|_1 d\tau \\
& (\text{we are applying Lemma \ref{lem0}}) \\
& \leq c'C\hat{C}\frac{t}{s} \| \psi_t h\|_1,
\end{align*}
Therefore, using $\psi_s\geq \big(\frac{t}{s}\big)^\frac{\beta}{\alpha}\psi_t$,
we obtain
$$
c's^{-1}\int_0^t \|e^{-\tau \Lambda}h\|_1 d\tau \leq c'C\hat{C}\frac{t}{s}\bigg(\frac{t}{s}\bigg)^{-\frac{\beta}{\alpha}} \| g_s\|_1.
$$
Thus, by $(\star)$,
$
\big(1-c'C\hat{C}\left(\frac{t}{s}\right)^\frac{\alpha-\beta}{\alpha}\big)\langle g_s\rangle \leq \langle \psi_s e^{-t \Lambda}h\rangle.
$
Since $\beta<\alpha$, we can  select $s>t$ such that $c'C\hat{C}\left(\frac{t}{s}\right)^\frac{\alpha-\beta}{\alpha} = \frac{1}{2}$, which yields the bound
\[
\langle\psi_s e^{-t\Lambda}\psi_s^{-1} g_s\rangle \geq \frac{1}{2}\langle g_s\rangle.
\]

Finally, using $\psi_t\geq \psi_s\geq \big(\frac{t}{s}\big)^\frac{\beta}{\alpha}\psi_t$ and setting $2\nu:=\big(\frac{t}{s}\big)^\frac{\beta}{\alpha}=\big(2c'C\hat{C}\big)^{-\frac{\beta}{\alpha-\beta}}$, we have
\[
\langle\psi_t e^{-t\Lambda}\psi_t^{-1} g_t\rangle=\langle\psi_t e^{-t\Lambda}\psi_s^{-1} g_s\rangle\geq\langle\psi_s e^{-t\Lambda}\psi_s^{-1} g_s\rangle\geq \frac{1}{2}\langle g_s\rangle\geq\frac{1}{2}\bigg(\frac{t}{s}\bigg)^\frac{\beta}{\alpha}\langle g_t\rangle=\nu\langle g_t\rangle.
\]
\end{proof}

\begin{remark}
\label{repr_details}
In the proof of Proposition \ref{claim1_lb}, we calculate $(\Lambda^\varepsilon)^*\psi_s$ arguing as in the proof of Proposition \ref{prop2}:
$$
(\Lambda^\varepsilon)^*\psi=(-\Delta)^{\frac{\alpha}{2}}\psi + {\rm div\,}  (b_\varepsilon \psi), \quad \psi=\psi_s,
$$
where
\begin{align*}
(-\Delta)^{\frac{\alpha}{2}}\psi  = -s^{-\frac{\beta}{\alpha}}\beta(d+\beta-2) \frac{\gamma(d+\beta-2)}{\gamma(d+\beta-\alpha)}|x|^{\beta-\alpha} + h_0
\end{align*}
for $h_0 :=- I_{2-\alpha}\Delta (\psi-\tilde{\psi}) \in L^\infty$, $\|h_0\|_\infty \leq c_0 s^{-1}$.
In turn,
$$
{\rm div\,}  (b_\varepsilon \psi) = {\rm div\,}  (b\tilde{\psi}) + W_\varepsilon +  h_1 + h_2 + h_3
$$
where $\|h_i\|_\infty \leq c_i s^{-1}$, $i=1,2,3$.
Since, by the choice of $\beta$, $-\beta(d+\beta-2) \frac{\gamma(d+\beta-2)}{\gamma(d+\beta-\alpha)}|x|^{-\alpha}\tilde{\psi} + {\rm div\,}  (b\tilde{\psi})=0$, we have
$$
(\Lambda^\varepsilon)^*\psi=\mathbf{1}_{B(0,s^{\frac{1}{\alpha}})}W_\varepsilon+v_\varepsilon, \quad v_\varepsilon:=\mathbf{1}_{B^c(0,s^{\frac{1}{\alpha}})}W_\varepsilon + h_0+h_1+h_2+h_3,
$$
where, it easily seen, $\|v_\varepsilon\|_\infty \leq c's^{-1}$, as claimed.

\end{remark}

\begin{proposition}
\label{ANcorol2} 
For every $R_0>0$ there exist constants $0<r<R_0<R$ such that for all $t>0$
\begin{equation*}
\frac{\nu}{2}\psi_t(x)\leq e^{-t\Lambda^*}\psi_t \mathbf 1_{R_t,r_t}(x) \;\; \text{ for all }x \in B(0,R_{0,t}), \quad x \neq 0.
\end{equation*} 
where
$
r_t:=rt^{\frac{1}{\alpha}}$, $R_{0,t}:=R_0t^{\frac{1}{\alpha}}$, $R_t:=Rt^{\frac{1}{\alpha}}$, 
$\mathbf 1_{R_t,r_t}:=\mathbf 1_{B(0,R_t)}-\mathbf 1_{B(0,r_t)}.
$
\end{proposition}

\begin{proof}
It suffices to prove 
that, for all $g:=\psi_t h$, $0\leq h\in \mathcal S$ with $\sprt h\subset B(0,R_{0,t})$,
\[
\frac{\nu}{2}\langle g\rangle\leq\langle \mathbf 1_{R_t,r_t}\psi_t e^{-t\Lambda}\psi_t^{-1}g\rangle.
\]

By the upper bound,
\begin{align*}
\langle\mathbf 1_{B(0,r_t)}\psi_t e^{-t\Lambda}\psi_t^{-1}g\rangle &\leq C\langle\mathbf 1_{B(0,r_t)}\psi_t,e^{-tA}g\rangle\\
&\leq CC_1t^{-\frac{d}{\alpha}}\|\mathbf 1_{B(0,r_t)}\psi_t\|_1\|g\|_1\\
&= CC_1\|\mathbf 1_{B(0,r)}\psi_1\|_1\|g\|_1, \quad \|\mathbf 1_{B(0,r)}\psi_1\|_1 \rightarrow 0 \text{ as } r \downarrow 0.
\end{align*}
\begin{align*}
\langle\mathbf 1_{B^c(0,R_t)}\psi_t e^{-t\Lambda}\psi_t^{-1}g\rangle &\leq C\langle\mathbf 1_{B^c(0,R_t)}\psi_t,e^{-tA}g\rangle\\
&\leq C\langle e^{-tA}\mathbf 1_{B^c(0,R_t)},g\mathbf 1_{B(0,R_{0,t})}\rangle \\
&\leq 2C\sup_{x\in B(0,R_{0,t})}e^{-tA}\mathbf 1_{B^c(0,R_t)}(x)\|g\|_1 \\
&\leq C(R_0,R)\|g\|_1, \quad C(R_0,R)\rightarrow 0 \text{ as } R-R_0 \uparrow \infty
\end{align*}
where at the last step we have used, for $x \in B(0,R_{0,t})$, $y \in B^c(0,R_{t})$ and $\tilde{x}=R_0^{-1}t^{-\frac{1}{\alpha}}x \in B(0,1)$, $\tilde{y}=R^{-1}t^{-\frac{1}{\alpha}}y \in B^c(0,1)$, 
\begin{align*}
e^{-tA}(x,y) & \leq k_0 t |x-y|^{-d-\alpha} \leq k_0 t |R_0 t^{\frac{1}{\alpha}}\tilde{x}-R t^{\frac{1}{\alpha}}\tilde{y}|^{-d-\alpha} < 2k_0 t^{-\frac{d}{\alpha}}(R-R_0)^{-d-\alpha}|\tilde{y}|^{-d-\alpha}.
\end{align*}
It remains to apply Proposition \ref{claim1_lb} to obtain $\frac{\nu}{2}\langle g\rangle\leq\langle \mathbf 1_{R_t,r_t}\psi_t e^{-t\Lambda}\psi_t^{-1}g\rangle$.
\end{proof}

\begin{proposition}
\label{lem4_lb}
$\langle h \rangle = \langle e^{-t\Lambda^*}h \rangle$ for every $h \in L^1$, $t>0$.
\end{proposition}
\begin{proof}
Proposition \ref{lem4_lb} follows from $\langle h \rangle = \langle e^{-t(\Lambda^\varepsilon)^*}h \rangle$ and Proposition \ref{constr_d}.
\end{proof}

\begin{proposition}
\label{claim4_lb}
For every $R_0>0$ there exist constants $0<r<R_0<R$ such that for all $t>0$ 
$$
\frac{1}{2} \leq e^{-t\Lambda}\mathbf{1}_{R_t,r_t}(x) \quad \text{ for all $x \in B(0,R_{0,t})$},
$$
where
$
r_t:=rt^{\frac{1}{\alpha}}$, $R_{0,t}:=R_0t^{\frac{1}{\alpha}}$, $R_t:=Rt^{\frac{1}{\alpha}}$, 
$\mathbf 1_{R_t,r_t}:=\mathbf 1_{B(0,R_t)}-\mathbf 1_{B(0,r_t)}.
$

\end{proposition}

\begin{proof}
We essentially repeat the proof of Proposition \ref{ANcorol2}. 
It suffices to prove that, for all $0\leq h\in \mathcal S$ with $\sprt h\subset B(0,R_{0,t})$,
\[
\frac{1}{2}\langle h \rangle \leq \langle \mathbf{1}_{R_t,r_t} e^{-t\Lambda^*}h\rangle.
\]
By the upper bound,
\begin{align*}
\langle\mathbf 1_{B(0,r_t)} e^{-t\Lambda^*} h\rangle &\leq C\langle\mathbf 1_{B(0,r_t)}\psi_t,e^{-tA}h\rangle\\
&\leq CC_1t^{-\frac{d}{\alpha}}\|\mathbf 1_{B(0,r_t)}\psi_t\|_1\|h\|_1\\
&= o(r)\|h\|_1, \quad o(r)\rightarrow 0 \text{ as } r\downarrow 0;
\end{align*}
\begin{align*}
\langle\mathbf 1_{B^c(0,R_t)}e^{-t\Lambda^*} h\rangle &\leq C\langle\mathbf 1_{B^c(0,R_t)}\psi_t,e^{-tA}h\rangle\\
&\leq C\langle e^{-tA}\mathbf 1_{B^c(0,R_t)},h \mathbf 1_{B(0,R_{0,t})}\rangle \\
&\leq C\sup_{x\in B(0,R_{0,t})}e^{-tA}\mathbf 1_{B^c(0,R_t)}(x)\|h\|_1\\
&= C(R_0,R)\|h\|_1, \quad C(R_0,R)\rightarrow 0 \text{ as } R-R_0 \uparrow \infty.
\end{align*}

The last two estimates and Proposition \ref{lem4_lb} yield $\frac{1}{2}\langle h \rangle \leq \langle \mathbf{1}_{R_t,r_t} e^{-t\Lambda^*}h\rangle$.  
\end{proof}

\bigskip

\textbf{3.~}We are in position to complete the proof of the lower bound using the so-called $3q$ argument.

Set $q_t(x,y):=\psi_t^{-1}(x)e^{-t\Lambda^*}(x,y)$, $x \neq 0$. 

(a)~Let $x,y \in B^c(0,t^{\frac{1}{\alpha}})$, $x \neq y$. Then, using that $\psi_{3t}^{-1}\geq 1$, we have by Corollary \ref{cor11},
$$
q_{3t}(x,y) \geq e^{-{3t}\Lambda^*}(x,y) \geq ce^{-3tA}(x,y).
$$

Let $r_t=rt^{\frac{1}{\alpha}}$, $R_t=Rt^{\frac{1}{\alpha}}$ be as in Proposition \ref{ANcorol2} and Proposition \ref{claim4_lb}, where we fix $R_0=1$ (hence $r<1$).

\smallskip

(b)~Let $x \in B(0,t^{\frac{1}{\alpha}})$, $|y| \geq rt^{\frac{1}{\alpha}}$, $x \neq y$. By the reproduction property,
\begin{align*}
q_{2t}(x,y) &\geq \psi_{2t}^{-1}(x)\langle e^{-t\Lambda^*}(x,\cdot) \psi_t^{-1}(\cdot)\psi_t(\cdot) e^{-t\Lambda^*}(\cdot,y)\mathbf{1}_{R_t,r_t}(\cdot) \rangle \notag \\
&\geq \psi_{2t}^{-1}(x)\psi_t^{-1}(R_t) \langle e^{-t\Lambda^*}(x,\cdot)\psi_t(\cdot) e^{-t\Lambda^*}(\cdot,y)\mathbf{1}_{R_t,r_t}(\cdot) \rangle \notag \\
&\geq\psi_{2t}^{-1}(x)\psi_t^{-1}(R_t)\bigl(e^{-t\Lambda^*}\psi_t\mathbf{1}_{R_t,r_t}\bigr)(x)\inf_{r_t\leq |z|\leq R_t} e^{-t\Lambda^*}(z,y) \notag\\
& \text{(we are applying Corollary \ref{cor11}, Proposition \ref{ANcorol2} and using $\psi_t^{-1}(R_t)=1$)}\\
&\geq \frac{\nu}{2}\psi_{2t}^{-1}(x)\psi_t(x)c(r) \inf_{r_t\leq |z|\leq R_t}e^{-tA}(z,y) \\
& (\text{we are using $\psi_t \geq \psi_{2t}$}) \\
&\geq C_1 e^{-2tA}(x,y). 
\end{align*}

(b') Let $x \in B(0,t^{\frac{1}{\alpha}}), |y| \geq t^{\frac{1}{\alpha}}$, $x \neq y$. Arguing as in (b), we obtain
$$
q_{3t}(x,y) \geq C_2 e^{-3tA}(x,y).
$$

(c)~Let $|x| \geq rt^{\frac{1}{\alpha}}$, $y \in B(0,t^{\frac{1}{\alpha}})$, $x \neq y$. We have
\begin{align*}
q_{2t}(x,y) &\geq \psi_{2t}^{-1}(x) \langle e^{-t\Lambda^*}(x,\cdot)e^{-t\Lambda^*}(\cdot,y)\mathbf{1}_{R_t,r_t}(\cdot) \rangle \notag \\ 
&= \psi_{2t}^{-1}(x)\langle  e^{-t\Lambda^*}(x,\cdot) e^{-t\Lambda}(y,\cdot)\mathbf{1}_{R_t,r_t}(\cdot)\rangle \notag \\ 
& (\text{we are using $\psi_{2t}^{-1} \geq 1$ and applying Corollary \ref{cor11}}) \notag \\
&\geq c(r)\langle e^{-tA}(x,\cdot) e^{-t\Lambda}(y,\cdot)\mathbf{1}_{R_t,r_t}(\cdot) \rangle \notag \\ 
& (\text{we are applying \eqref{st_bd}}) \\
& \geq C_3(r)t(Rt^{\frac{1}{\alpha}}+|x|)^{-d-\alpha}\langle e^{-\Lambda}(y,\cdot)\mathbf{1}_{R_t,r_t}(\cdot)\rangle \notag \\ 
& \text{(we are applying Proposition \ref{claim4_lb})} \notag \\
& \geq C_3(r)2^{-1}t(Rt^{\frac{1}{\alpha}}+|x|)^{-d-\alpha} \geq C_4(r) e^{-2tA}(x,y). \notag
\end{align*}

(c') Let $|x| \geq t^{\frac{1}{\alpha}}$, $y \in B(0,t^{\frac{1}{\alpha}})$, $x \neq y$. Arguing as in (c), we obtain
$$
q_{3t}(x,y) \geq C_5(r) e^{-3tA}(x,y).
$$

(d)~Let $x,y \in B(0,t^{\frac{1}{\alpha}})$, $x \neq y$. By the reproduction property,
\begin{align*}
q_{3t}(x,y) & \geq \psi_{3t}^{-1}(x)\langle e^{-t\Lambda^*}(x,\cdot)e^{-2t\Lambda^*}(\cdot,y)\mathbf{1}_{R_t,r_t}(\cdot)\rangle \\
& (\text{we are using (c)}) \\
& \geq C_4(r)\psi_{3t}^{-1}(x)\langle e^{-t\Lambda^*}(x,\cdot)\psi_{2t}(\cdot)e^{-2tA}(\cdot,y)\mathbf{1}_{R_t,r_t}(\cdot) \rangle\\
& (\text{we are using $\psi_{2t} \geq 2^\frac{\beta}{\alpha} \psi_t$ and $e^{-2tA}(z,y) \geq c(r,R)t^{-\frac{d}{\alpha}}>0$ for $r_t \leq |z| \leq R_t$, $|y| \leq t^{\frac{1}{\alpha}}$}) \\
& \geq c(r,R)C_4 2^\frac{\beta}{\alpha}\psi_{3t}^{-1}(x)t^{-\frac{d}{\alpha}} \langle e^{-t\Lambda^*}(x,\cdot)\mathbf{1}_{R_t,r_t}(\cdot) \psi_t(\cdot)\rangle \\
&\text{(we are applying Proposition \ref{ANcorol2} and using $\psi_{t} \geq \psi_{3t}$}) \\
& \geq c(r,R)C_4 2^\frac{\beta}{\alpha} \frac{\nu}{2}t^{-\frac{d}{\alpha}} \\
& (\text{we are applying \eqref{st_bd}}) \\
& \geq C_5(r,R) e^{-3tA}(x,y).
\end{align*}

By (a), (b'), (c'), (d),
$
q_{3t}(x,y) \geq C e^{-3tA}(x,y)$ for all $x,y \in \mathbb R^d$, $x \neq y$, $x \neq 0$,
and so
$$
e^{-3t\Lambda^*}(x,y) \geq C e^{-3tA}(x,y)\psi_{3t}(x), \quad  t>0.$$ The lower bound is proved.

\bigskip

\section{Construction of the semigroup $e^{-t\Lambda_r}$, $\Lambda_r = (-\Delta)^{\frac{\alpha}{2}} - b \cdot \nabla$ in $L^r$, $1 \leq r<\infty$} 

\label{sect_d}

Set $b_\varepsilon(x):=\kappa|x|_\varepsilon^{-\alpha}x$, $\kappa>0$, $|x|_\varepsilon:=\sqrt{|x|^2+\varepsilon}$, $\varepsilon>0$, $$\Lambda_r^\varepsilon:=(-\Delta)^{\frac{\alpha}{2}} - b_\varepsilon \cdot \nabla, \quad D(\Lambda_r^\varepsilon)=\mathcal W^{\alpha,r} := \big(1+(-\Delta)^{\frac{\alpha}{2}}\big)^{-1}L^r.$$

To prove that $-\Lambda^\varepsilon \equiv -\Lambda_r^\varepsilon$ is the generator of a holomorphic semigroup in $L^r$, $1 \leq r<\infty$, we appeal to the Hille Perturbation Theorem \cite[Ch.\,IX, sect.\,2.2]{Ka}. To verify its assumptions, we use a well known estimate
$$
|\nabla \big(\zeta + A\big)^{-1}(x,y)| \leq C \bigl(\Real \zeta + A\bigr)^{-\frac{\alpha-1}{\alpha}}(x,y), \quad \Real \zeta>0, \quad C=C(d,\alpha),\quad A \equiv (-\Delta)^{\frac{\alpha}{2}}.
$$
Then for $Y=L^p$ $$\|b_\varepsilon \cdot \nabla \big(\zeta + A\big)^{-1}\|_{Y \rightarrow Y} \leq C\|b_\varepsilon\|_\infty  \|\big(\Real \zeta + A\big)^{-\frac{\alpha-1}{\alpha}})\|_{Y \rightarrow Y} \leq C\|b_\varepsilon\|_\infty (\Real \zeta)^{-\frac{\alpha-1}{\alpha}},$$
and so $\|b_\varepsilon \cdot \nabla \big(\zeta + A\big)^{-1}\|_{Y \rightarrow Y}$, $\Real \zeta \geq c_\varepsilon$, can be made arbitrarily small by selecting $c_\varepsilon$ sufficiently large.
It follows that the Neumann series for $$(\zeta + \Lambda^\varepsilon)^{-1}=(\zeta + A)^{-1}(1+T)^{-1}, \quad T:= - b_\varepsilon \cdot \nabla (\zeta + A)^{-1},$$converges in $L^p$ and $C_u$ and satisfies $\|(\zeta + \Lambda^\varepsilon)^{-1}\|_{Y \rightarrow Y} \leq C_\varepsilon|\zeta|^{-1}$, $\Real \zeta \geq c_\varepsilon$, i.e.\,$-\Lambda^\varepsilon$ is the generator of a holomorphic semigroup.

The same argument (with $Y=C_u$) shows that $\Lambda^\varepsilon:=(-\Delta)^{\frac{\alpha}{2}} - b_\varepsilon \cdot \nabla$ with $D(\Lambda^\varepsilon):= D((-\Delta)^{\frac{\alpha}{2}}_{C_{u}})$ generates a holomorphic semigroup in $C_u$.

\begin{proposition}
\label{prop_contr}
For every $r \in [1,\infty[$ and $\varepsilon>0$, $e^{-t\Lambda_r^{\varepsilon}}$ is a contraction $C_0$ semigroup in $L^r$. There exists a constant $c \neq c(\varepsilon)$ such that
$$
\|e^{-t\Lambda_r^{\varepsilon}}\|_{r \rightarrow q} \leq c_N t^{-\frac{d}{\alpha}(\frac{1}{r}-\frac{1}{q})}, \quad t>0,
$$
for all $1 \leq r < q \leq \infty$.

In particular, there is a constant $c_S>0$, $c_S\neq c_S(\varepsilon)$ such that {\rm($\Lambda^\varepsilon \equiv \Lambda^\varepsilon_2$)}
\[
\Real \langle \Lambda^\varepsilon u,u\rangle\geq c_S\|u\|_{2j}^2,\quad u\in D(\Lambda^\varepsilon).
\]
\end{proposition}

\begin{proof}
First, let $1<r<\infty$. Set $u\equiv u(t):=e^{-t\Lambda_r^\varepsilon}f$, $f \in L^1\cap L^\infty$, and write $A:=(-\Delta)^{\frac{\alpha}{2}}$. Multiplying the equation $\partial_t u + \Lambda_r^\varepsilon u=0$ by $\bar{u}|u|^{r-2}$ and integrating over the spatial variables we obtain (taking into account that $D(\Lambda^\varepsilon_r)=D(A_r)\subset W^{1,r}$)
\[
\frac{1}{r}\partial_t\|u\|_r^r + \Real\langle Au,u|u|^{r-2}\rangle - \Real \langle b_\varepsilon \cdot \nabla u,u|u|^{r-2}\rangle=0.
\]
Note that, since $-A$ is a Markov generator, $$\Real\langle Au,u|u|^{r-2}\rangle \geq \frac{4}{rr'}\|A^{\frac{1}{2}}|u|^{\frac{r}{2}}\|_2^2$$  (indeed, by \cite[Theorem 2.1]{LS} or by Theorem \ref{thm_M} in Appendix \ref{appendix_A}, $\Real\langle Au,u|u|^{r-2}\rangle \geq \frac{4}{rr'}\|A^{\frac{1}{2}}u^{\frac{r}{2}}\|_2^2$, $u^{\frac{r}{2}}:=u |u|^{\frac{r}{2}-1}$, and by the Beurling-Deny theory $\|A^{\frac{1}{2}}u^{\frac{r}{2}}\|_2^2 \geq \|A^{\frac{1}{2}}|u|^{\frac{r}{2}}\|_2^2$). Integration by parts yields 
$$
-\Real \langle b_\varepsilon \cdot \nabla u,u|u|^{r-2}\rangle = \frac{\kappa}{r}\big\langle \big(d|x|_\varepsilon^{-\alpha}-\alpha |x|_\varepsilon^{-\alpha-2}|x|^2\big)|u|^r\big\rangle \geq \kappa\frac{d-\alpha}{r} \langle |x|_\varepsilon^{-\alpha}|u|^r\rangle. 
$$
Thus,
\begin{equation}
\label{i_5}
-\partial_t\|u\|_r^r\geq \frac{4}{r'}\|A^{\frac{1}{2}}|u|^{\frac{r}{2}}\|_2^2
\end{equation}
 From \eqref{i_5} we obtain $\|u(t)\|_r \leq \|f\|_r$, $t \geq 0$ and since $L^1\cap L^\infty$ is dense in $L^r$, $\|e^{-t\Lambda_r^{\varepsilon}}\|_{r \rightarrow r} \leq 1$ as needed. 

Since $e^{-t\Lambda^\varepsilon_1} \upharpoonright L^1\cap L^r = e^{-t\Lambda^\varepsilon_r} \upharpoonright L^1\cap L^r$, the latter clearly yields 
$$
\|e^{-t\Lambda_1^{\varepsilon}}f\|_{r} \leq \|f\|_r, \quad f \in L^1\cap L^\infty.
$$
Sending $r \uparrow \infty$, we have $\|e^{-t\Lambda_r^\varepsilon}f\|_\infty \leq \|f\|_\infty$, and sending $r \downarrow 1$, we have $\|e^{-t\Lambda_1^{\varepsilon}}\|_{1 \rightarrow 1} \leq 1$.

Let us prove the ultracontractivity of $e^{-t\Lambda_r^\varepsilon}$. By \eqref{i_5}, 
\[
-\partial_t\|u\|_{2r}^{2r} \geq \frac{4}{(2r)^\prime}\|A^{\frac{1}{2}}|u|^r\|_2^2, \quad 1\leq r<\infty. 
\]
Using the Nash inequality $\|A^\frac{1}{2} h \|_2^2 \geq C_N \|h\|_2^{2 + \frac{2\alpha}{d}} \|h\|_1^{-\frac{2\alpha}{d}}$ and $\|u(t)\|_r \leq \|f\|_r$,  we have, setting $v := \|u\|_{2r}^{2r},$ 
\[
\partial_t v^{-\frac{\alpha}{d}} \geq c_1  \|f\|_r^{-\frac{2r\alpha}{d}},
\]
where $c_1 = C_N \frac{\alpha}{d} \frac{4}{(2r)^\prime}$.
Integrating this inequality yields
\begin{equation}
\label{b_}
\tag{$\ast$}
\|e^{-t \Lambda^\varepsilon_r} \|_{r \rightarrow 2r} \leq c_1^{-\frac{d}{2 \alpha r}}  t^{-\frac{d}{\alpha} (\frac{1}{r} - \frac{1}{2r})}, \quad t > 0,
\end{equation}
and so, by semigroup property,
\[
\|e^{-t \Lambda^\varepsilon_r} \|_{1 \rightarrow 2^m } \leq c_N t^{-\frac{d}{\alpha} (1 - \frac{1}{2^m })}, \quad t>0, \quad m \geq 1,
\]
where the constant $c_N\neq c_N(m)$. Thus, sending $m$ to infinity we arrive at $\|e^{-t \Lambda^\varepsilon_r} \|_{1 \rightarrow \infty} \leq c_N t^{-\frac{d}{\alpha}}, \; t>0$.
 The latter and the contractivity of $e^{-t\Lambda_r^\varepsilon}$ in all $L^q$, $1 \leq q \leq \infty$ yield via interpolation the desired bound
$\|e^{-t\Lambda_p^{\varepsilon}}\|_{p \rightarrow q} \leq c_N t^{-\frac{d}{\alpha}(\frac{1}{p}-\frac{1}{q})}
$, $t>0$, for all $1 \leq p<q \leq \infty$.

Finally, since $D(\Lambda^\varepsilon)=D(A)$, we have, for $u\in D(A)$, $\Real\langle \Lambda^\varepsilon u,u\rangle\geq \|A^\frac{1}{2}u\|_2^2\geq c_S\|u\|_{2j}^2$
\end{proof}

\subsection{Case $d \geq 4$} We will first provide an elementary argument that allows to treat all $d=4,5,\dots$ but the main case $d=3$.

\begin{proposition}
\label{constr_d0}
For every $r \in [1,\infty[$ the limit
$$
s\mbox{-}L^r\mbox{-}\lim_{\varepsilon \downarrow 0} e^{-t\Lambda_r^{\varepsilon}} \quad (\text{loc.\,uniformly in $t \geq 0$})
$$ 
exists and determines a contraction $C_0$ semigroup on $L^r$, say $e^{-t\Lambda_r}$.

For all $1 \leq r < q \leq \infty$,
$$
\|e^{-t\Lambda_r}\|_{r \rightarrow q} \leq c_N t^{-\frac{d}{\alpha}(\frac{1}{r}-\frac{1}{q})}, \quad t>0
$$
with $c_N$ from Proposition \ref{prop_contr}
\end{proposition}

\begin{proof}[Proof of Proposition \ref{constr_d0}]

First, let $r=2$. Set $u^\varepsilon(t):=e^{-t\Lambda^\varepsilon}f$, $f \in C_c^\infty$.

\begin{claim}
\label{claim1}
$\|\nabla u^\varepsilon(t)\|_2 \leq \|\nabla f\|_2$, $t \geq 0.$
\end{claim}

\begin{proof}[Proof of Claim \ref{claim1}]
Denote
$
u \equiv u^\varepsilon$, $w:=\nabla u$, $w_i:=\nabla_i u.
$
Due to $f\in C^\infty_c$ and $\nabla^n_i b^i_\varepsilon\in C^\infty\cap L^\infty$, $i=1,\dots d$, $n \geq 1$ we can and will differentiate the equation $\partial_t u + \Lambda^\varepsilon u=0$ in $x_i$, obtaining
$$
\partial_t w_i + (-\Delta)^{\frac{\alpha}{2}}w_i - b_\varepsilon \cdot \nabla w_i - (\nabla_i b_\varepsilon) \cdot w=0.
$$
Multiplying the latter by $\bar{w_i}$, integrating by parts and summing up in $i=1,\dots,d$ we have
$$
\frac{1}{2} \partial_t \|w\|_2^2 + \sum_{i=1}^d \|(-\Delta)^{\frac{\alpha}{4}}w_i\|_2^2 - \Real\sum_{i=1}^d\langle  b_\varepsilon \cdot \nabla w_i,w_i \rangle - \Real\sum_{i=1}^d \langle (\nabla_i b_\varepsilon) \cdot w,w_i \rangle =0,
$$

\begin{align*}
-\Real\langle b_\varepsilon \cdot \nabla w_i,w_i\rangle = \frac{\kappa}{2}\langle (d|x|_\varepsilon^{-\alpha}-\alpha |x|_\varepsilon^{-\alpha-2}|x|^2)w_i,w_i\rangle,
\end{align*}
$$
-\langle (\nabla_i b_\varepsilon) \cdot w,w_i \rangle=-\kappa \langle |x|_\varepsilon^{-\alpha}w_i,w_i\rangle +\kappa \alpha \langle |x|_\varepsilon^{-\alpha-2} x_i \bar{w}_i (x \cdot w) \rangle.
$$
Thus, 
\begin{align*}
\frac{1}{2} \partial_t \|w\|_2^2 + \sum_{i=1}^d \|(-\Delta)^{\frac{\alpha}{4}}w_i\|_2^2 & + \kappa \frac{ d-\alpha  }{2}\langle |x|_\varepsilon^{-\alpha}|w|^2\rangle +\frac{\kappa\alpha\varepsilon}{2}\langle |x|_\varepsilon^{-\alpha-2}|w|^2\rangle\\
& - \kappa \langle |x|_\varepsilon^{-\alpha}|w|^2\rangle + \kappa \alpha \langle |x|_\varepsilon^{-\alpha-2} |x \cdot w|^2 \rangle = 0,
\end{align*}
and so, since $\kappa>0$,
\[
\frac{1}{2} \partial_t \|w\|_2^2 + \sum_{i=1}^d \|(-\Delta)^{\frac{\alpha}{4}}w_i\|_2^2 + \kappa \frac{ d-\alpha -2 }{2}\langle |x|_\varepsilon^{-\alpha}|w|^2\rangle + \kappa \alpha \langle |x|_\varepsilon^{-\alpha-2} |x \cdot w|^2 \rangle \leq 0.
\]
Since $d \geq 4$, $\alpha<2$, we have $ d-\alpha-2>0$. Thus, integrating in $t$, we obtain
$
\|w(t)\|_2^2 \leq \|\nabla f\|_2^2$, $t \geq 0,
$
as needed.
\end{proof}

Next, set $u_n:=u^{\varepsilon_n}$, $u_m:=u^{\varepsilon_m}$ and $g(t):=u_n(t) - u_m(t), \quad t\geq 0.$

\begin{claim}
\label{claim2}
$\|g(t)\|_2 \rightarrow 0$ uniformly in $t \in [0,1]$ as $n,m \rightarrow \infty$.
\end{claim}

\begin{proof}[Proof of Claim \ref{claim2}]
We subtract the equations for $u_n$ and $u_m$ and obtain
$$
\partial_t g + (-\Delta)^{\frac{\alpha}{2}}g - b_n \cdot \nabla g - (b_n-b_m) \cdot \nabla u_m=0,
$$
\begin{equation}
\label{id_g}
\partial_t \|g\|_2^2 + \|(-\Delta)^{\frac{\alpha}{4}}g\|_2^2 - \Real\langle b_n \cdot \nabla g, g\rangle - \Real \langle (b_n-b_m) \cdot \nabla u_m,g\rangle=0.
\end{equation}
Concerning the last two terms, we have:
$$
-\Real \langle b_n \cdot \nabla g, g\rangle = \frac{\kappa}{2}\langle (d|x|_\varepsilon^{-\alpha}-\alpha |x|_\varepsilon^{-\alpha-2}|x|^2 g,g\rangle \geq \kappa \frac{d-\alpha}{2} \langle |x|_\varepsilon^{-\alpha},|g|^2\rangle,
$$
\begin{align*}
|\langle (b_n-b_m) \cdot \nabla u_m,g\rangle| & \leq |\langle \mathbf{1}_{B(0,1)} (b_n-b_m) \cdot \nabla u_m,g\rangle| + |\langle \mathbf{1}^c_{B(0,1)} (b_n-b_m) \cdot \nabla u_m,g\rangle| \\
& \text{(we are using $\|g\|_\infty \leq 2\|f\|_\infty$, $\|g\|_2 \leq 2\|f\|_2$)} \\
& \leq \|\mathbf{1}_{B(0,1)} (b_n-b_m)\|_2 \|\nabla u_m\|_22\|f\|_\infty + \|\mathbf{1}^c_{B(0,1)} (b_n-b_m)\|_\infty \|\nabla u_m\|_2 2\|f\|_2 \\
& (\text{we are using Claim \ref{claim1}}) \\
& \leq \|\mathbf{1}_{B(0,1)} (b_n-b_m)\|_2 \|\nabla f\|_22\|f\|_\infty + \|\mathbf{1}^c_{B(0,1)} (b_n-b_m)\|_\infty \|\nabla f\|_2 2\|f\|_2 \\
& \rightarrow 0 \quad \text{ as $n,m \rightarrow \infty$}.
\end{align*}
Thus, integrating \eqref{id_g} in $t$ and using the last two observations, we end the proof of Claim \ref{claim2}.
\end{proof}

By Claim \ref{claim2}, $\{e^{-t\Lambda^{\varepsilon_n} }f\}_{n=1}^\infty$, $f \in C_c^\infty$ is a Cauchy sequence in $L^\infty([0,1],L^2)$. Set
\begin{equation}
\label{T_conv}
T_2^t f:=s\mbox{-}L^2\mbox{-}\lim_n e^{-t\Lambda^{\varepsilon_n}}f \text{ uniformly in }0 \leq t \leq 1.
\end{equation}
(Clearly, the limit does not depend on the choice of $\{\varepsilon_n\}\downarrow 0$.)
Since $e^{-t\Lambda^{\varepsilon_n}}$ are contractions in $L^2$, we have $\|T_2^tf\|_2 \leq \|f\|_2$, $t \in [0,1]$. Extending $T_2^t$ by continuity to $L^2$, we obtain that $T_2^t$ is strongly continuous. Furthermore,
$$
T_2^tf =\lim_n e^{-t\Lambda^{\varepsilon_n}}f \text{ in $L^2$  for all } f \in L^2, \quad 0 \leq t \leq 1.
$$
Finally, extending $T_2^t$ to all $t \geq 0$ using the reproduction property, we obtain a contraction $C_0$ semigroup $T_2^t=:e^{-t\Lambda}$, $t \geq 0$.

\medskip

Now, let $1 \leq r < \infty$. Since $e^{-t\Lambda^{\varepsilon}}$ is a contraction in $L^r$, we obtain, by construction \eqref{T_conv} of $e^{-t\Lambda}f$, $f \in C_c^\infty$, appealing e.g.\,to Fatou's Lemma, that 
$$
\|e^{-t\Lambda}f\|_r \leq \|f\|_r, \quad t \geq 0.
$$
Thus, extending $e^{-t\Lambda}$ by continuity to $L^r$, we can define contraction semigroups $T_r^t:=[e^{-t\Lambda}]_{L^r \rightarrow L^r}^{\clos}$, $t \geq 0$. The strong continuity of $T_r^t$ in $L^r$ is a consequence of strong continuity of $e^{-t\Lambda}$, contractivity of $T^t_r$ and Fatou's Lemma. Write $T_r^t=:e^{-t\Lambda_r}$. Clearly,
$$
e^{-t\Lambda_r} =s\mbox{-}L^r\mbox{-}\lim_n e^{-t\Lambda_r^{\varepsilon_n}}, \quad t \geq 0.
$$

The latter and Proposition \ref{prop_contr} complete the proof of Proposition \ref{constr_d0}.
 
\end{proof}

\subsection{Case $d = 3$} 

\label{sect_constr_d3}

The proof of the next proposition works in all dimensions $d \geq 3$.

\begin{proposition}
\label{constr_d}
For every $r \in [1,\infty[$ the limit
$$
s\mbox{-}L^r\mbox{-}\lim_{\varepsilon \downarrow 0} e^{-t\Lambda_r^{\varepsilon}}\quad (\text{loc.\,uniformly in $t \geq 0$})
$$ 
exists and determines a contraction $C_0$ semigroup on $L^r$, say, $e^{-t\Lambda_r}$.
There exists a constant $c_N \neq c_N(\varepsilon)$ such that
$$
\|e^{-t\Lambda_r}\|_{r \rightarrow q} \leq c_N t^{-\frac{d}{\alpha}(\frac{1}{r}-\frac{1}{q})}, \quad t>0,
$$
for all $1 \leq r \leq q \leq \infty$.
\end{proposition}

\begin{proof}[Proof of Proposition \ref{constr_d}]
Denote $u^\varepsilon(t):=e^{-t\Lambda_r^\varepsilon}f$, $f \in C_c^\infty.$
For brevity, write
$
u \equiv u^\varepsilon$ and $w:=\nabla u$.

\begin{claim}
\label{claim3} For every $r \in ]1,\infty[$,
\begin{align*}
\frac{1}{r}\|w(t_1)\|_r^r & + \frac{4}{rr'}\int_0^{t_1} \sum_{i=1}^d\|(-\Delta)^{\frac{\alpha}{4}}(w_i|w|^{\frac{r-2}{2}})\|_2^2 dt \\
&+ \kappa\frac{d-\alpha-r}{r}\int_0^{t_1} \langle |x|_\varepsilon^{-\alpha}|w|^r\rangle dt +\alpha \kappa  \int_0^{t_1} \langle |x|_\varepsilon^{\alpha-2}|x \cdot w|^2|w|^{r-2}\rangle dt \leq \frac{1}{r}\|\nabla f\|_r^r, \quad t_1 > 0.
\end{align*}

In particular, for $1<r<d-\alpha$, 
$$
\|w(t_1)\|_r^r + \frac{4}{r'}c_S d^{-\frac{\alpha}{d}}\int_0^{t_1} \|w\|_{rj}^{r}dt \leq \|\nabla f\|^r_r, \quad t_1 > 0, \quad j:=\frac{d}{d-\alpha}.
$$
\end{claim}
\begin{proof}[Proof of Claim \ref{claim3}]
Set $w_i:=\nabla_i u$. We differentiate $\partial_t u + \Lambda_r^\varepsilon u=0$ in $x_i$, obtaining identity
$$
\partial_t w_i + (-\Delta)^{\frac{\alpha}{2}}w_i - b_\varepsilon \cdot \nabla w_i - (\nabla_i b_\varepsilon) \cdot w=0,
$$
which we multiply by $\bar{w}_i|w|^{r-2}$, integrate over the spatial variables and then sum in $1 \leq i \leq d$ to obtain
$$
\frac{1}{r} \partial_t \|w\|_r^r + \Real\langle (-\Delta)^{\frac{\alpha}{2}}w,w|w|^{r-2}\rangle - \Real\sum_{i=1}^d\langle  b_\varepsilon \cdot \nabla w_i,w_i|w|^{r-2} \rangle - \Real\sum_{i=1}^d \langle (\nabla_i b_\varepsilon) \cdot w,w_i|w|^{r-2} \rangle =0.
$$
By Theorem \ref{thm_M} (Appendix A),
$$
\Real\langle (-\Delta)^{\frac{\alpha}{2}}w,w|w|^{r-2}\rangle \geq \frac{4}{rr'}\langle(-\Delta)^{\frac{\alpha}{4}}(w|w|^{\frac{r-2}{2}}),(-\Delta)^{\frac{\alpha}{4}}(w|w|^{\frac{r-2}{2}})\rangle \equiv \frac{4}{rr'}\sum_{i=1}^d\|(-\Delta)^{\frac{\alpha}{4}}(w_i|w|^{\frac{r-2}{2}})\|_2^2.
$$
Next, integrating by parts, we obtain
$$
-\Real\sum_{i=1}^d\langle  b_\varepsilon \cdot \nabla w_i,w_i|w|^{r-2} \rangle = \frac{\kappa}{r} \langle (d|x|^{-\alpha}_\varepsilon - \alpha |x|_\varepsilon^{-\alpha-2}|x|^2)|w|^r\rangle \geq \kappa\frac{d-\alpha}{r} \langle |x|_\varepsilon^{-\alpha}|w|^r\rangle,
$$
and
$$
\Real\sum_{i=1}^d \langle (\nabla_i b_\varepsilon) \cdot w,w_i|w|^{r-2} \rangle = \kappa \langle |x|_\varepsilon^{-\alpha}|w|^r\rangle - \alpha \kappa \langle  |x|_\varepsilon^{-\alpha-2}(x \cdot w)^2|w|^{r-2}\rangle.
$$
The first required inequality follows.

Now, let $1<r<d-\alpha$. Note that
\begin{align*}
&\sum_{i=1}^d\|(-\Delta)^{\frac{\alpha}{4}}(w_i|w|^{\frac{r-2}{2}})\|_2^2  \geq c_S \sum_{i=1}^d \|w_i|w|^{\frac{r-2}{2}}\|_{2j}^2=c_S\sum_{i=1}^d \langle |w_i|^{2j}|w|^{(r-2)j}\rangle^{\frac{1}{j}} \\
& \geq c_S\biggl( \langle |w|^{(r-2)j} \sum_{i=1}^d|w_i|^{2j}\rangle \biggr)^{\frac{1}{j}} 
\\
& \bigg(\text{we use $\big(\sum_{i=1}^d |w|^{2j}\big)^{1/j} \geq \big(\sum_{i=1}^d |w_i|^2\big) d^{-1/j'}=|w|^2  d^{-1/j'}$}\bigg) \\
& \geq c_S d^{-1/j'} \langle |w|^{rj}\rangle^{\frac{1}{j}}=c_S d^{-\frac{\alpha}{d}}\|w\|_{rj}^r.  
\end{align*}
The second required inequality follows.
\end{proof}

Next, set $u_n:=u^{\varepsilon_n}$, $u_m:=u^{\varepsilon_m}$. Let $g(t):=u_n(t) - u_m(t)$, $t \geq 0$. 

\begin{claim}
\label{claim4}
$\|g(t)\|_2 \rightarrow 0$ uniformly in $t \in [0,1]$ as $n,m \rightarrow \infty$.
\end{claim}

\begin{proof}[Proof of Claim \ref{claim4}]
We subtract the equations for $u_n$ and $u_m$:
$$
\partial_t g + (-\Delta)^{\frac{\alpha}{2}}g - b_n \cdot \nabla g - (b_n-b_m) \cdot \nabla u_m=0.
$$
Multiplying the latter by $\bar{g}$ and integrating, we obtain 
\begin{equation*}
\|g(t_1)\|_2^2 + \int_0^{t_1} \|(-\Delta)^{\frac{\alpha}{4}}g\|_2^2 dt - \Real\int_0^{t_1} \langle b_n \cdot \nabla g, g\rangle dt - \Real\int_0^{t_1} \langle (b_n-b_m) \cdot \nabla u_m,g\rangle dt =0
\end{equation*}
for every $t_1 > 0$.
Since
$$
-\Real\langle b_n \cdot \nabla g, g\rangle = \frac{\kappa}{2}\langle (d|x|_\varepsilon^{-\alpha}-\alpha |x|_\varepsilon^{-\alpha-2}|x|^2 g,g\rangle \geq \kappa \frac{d-\alpha}{2} \langle |x|_\varepsilon^{-\alpha},|g|^2\rangle,
$$
we have 
\begin{equation}
\label{id_g2}
\|g(t_1)\|_2^2 + \int_0^{t_1} \|(-\Delta)^{\frac{\alpha}{4}}g\|_2^2 dt + \kappa \frac{d-\alpha}{2} \int_0^{t_1}\langle |x|^{-\alpha},|g|^2\rangle dt\leq \big|\int_0^{t_1} \langle (b_n-b_m) \cdot \nabla u_m,g\rangle dt \big|.
\end{equation}
Let us estimate the RHS of (10). Fix $1<r<d-\alpha$ (as in the second assertion of Claim \ref{claim3}). Then
\begin{align*}
|\langle (b_n-b_m) \cdot \nabla u_m,g\rangle| & \leq |\langle \mathbf{1}_{B(0,1)} (b_n-b_m) \cdot \nabla u_m,g\rangle| + |\langle \mathbf{1}_{B^c(0,1)} (b_n-b_m) \cdot \nabla u_m,g\rangle| \\
& \text{(we apply estimates $\|g\|_\infty \leq 2\|f\|_\infty$, $\|g\|_{(rj)'} \leq 2\|f\|_{(rj)'}$)} \\
& \leq \|\mathbf{1}_{B(0,1)} (b_n-b_m)\|_{(rj)'} \|\nabla u_m\|_{rj}2\|f\|_\infty + \|\mathbf{1}_{B^c(0,1)} (b_n-b_m)\|_\infty \|\nabla u_m\|_{rj} 2\|f\|_{(rj)'}.
\end{align*}
Clearly $\|\mathbf{1}_{B^c(0,1)} (b_n-b_m)\|_\infty\rightarrow 0$ as $n,m\rightarrow\infty$. The same is true for $\|\mathbf{1}_{B(0,1)} (b_n-b_m)\|_{(rj)'}$ since $(rj)'=\frac{rd}{rd-d+\alpha}<\frac{d}{\alpha-1}$.
Thus, in view of Claim \ref{claim3},
\begin{align*}
& \int_0^{t_1} |\langle (b_n-b_m) \cdot \nabla u_m,g\rangle| dt \\
& \leq \bigg(\|\mathbf{1}_{B(0,1)} (b_n-b_m)\|_{(rj)'}\|f\|_\infty  + \|\mathbf{1}_{B^c(0,1)} (b_n-b_m)\|_\infty \|f\|_{(rj)'}\bigg)2\int_0^{t_1} \|\nabla u_m\|_{rj}dt \rightarrow 0
\end{align*}
as $n,m \rightarrow \infty$.
\end{proof}

Now, we argue as in the proof of Proposition \ref{constr_d0} to obtain that for every $r \in [1,\infty[$ the limit
$
s\mbox{-}L^r\mbox{-}\lim_n e^{-t\Lambda_r^{\varepsilon_n}}$, $t \geq 0
$ 
exists and determines a contraction $C_0$ semigroup on $L^r$. It is easily seen that the limit does not depend on the choice of $\varepsilon_n$. 

The last assertion follows now from Proposition \ref{prop_contr}.

The proof of Proposition \ref{constr_d} is completed.
\end{proof}

\bigskip

\section{Construction of the semigroup $e^{-t\Lambda^*_r}$, $\Lambda^*_r = (-\Delta)^{\frac{\alpha}{2}} + \nabla \cdot b$ in $L^r$, $1 \leq r<\infty$} 

\label{sect_d2}

Set $(\Lambda^\varepsilon)_r^*:=(-\Delta)^{\frac{\alpha}{2}} +\nabla \cdot b_\varepsilon$, $D\big((\Lambda^\varepsilon)_r^*\big)=\mathcal W^{\alpha,r}$. By the Hille Perturbation Theorem, $-(\Lambda^\varepsilon)_r^*$ is the generator of a holomorphic $C_0$ semigroup in $L^r$ (arguing as in Section \ref{sect_d}; the argument there also shows that  $(\Lambda^\varepsilon)^*:=(-\Delta)^{\frac{\alpha}{2}} +\nabla \cdot b_\varepsilon$, $D\big((\Lambda^\varepsilon)^*\big)=D((-\Delta)^{\frac{\alpha}{2}}_{C_u})$ is the generator of a holomorphic semigroup in $C_u$).

\begin{proposition}
\label{prop_contr2}
For every $r \in [1,\infty[$ and $\varepsilon>0$, $e^{-t(\Lambda^{\varepsilon})_r^*}$ is a contraction $C_0$ semigroup. There exists a constant $c_N \neq c_N(\varepsilon)$ such that
$$
\|e^{-t(\Lambda^{\varepsilon})_r^*}\|_{r \rightarrow q} \leq c_N t^{-\frac{d}{\alpha}(\frac{1}{r}-\frac{1}{q})}, \quad t>0,
$$
for all $1 \leq r \leq q \leq \infty$.
\end{proposition}

\begin{proof} 
The semigroup $e^{-t(\Lambda^{\varepsilon})_r^*}$ is constructed in $L^r$ repeating the argument in Section \ref{sect_d}. The ultra contractivity estimate for $1 < r \leq q < \infty$ follows from Proposition \ref{prop_contr} by duality, and for all $1 \leq r \leq q \leq \infty$ upon taking limits $r \downarrow 1$, $q \uparrow \infty$.
\end{proof}

\begin{proposition}
\label{constr_d2}
For every $r \in [1,\infty[$ the limit
$$
s\mbox{-}L^r\mbox{-}\lim_{\varepsilon \downarrow 0} e^{-t(\Lambda^{\varepsilon})_r^*}  \quad (\text{loc.\,uniformly in $t \geq 0$})
$$ 
exists and determines a contraction $C_0$ semigroup in $L^r$, say, $e^{-t\Lambda^*_r}$. 
There exists a constant $c_N$ such that
$$
\|e^{-t\Lambda_r^*}\|_{r \rightarrow q} \leq c_N t^{-\frac{d}{\alpha}(\frac{1}{r}-\frac{1}{q})}, \quad t>0,
$$
for all $1 \leq r \leq q \leq \infty$.

We have for $1<r<\infty$
$$
\langle e^{-t\Lambda_{r'}(b)}f,g\rangle=\langle f,e^{-t\Lambda^*_{r}(b)}g\rangle, \quad t>0, \quad f \in L^{r'}, \quad r'=\frac{r}{r-1}, \quad g \in L^{r}.
$$
\end{proposition}

\begin{proof}
First, let $r=2$.
In view of Proposition \ref{prop_contr2}, we can argue as in the proof of \cite[Prop.\,10]{KSS}, appealing to the Rellich-Kondrashov Theorem, to obtain: For every sequence $\varepsilon_n \downarrow 0$ there exists a subsequence $\varepsilon_{n_m}$ such that the limit
\begin{equation}
\label{sg_conv}
s\mbox{-}L^{2}\mbox{-}\lim_{m} e^{-t(\Lambda^{\varepsilon_{n_m}})^*} \quad (\text{loc.\,uniformly in $t \geq 0$})
\end{equation}
exists and determines a $C_0$ semigroup in $L^2$. 

On the other hand,
since 
$$
\langle e^{-t\Lambda^\varepsilon}f,g\rangle=\langle f,e^{-t(\Lambda^\varepsilon)^*}g\rangle, \quad t>0, \quad f, g \in L^2,
$$
it follows from Proposition \ref{constr_d} that for every $g \in L^2$ $e^{-t(\Lambda^\varepsilon)^*}g$ converge weakly in $L^2$ as $\varepsilon \downarrow 0$. Thus, the limit in \eqref{sg_conv} does not depend on the choice of $\varepsilon_{n_m}$ and $\varepsilon_n$.

For $1 \leq r < \infty$, we repeat the argument in the end of the proof of Proposition \ref{constr_d0}, appealing to Proposition \ref{prop_contr2}.

The last assertion follows from the analogous property of $e^{-t\Lambda_{r'}^\varepsilon}$, $e^{-t(\Lambda^\varepsilon)_r^*}$, $\varepsilon>0$ and Propositions \ref{constr_d}, \ref{constr_d2}.
\end{proof}

\appendix

\section{$L^r$ (vector) inequalities for symmetric Markov generators}

\label{appendix_A}

Let $X$ be a set and $\mu$ a $\sigma$-finite measure on $X$. Let $T^t=e^{-t A}$, $t \geq 0$, be a symmetric Markov semigroup in $L^2(X, \mu)$. 
Let $$T_r^t:=\big[ T^t \upharpoonright L^2 \cap L^r \bigr]_{L^r \rightarrow L^r}, \quad t \geq 0,$$
a contraction $C_0$ semigroup on $L^r$, $r \in [1,\infty[$.
Put $T^t_r=:e^{-tA_r}$.

\begin{theorem} 
\label{thm_M}
Let $f_i \in D(A_r)$ {\rm($1 \leq i \leq m$)}, $r \in ]1,\infty[$. Set $f:=(f_i)_{i=1}^m$, $f_{(r)}:=f|f|^{\frac{r-2}{2}}$. Then $f_i|f|^{\frac{r-2}{2}} \in D(A^{\frac{1}{2}})$ {\rm($1 \leq i \leq m$)} and, applying the operators coordinate-wise, we have
\begin{equation}
\tag{$i$}
\frac{4}{rr'}\langle A^{\frac{1}{2}}f_{(r)},A^{\frac{1}{2}}f_{(r)}\rangle \leq \Real\langle A_r f, f|f|^{r-2}\rangle \leq \varkappa(r)\langle A^{\frac{1}{2}}f_{(r)},A^{\frac{1}{2}}f_{(r)}\rangle,
\end{equation}
where $\varkappa(r):=\sup_{s \in ]0,1[}\big[(1+s^{\frac{1}{r}})(1+s^{\frac{1}{r'}})(1+s^{\frac{1}{2}})^{-2}\big]$, $r'=\frac{r}{r-1}$,
\begin{equation}
\tag{$ii$}
\big| \Imag \langle A_r f, f|f|^{r-2}  \rangle \big| \leq \frac{|r-2|}{2\sqrt{r-1}}
\,\Real \langle A_r f, f|f|^{r-2} \rangle,
\end{equation}
where
$$
\langle A^{\frac{1}{2}}f_{(r)},A^{\frac{1}{2}}f_{(r)}\rangle=\sum_{i=1}^m \|A^{\frac{1}{2}}(f_i|f|^{\frac{r-2}{2}})\|_2^2, \qquad \langle A_r f, f|f|^{r-2}\rangle = \sum_{i=1}^m \langle A_r f_i, f_i|f|^{r-2}\rangle.
$$
\end{theorem}

Theorem \ref{thm_M} is a prompt but useful modification of  \cite[Theorem 2.1]{LS} (corresponding to the case $m=1$): it allows us to control higher-order derivatives of $u(t)=e^{-t\Lambda}f$, $\Lambda \supset (-\Delta)^{\frac{\alpha}{2}} - b \cdot \nabla$, $f \in C_c^\infty$ in the proof of Proposition \ref{constr_d} (see Claim \ref{claim3} there).

For the sake of completeness, we included the detailed proof below.

\medskip

\textbf{1.~}We will need

\begin{claim}
\label{claim_M}
There exists a finitely additive measure $\mu_t$ on $X \times X$,
symmetric in the sense that $\mu_t (A \times B) = \mu_t (B \times A)$ on any $\mu$-measurable sets of finite measure $A$ and $B$, and satisfying
\[
	\langle T^t f, g \rangle = \int_{X \times X} f(x) \overline{g(x)} d\mu_t (x, y)
	\;\;\; (f, g \in L^1 \cap L^\infty).
\]
\end{claim}

In order to justify the claim, let us introduce the Banach space $\mathcal{L}^\infty = \mathcal{L}^\infty (X, \mathcal{M}_\mu)$, the Banach space of all bounded $\mu$-measurable functions, endowed with the norm $\|\!|f|\!\| := \sup \{|f(x)| \mid x \in X\}$.

Let $N^\infty \equiv \mathcal{N}^\infty(X, \mathcal{M}_\mu)$ be the set of all $\mu$-negligible functions, so that $L^\infty = \mathcal{L}^\infty / \mathcal{N}^\infty$. Denoting by $\pi : f \to \widetilde{f}$ the canonical mapping of $\mathcal{L}^\infty$ onto $L^\infty$, we can identify $L^\infty$ with $\pi(\mathcal{L}^\infty)$. Since $\mu$ is $\sigma$-finite, there exists a lifting $\rho : L^\infty \to \mathcal{L}^\infty$, a linear multiplicative positivity preserving map such that
\[
	\rho (\mathbf{1}_G) = \mathbf{1}_G \mbox{ for all } G \in \mathcal{M}_\mu \mbox{ with } \mu(G) < \infty.
\]
Given $t > 0$ define $T_\rho^t : \mathcal{L}^\infty \to \mathcal{L}^\infty$ by
\[
	T_\rho^t f := \rho (T^t_\infty f),
\]
and so $T^t_\rho$ is a positivity preserving semigroup, and
\[
	\langle T^t_\rho f, g \rangle = \langle T^t \widetilde{f}, \widetilde{g} \rangle
	\;\;\; (\widetilde{f}, \widetilde{g} \in L^\infty \cap L^1).
\]
The following set function is associated with the semigroup $T^t_\infty$:
\[
	P(t, x, G) := (T^t_\rho \mathbf{1}_G)(x) \;\;\; (t > 0, x \in X, G \in \mathcal{M}_\mu).
\]
This function satisfies the following evident properties:
\begin{enumerate}
\item $P(t, x, G)$ ($G \in \mathcal{M}_\mu$) is finitely additive.
\item $P(t, x, X) \leq 1$.
\item $\int f(y) P(t, \cdot, dy)$ exists and equals to $T^t_\rho f(\cdot)$ ($f \in \mathcal{L}^\infty$).
\end{enumerate}

Set by definition
\[
	\mu_t (A \times B) = \int_A P(t, x, B) d\mu(x) \;\;\; (A, B \in \mathcal{M}_\mu).
\]

The claimed symmetry of $\mu_t$ is a direct consequence of the self-adjointness of $T^t$ and the fact that we can identify $T^t_\infty \mathbf{1}_G$ and $T^t \mathbf{1}_G$ 
for every $G \in \mathcal{M}_\mu$ of finite measure.

\medskip

\textbf{2.~}We are in position to complete the proof of Theorem \ref{thm_M}.

\begin{proof}[Proof of Theorem \ref{thm_M}]
We will need the following elementary estimates: for all $s,t \in [0,\infty[$, $r \in [1,\infty[$,
\begin{align*}
\label{star_est}
	& \frac{4}{r r^\prime} (s^{r} + t^{r} - 2b(st)^{\frac{r}{2}}) \\
		& \leq s^r + t^r - b(st^{r-1} + ts^{r-1}) \\
		&\leq \varkappa(r) (s^r + t^r - 2b(st)^{\frac{r}{2}}), \qquad b \in [-1,1]
\tag{$\ast$}
\end{align*}
(Lemma \ref{lem:helperest}($l_3$), ($l_5$) below)
\begin{align*}
\label{star_est2}
	|a| |st^{r-1} - ts^{r-1}| \leq \frac{|r-2|}{2\sqrt{r-1}} \big[s^r + t^r - \sqrt{1-a^2}(st^{r-1} + ts^{r-1})\big], \qquad a \in [-1,1]
\tag{$\ast\ast$}
\end{align*}
(Lemma \ref{lem:helperest}($l_4$) below).

We are going to establish the following inequalities: for all $f \in L^r$
\begin{align}
\label{ineq_T}
	\frac{4}{r r^\prime} \langle (1 - T_2^t) f_{(r)}, f_{(r)} \rangle 
	\leq  \Real\langle (1 - T_r^t) f, f|f|^{r-2}  \rangle 
	\leq \varkappa(r) \langle (1 - T_2^t) f_{(r)}, f_{(r)} \rangle,
\end{align}
\begin{equation}
\label{ineq_Im}
	\big| \Imag \langle (1 - T^t_r) f, f|f|^{r-2} \rangle \big|
		\leq \frac{|r-2|}{2\sqrt{r-1}} \Real \langle (1 - T^t_r) f, f|f|^{r-2} \rangle.
\end{equation}
The the required estimates would follow from the definitions of $A_r$ and $A^\frac{1}{2}$. Indeed, for $f \in D(A_r),$
\[
s\mbox{-}L^p\mbox{-}\lim_{t \downarrow 0} \frac{1}{t} (1 - T^t_r) f \text{ exists and equals to } A_r f.
\]
Combining the LHS of \eqref{ineq_T} and Fatou's Lemma, it is seen that $\mathcal{J} := \lim_{t \downarrow 0} \frac{1}{t} \langle (1 - T^t) f_{(r)}, f_{(r)} \rangle$ exists and is finite. By the spectral theorem for self-adjoint operators, the latter means that $f_{(r)} \in D(A^\frac{1}{2})$ and $\mathcal{J} = \| A^\frac{1}{2} f_{(r)} \|_2^2$.

First, let $f \in L^1 \cap L^\infty$ with $\sprt f \subset G, \; G \in \mathcal{M}_\mu$, $\mu(G) < \infty$.
Using Claim \ref{claim_M}, we have
\begin{align*}
\langle T^t f, f|f|^{r-2}  \rangle & = \frac{1}{2} \langle T^t f, f|f|^{r-2} \rangle + \frac{1}{2} \langle f, T^t (f|f|^{r-2}) \rangle \\
& = \frac{1}{2} \int [f(x) \cdot \bar{f}(y) |f(y)|^{r-2} + f(y)\cdot \bar{f}(x) |f(x)|^{r-2}] d\mu_t (x, y), \\
& \\
\langle T^t f_{(r)}, f_{(r)} \rangle &= \frac{1}{2}\int f_{(r)}(x) \cdot \bar{f}_{(r)}(y) d \mu_t(x, y) + \frac{1}{2}\int \bar{f}_{(r)}(x) \cdot f_{(r)}(y) d \mu_t(x, y), \\ & \\
\langle T^t \mathbf{1}_G, |f|^r \rangle&  = \langle \mathbf{1}_G, T^t |f|^r \rangle \\
& = \frac{1}{2} \langle P(t, \cdot, G) |f(\cdot)|^r \rangle + \frac{1}{2} \langle \mathbf{1}_G (\cdot) \int |f(y)|^r P(t, \cdot, dy) \rangle \\
& = \frac{1}{2} \int [|f(x)|^r + |f(y)|^r] d \mu_t (x, y), \\ & \\
\| f \|_r^r &  = \langle T^t \mathbf{1}_G, |f|^r \rangle + \langle (1 - T^t \mathbf{1}_G), |f|^r \rangle.
\end{align*}
Setting $s:=|f(x)|$, $l:=|f(y)|$, $\beta:=\frac{f(x) \cdot \bar{f}(y)}{|f(x)||f(y)|}$, $b:=\Real \beta$, $a:=\Imag \beta$, we obtain
$$
	\langle (1 - T^t) f, f|f|^{r-2}  \rangle = \langle (1 - T^t \mathbf{1}_G), |f|^r \rangle
		+ \frac{1}{2} \int [s^r + l^r - \beta sl^{r-1} - \bar{\beta}ls^{r-1})] d\mu_t,
$$
$$
\Real	\langle (1 - T^t) f, f|f|^{r-2}  \rangle = \langle (1 - T^t \mathbf{1}_G), |f|^r \rangle
		+ \frac{1}{2} \int [s^r + l^r - b(sl^{r-1} + ls^{r-1})] d\mu_t,
$$
$$
	\langle (1 - T^t) f_{(r)}, f_{(r)} \rangle = \langle (1 - T^t \mathbf{1}_G), |f|^r \rangle
		+ \frac{1}{2} \int [s^r + l^r - 2 b (st)^\frac{r}{2}] d\mu_t,
$$
$$
	\Imag \langle (1 - T^t) f, f|f|^{r-2} \rangle =
		\frac{1}{2} \int a(sl^{r-1} - ls^{r-1}) d \mu_t.
$$

Next, employing \eqref{star_est}, \eqref{star_est2},
we obtain \eqref{ineq_T}, \eqref{ineq_Im} but for $f \in L^1 \cap L^\infty$ with $\sprt f \in G$, $\mu(G) < \infty$.

To end the proof, we note that $\mu$ is a $\sigma$-finite measure, and so we can first get rid of the condition ``$\sprt f \in G$, $\mu(G) < \infty$'', and then, using the truncated functions
\[
	g_n = \left\{ \begin{array}{ll} g, & \text{ if } |g| \leq n, \\
						0, & \text{ if } |g| > n, \end{array} \right. \;\;\; n = 1, 2, \ldots
\]
and the Dominated Convergence Theorem, to get rid of ``$f \in L^1 \cap L^\infty$''.
\end{proof}

For the sake of completeness, we also include the following result concerning the scalar case.

\begin{theorem}
If $0 \leq f \in D(A_r)$, then 
\begin{equation}
	\frac{4}{r r^\prime} \| A^\frac{1}{2} f^\frac{r}{2} \|^2_2 \leq
		\langle A_r f, f^{r-1} \rangle \leq \| A^\frac{1}{2} f^\frac{r}{2} \|^2_2;
		\tag{$iii$}
\end{equation}
Moreover, if $r \in [2, \infty[$ and $f \in D(A) \cap L^\infty$, then $f_{(r)}:=|f|^{\frac{r}{2}}\sgn f \in D(A^\frac{1}{2})$ and
\begin{equation}
	\frac{4}{r r^\prime} \| A^\frac{1}{2} f_{(r)} \|^2_2 \leq
		\Real \langle A f, f^{r-1}\sgn f \rangle \leq \varkappa(r) \| A^\frac{1}{2} f_{(r)} \|^2_2, \qquad \sgn f:=\frac{f}{|f|}
		\tag{$i'$}
\end{equation}
If $r \in [2, \infty[$ and $0 \leq f \in D(A) \cap L^\infty$, then $f^\frac{r}{2} \in D(A^\frac{1}{2})$ and
\begin{equation}
	\frac{4}{r r^\prime} \| A^\frac{1}{2} f^\frac{r}{2} \|^2_2 \leq
		\langle A f, f^{r-1} \rangle \leq \| A^\frac{1}{2} f^\frac{r}{2} \|^2_2.
		\tag{$iii'$}
\end{equation}
\end{theorem}
\begin{proof}
Follows closely the proof of Theorem \ref{thm_M} where, instead of inequalities \eqref{ineq_T}, \eqref{ineq_Im}, we use
$$
	\frac{4}{r r^\prime} \langle (1 - T^t) f^{\frac{r}{2}}, f^{\frac{r}{2}} \rangle
		\leq \langle (1 - T^t) f, f^{r-1} \rangle
		\leq \langle (1 - T^t) f^{\frac{r}{2}}, f^{\frac{r}{2}} \rangle \;\;\; (f \in L^r_+).
$$
\end{proof}

In the proof of Theorem \ref{thm_M} we use

\begin{lemma}
\label{lem:helperest}
Let $s, t \in [0, \infty[$, $r \in [1, \infty[$ and $b \in [-1, 1]$. Then
\[
\label{eqn:lem1}
	\frac{4}{r r^\prime} (s^{\frac{r}{2}} - t^{\frac{r}{2}})^2 \leq
		(s-t)(s^{r-1} - t^{r-1}) \leq (s^{\frac{r}{2}} - t^{\frac{r}{2}})^2.
\tag{\emph{l}$_1$}
\]
\[
\label{eqn:lem2}
	(s^{\frac{r}{2}} + t^{\frac{r}{2}})^2 \leq
		(s+t)(s^{r-1} + t^{r-1}) \leq \varkappa(r) (s^{\frac{r}{2}} + t^{\frac{r}{2}})^2
\tag{\emph{l}$_2$}
\]
\[
\label{eqn:lem3}
	\frac{4}{r r^\prime} (s^{\frac{r}{2}} + t^{\frac{r}{2}} + 2b(st)^{\frac{r}{2}})
		\leq s^r + t^r + b(st^{r-1} + ts^{r-1}).
\tag{\emph{l}$_3$}
\]
\[
\label{eqn:lem4}
	|b| |st^{r-1} - ts^{r-1}| \leq \frac{|r-2|}{2\sqrt{r-1}} \big[s^r + t^r - \sqrt{1-b^2}(st^{r-1} + ts^{r-1})\big].
\tag{\emph{l}$_4$}
\]
\[
\label{eqn:lem5}
	s^r + t^r + b(st^{r-1} + ts^{r-1}) \leq \varkappa(r) (s^r + t^r + 2b(st)^{\frac{r}{2}}).
\tag{\emph{l}$_5$}
\]
\end{lemma}

\begin{proof}
The RHS of \eqref{eqn:lem1} and the LHS of  \eqref{eqn:lem2} are consequences of the inequality $2|\alpha||\beta| \leq \alpha^2 + \beta^2$.

The RHS of \eqref{eqn:lem2} follows from the definition of $\varkappa(r)$.

The LHS of \eqref{eqn:lem1} follows from
\[
	\frac{4}{r^2} (s^{\frac{r}{2}} - t^{\frac{r}{2}})^2 = (\int_t^s z^{\frac{r}{2} - 1} dz)^2
		\leq \int_t^s dz \cdot \int_t^s z^{r-2} dz.
\]

\eqref{eqn:lem3} is a consequence of the LHS of \eqref{eqn:lem1}.

To derive \eqref{eqn:lem4} set
\[
	A = st^{r-1} - ts^{r-1}, B = \frac{|r-2|}{2\sqrt{r-1}} (st^{r-1} + ts^{r-1}),
		C = \frac{|r-2|}{2\sqrt{r-1}} (s^r + t^r),
\]
and note that $A^2 + B^2 \leq C^2 \Rightarrow |A \sin \theta| + |B \cos \theta| \leq C$.

The inequality $A^2 + B^2 \leq C^2$ follows from
\[
\label{eqn:lem4supp}
	(st^{r-1} - ts^{r-1})^2 \leq \left( \frac{r-2}{r} \right)^2 (s^r - t^r)^2
\tag{$\star$}
\]
and the LHS of \eqref{eqn:lem1} and \eqref{eqn:lem2}.

 Setting $v = s/t$, \eqref{eqn:lem4supp} takes the form
\[
	| v^{r-1} - v | \leq \frac{|r-2|}{r} |v^r - 1|.
\]
All possible cases are reduced to the case where $v > 1$ and $r > 2$.

If $\frac{r-2}{r} v \geq 1$, then the inequality $v^{r-1} - v \leq \frac{r-2}{r} v^r - \frac{r-2}{r}$ is selfevident. If $1 < v < \frac{r}{r-2}$, we set $\psi(v) = \frac{r-2}{r} v^r - v^{r-1} + v - \frac{r-2}{r}$ and note that $\frac{d}{dv} \psi(v) \geq 0$ by Young's inequality.

Finally, \eqref{eqn:lem5} follows from the RHS of \eqref{eqn:lem2} and the following elementary inequality:
\[
	\frac{A + bB}{A + bC} \leq \frac{A + B}{A + C} \quad (b \in [-1, 1]), \text{ provided that } A > C \text{ and } B \geq C > 0.
\]
\end{proof}

\section{Extrapolation Theorem}

\label{appendix_B}

\begin{theorem}[{T.\,Coulhon-Y.\,Raynaud.} {\cite[Prop.\,II.2.1, Prop.\,II.2.2]{VSC}}.]
\label{thm_cr}
Let $U^{t,s}: L^1 \cap L^\infty \rightarrow L^1 + L^\infty$ be a two-parameter evolution family of operators:
\[U^{t,s} = U^{t,\tau}U^{\tau,s}, \quad 0 \leq s < \tau < t \leq \infty.
\]
Suppose that, for some $1 \leq p < q < r \leq \infty,$ $\nu > 0,$ $M_1$ and $M_2,$ the inequalities
\[
\| U^{t,s} f \|_p \leq M_1 \| f \|_p \quad \text{ and } \quad \| U^{t,s} f \|_r \leq M_2 (t-s)^{-\nu} \|  f \|_q
\]
are valid for all $(t,s)$ and $f \in L^1 \cap L^\infty.$ Then
\[
\| U^{t,s} f \|_r \leq M (t-s)^{-\nu/(1-\beta)} \| f \|_p ,
\]
where $\beta = \frac{r}{q}\frac{q-p}{r-p}$ and $M = 2^{\nu/(1-\beta)^2} M_1 M_2^{1/(1-\beta)}.$
\end{theorem}

\begin{proof} Set $2 t_s=t+s.$ The hypotheses and H\"older's inequality imply
\begin{align*}
\| U^{t, s} f \|_r & \leq M_2 (t-t_s)^{-\nu} \| U^{t_s,s} f \|_q \\
& \leq M_2 (t-t_s)^{-\nu} \| U^{t_s,s} f \|_r^\beta \;\| U^{t_s,s} f \|_p^{1-\beta} \\
& \leq M_2 M_1^{1-\beta} (t-t_s)^{-\nu} \| U^{t_s,s} f \|_r^\beta \;\| f \|_p^{1-\beta},
\end{align*}
and hence
\[
(t-s)^{\nu/(1-\beta)} \| U^{t,s} f \|_r/\| f \|_p \leq M_2 M_1^{1-\beta} 2^{\nu/(1-\beta)} \big [(t_s -s)^{\nu/(1-\beta)} \| U^{t_s,s} f \|_r\;/\| f \|_p \big ]^\beta.
\]
Setting $R_{2 T}: = \sup_{t-s \in ]0,T]} \big [ (t-s)^{\nu/(1-\beta)} \| U^{t,s} f \|_r/\| f \|_p \big ],$ we obtain from the last inequality that $R_{2 T} \leq M^{1-\beta} (R_T)^\beta.$ But $R_T \leq R_{2T}$, and so $R_{2T} \leq M.$
\end{proof}

\begin{corollary}
\label{cor_cr}
Let $U^{t,s}: L^1 \cap L^\infty \rightarrow L^1 + L^\infty$ be an evolution family of operators. Suppose that, for some $1 < p <q < r \leq \infty,$ $\nu > 0,$ $M_1$ and $M_2,$ the inequalities
\[
\| U^{t,s} f \|_r \leq M_1 \| f \|_r \quad \text{ and } \quad \| U^{t,s} f \|_q \leq M_2 (t-s)^{-\nu} \|  f \|_p
\]
are valid for all $(t,s)$ and $f \in L^1 \cap L^\infty.$ Then
\[
\| U^{t,s} f \|_r \leq M (t-s)^{-\nu/(1-\beta)} \| f \|_p ,
\]
where $\beta = \frac{r}{q} \frac{q-p}{r-p}$ and $M = 2^{\nu/(1-\beta)^2} M_1 M_2^{1/(1-\beta)}.$
\end{corollary}

\section{The range of an accretive operator}

\label{appC}

In the proof of Theorem \ref{nash_est} we use the following well known result.

Let $P$ be a closed operator on $L^1$ such that $\Real\langle(\lambda+ P)f,\frac{f}{|f|}\rangle\geq 0$ for all $f \in D(P)$,
and $R(\mu + P)$ is dense in $L^1$ for a $\mu>\lambda$. 

Then $R(\mu + P)=L^1$.

\smallskip

Indeed, let $y_n \in R(\mu + P)$, $n=1,2,\dots$, be a Cauchy sequence in $L^1$; $y_n=(\mu+P)x_n$, $x_n \in D(P)$. Write $[f,g] := \langle f,\frac{g}{|g|}\rangle$. Then
\begin{align*}
(\mu-\lambda)\|x_n-x_m\|_1&=(\mu-\lambda)[x_n-x_m,x_n-x_m] \\
&\leq (\mu-\lambda)[x_n-x_m,x_n-x_m] + [(\lambda+P)(x_n-x_m),x_n-x_m] \\
& = [(\mu+P)(x_n-x_m),x_n-x_m] \leq \|y_n-y_m\|_1.
\end{align*}
Thus, $\{x_n\}$ is itself a Cauchy sequence in $L^1$. Since $P$ is closed, the result follows.

\end{document}